\pdfoutput=1
\documentclass[paper=a4, english, final ]{scrartcl}

\usepackage{babel}

\usepackage[T1]{fontenc}
\usepackage[utf8]{inputenc}

\usepackage{lmodern}

\usepackage[final, babel ]{microtype}

\usepackage{amsfonts} \usepackage{amssymb}  \usepackage{mathtools} 

\providecommand\given{} \newcommand\SetSymbol[1][]{
   \mathrel{}\mathclose{}#1|\allowbreak\mathopen{}\mathrel{}}
\DeclarePairedDelimiterX\Set[1]{\lbrace}{\rbrace}{ \renewcommand\given{\SetSymbol[\delimsize]} #1 }

\DeclarePairedDelimiter\abs{\lvert}{\rvert}\DeclarePairedDelimiter\norm{\lVert}{\rVert}

\makeatletter
\let\oldabs\abs
\def\abs{\@ifstar{\oldabs}{\oldabs*}}
\let\oldnorm\norm
\def\norm{\@ifstar{\oldnorm}{\oldnorm*}}

\let\oldSet\Set
\def\Set{\@ifstar{\oldSet}{\oldSet*}}
\makeatother

 \mathtoolsset{showonlyrefs}

\usepackage{chngcntr}
\counterwithin{equation}{subsection}

\usepackage{amsthm} \usepackage{thmtools, thm-restate}

\declaretheorem[name=Proposition, numberlike=equation]{prop}

\declaretheorem[name=Definition, numberlike=equation, style=definition]{defn}
\declaretheorem[name=Lemma, numberlike=equation, style=definition]{lemma}
\declaretheorem[name=Corollary, numberlike=equation, style=definition]{cor}

\declaretheorem[name=Remark, numberlike=equation, style=remark]{rem}
\declaretheorem[name=Example, numberlike=equation, style=remark]{ex}

\usepackage{tikz-cd} \tikzset{
	pf/.style={commutative diagrams/.cd, every arrow, every label},
	surj/.style=commutative diagrams/two heads,
	inj/.style=commutative diagrams/hook,
	gl/.style=commutative diagrams/equal,
	mat/.style={matrix of math nodes, commutative diagrams/.cd, every cell},
	dr/.style={matrix of math nodes, commutative diagrams/.cd, every cell, column sep=small},
	seq/.style={matrix of math nodes, commutative diagrams/.cd, every cell, column sep=small}
	}
\newenvironment{diag*}{\[\begin{tikzpicture}[commutative diagrams/.cd, every diagram, baseline=(current bounding box.center)]}{\end{tikzpicture}\]\ignorespacesafterend}
\newenvironment{diag}{\begin{equation}\begin{tikzpicture}[commutative diagrams/.cd, every diagram, baseline=(current bounding box.center)]}{\end{tikzpicture}\end{equation}\ignorespacesafterend}

\usepackage{csquotes} \usepackage[style=authoryear, date=year,
	isbn=false,
]{biblatex}
\DeclareSourcemap{
	\maps[datatype=bibtex]{
\map[overwrite]{
			\step[fieldsource=author,
				match = {Yuri I. Manin|Yu. Manin},
				replace = {Yuri Manin}
			]
			\step[fieldsource=author,
				match = {Ieke Moerdijk|I. Moerdijk},
				replace = {Ieke Moerdijk}
			]
		}
	}
}

\author{\texorpdfstring{Enno Keßler \and Artan Sheshmani \and Shing-Tung Yau}{
Enno Keßler, Artan Sheshmani, Shing-Tung Yau
}
}
\title{Super quantum cohomology I:\\ Super stable maps of genus zero\\ with Neveu--Schwarz punctures}

\usepackage{stensor}
\DeclareRobustCommand{\tensor}{\stensor}

\usepackage{slashed}
\newcommand{\Dirac}{\slashed{D}}

\usepackage{faktor}
\usepackage{bm}

\usepackage{enumitem}
\setlist[enumerate]{label={\roman*)}}
\setlist[description]{labelindent=\parindent}

\usepackage[final, pdfusetitle ]{hyperref}

\DeclareMathOperator{\ACI}{I}

\DeclareMathOperator{\codim}{codim}

\DeclareMathOperator{\GL}{GL}

\DeclareMathOperator{\Hom}{Hom}
\DeclareMathOperator{\im}{im}
\DeclareMathOperator{\id}{id}

\DeclareMathOperator{\Mat}{Mat}

\DeclareMathOperator{\SpGL}{Sp}

\DeclareMathOperator{\Tr}{Tr}

\newcommand{\ic}{\mathrm{i}}

\newcommand{\dual}[1]{{#1}^{\vee}}

\newcommand{\cat}[1]{\mathsf{#1}}

\newcommand{\LieAlg}[1]{\mathfrak{#1}}

\newcommand{\Top}[1]{{\|#1\|}}
\newcommand{\Red}[1]{{#1}_{red}}
\newcommand{\Smooth}[1]{{|#1|}}

\renewcommand{\d}{\mathop{}\!d}

\newcommand{\p}[1]{{p(#1)}}

\newcommand{\VSec}[2][]{\Gamma_{#1}\left(#2\right)}

\newcommand{\tangent}[2][]{T_{#1}#2} \newcommand{\differential}[1]{\d{#1}} \newcommand{\cotangent}[1]{\dual{T}#1}

\newcommand{\Integers}{\mathbb{Z}}
\newcommand{\Z}{\Integers}

\newcommand{\RealNumbers}{\mathbb{R}}
\newcommand{\R}{\RealNumbers}
\newcommand{\ComplexNumbers}{\mathbb{C}}
\newcommand{\C}{\ComplexNumbers}

\newcommand{\ProjectiveSpace}[2][]{\mathbb{P}_{#1}^{#2}}

\newcommand{\cD}{\mathcal{D}}
\newcommand{\cO}{\mathcal{O}}

\newcommand{\targetACI}{J}
\newcommand{\DJBar}{\overline{D}_\targetACI}
\newcommand{\DelJBar}{\overline{\partial}_\targetACI}
\newcommand{\groupoid}[1]{\mathcal{#1}}
\DeclareMathOperator{\DLaplace}{\Delta^\cD}
\DeclareMathOperator{\ev}{ev}

\begin{document}
\maketitle
\begin{abstract}
	In this article we define stable supercurves and super stable maps of genus zero via labeled trees.
	We prove that the moduli space of stable supercurves and super stable maps of fixed tree type are quotient superorbifolds.
	To this end, we prove a slice theorem for the action of super Lie groups on Riemannian supermanifolds and discuss superorbifolds.
	Furthermore, we propose a Gromov topology on super stable maps such that the restriction to fixed tree type yields the quotient topology from the superorbifolds and the reduction is compact.
	This would, possibly, lead to the notions of super Gromov--Witten invariants and small super quantum cohomology to be discussed in sequels.
\end{abstract}

\tableofcontents

\section{Introduction}
In this article we work towards a compactification of the moduli space of super \(\targetACI\)-holomorphic curves of genus zero.
Super \(\targetACI\)-holomorphic curves are a supergeometric generalization of \(\targetACI\)-holomorphic curves that we introduced in~\cite{KSY-SJC}.

\(\targetACI\)-holomorphic curves or pseudoholomorphic curves have been of great interest to mathematics since the discovery that \(\targetACI\)-holomorphic curves allow to construct invariants of symplectic manifolds in~\cite{G-PHCSM} and that those invariants are related to topological superstring theory, see~\cite{W-TDGISTMC}.
Systematic development of the theory of \(\targetACI\)-holomorphic curves involved many authors, see, for example~\cites{CDGP-PCYMESST}{RT-MTQC}{BM-SSMGMI}, or the textbooks~\cites{HKKTPVVZ-MS}{McDS-JHCST}.
A crucial step towards Gromov--Witten invariants is the compactification of the moduli space of \(\targetACI\)-holomorphic curves via stable maps which was first proposed in~\cites{KM-GWCQCEG}.
This compactification takes into account that a family of \(\targetACI\)-holomorphic curves of genus zero may degenerate into a tree of bubbles in the limit.

Super Riemann surfaces on the other hand have first appeared as generalizations of Riemann surfaces with anti-commutative variables in superstring theory, see~\cites{F-NSTTDCFT}{dHP-GSP}.
The mathematical formalization of the anti-commutative variables required for supersymmetry is known as supergeometry.
An early overview is given in~\cite{L-ITS}.
The moduli space of super Riemann surfaces has been studied from the mathematical perspective in~\cites{CR-SRSUTT}{LBR-MSRS}{DW-SMNP}{FKP-MSSCCLB} and from the perspective of string theory in~\cites{W-NSRSTM}.

Here, we construct super stable maps of genus zero and their moduli spaces using the language of super differential geometry and superorbifolds.
We are guided by a functoriality principle:
The reduction of any result on super stable maps, that is setting all odd variables to zero, should reproduce the corresponding results for classical \(\targetACI\)-holomorphic curves.
Hence, we define super stable maps as marked trees of super \(\targetACI\)-holomorphic curves of genus zero.
That is, every vertex of the tree represents one map from the projective superspace \(\ProjectiveSpace[\C]{1|1}\) to an almost Kähler manifold that satisfies the differential equation of super \(\targetACI\)-holomorphic curves.
We show that the stability conditions proposed in~\cite{KM-GWCQCEG} also imply stability in the supergeometric setting, and consequently, the space of super stable maps of fixed tree type is a superorbifold.
Furthermore, we propose a generalization of Gromov topology for super stable maps.

To be more precise, we recall some results from~\cite{KSY-SJC}.
A map \(\Phi\colon \ProjectiveSpace[\C]{1|1}\to N\) from \(\ProjectiveSpace[\C]{1|1}\), the only super Riemann surface of genus zero, to a fixed almost Kähler manifold \((N, \omega, \targetACI)\) is called a super \(\targetACI\)-holomorphic curve if
\begin{equation}
	\DJBar\Phi
	= \frac12\left.\left(\differential{\Phi} + \targetACI\circ \differential{\Phi}\circ \ACI\right)\right|_\cD
	\in\VSec{\dual{\cD}\otimes\Phi^*\tangent{N}}
\end{equation}
vanishes.
Here, \(\ACI\) is the almost complex structure on \(\ProjectiveSpace[\C]{1|1}\) and \(\cD\subset\tangent{\ProjectiveSpace[\C]{1|1}}\) is a totally non-integrable distribution of complex rank \(0|1\) that defines the super Riemann surface structure.
We have shown that under certain transversality conditions on \(\Phi\) there is a deformation space of super \(\targetACI\)-holomorphic curves around \(\Phi\) of real dimension
\begin{equation}
	2n + 2\left<c_1(\tangent{N}), A\right>|2\left<c_1(\tangent{N}), A\right>,
\end{equation}
where \(2n\) is the real dimension of \(N\) and \(A\in H_2(N, \Integers)\) is the homology class of the image of \(\Phi\).
If the almost Kähler manifold \(N\) is chosen such that the transversality conditions are satisfied for all maps the moduli space \(\mathcal{M}_0(A)\) of super \(\targetACI\)-holomorphic curves is a supermanifold.

To better understand the supermanifold structure of \(\mathcal{M}_0(A)\) we consider its point functor
\begin{equation}
	\underline{\mathcal{M}_0(A)}\colon \cat{SPoint}^{op}\to \cat{Man}
\end{equation}
in the sense of Molotkov--Sachse:
Every superpoint \(C\in\cat{SPoint}\), that is a supermanifold of dimension \(0|s\), gives rise to the manifold \(\underline{\mathcal{M}_0(A)}(C)\) of maps \(C\to \mathcal{M}_0(A)\) and any map \(C'\to C\) between superpoints gives rise to a map \(\underline{\mathcal{M}_0(A)}(C)\to \underline{\mathcal{M}_0(A)}(C')\).
It was shown in~\cites{M-IDCSM}{S-GAASTS} that the point functor allows the reconstruction of the supermanifold structure.
A \(C\)-point \(C\to \mathcal{M}_0(A)\) is given by a super \(\targetACI\)-holomorphic curve \(\Phi\colon \ProjectiveSpace[\C]{1|1}\times C\to N\) parametrized by \(C\).
In particular, the \(\R^{0|0}\)-points \(\underline{\mathcal{M}_0(A)}(\R^{0|0})\) form the manifold of classical \(\targetACI\)-holomorphic curves of genus zero.

In this article, we construct a moduli space of simple super stable maps of genus zero with \(k\) marked points as a functor
\begin{equation}
	\underline{\overline{\mathcal{M}}^*_{0,k}(A)}\colon \cat{SPoint}^{op}\to \cat{Top}
\end{equation}
such that \(\underline{\overline{\mathcal{M}}^*_{0,k}(A)}(\R^{0|0})\) is the compact moduli space of simple stable maps of genus zero with \(k\) marked points as presented in~\cite{McDS-JHCST}.
A super stable map of genus zero parametrized over the superpoint \(C\) is then given by a tree, where each vertex corresponds to a super \(\targetACI\)-holomorphic curve of genus zero with marked points and each edge prescribes a common \(C\)-point of the super \(\targetACI\)-holomorphic curves at the corresponding vertices.
The stability condition that vertices mapping to a point have at least three marked points yield that the number of automorphisms of the super stable map is finite.
More precisely, the only possible automorphisms are reflections of the odd directions on each vertex for super stable maps over \(\R^{0|0}\).
In order for the functor to take values in the category of topological spaces we define a generalization of the Gromov topology for super stable maps.

The functor~\(\underline{\overline{\mathcal{M}}^*_{0,k}(A)}\) does not represent a supermanifold or superorbifold because already the moduli space \(\underline{\overline{\mathcal{M}}^*_{0,k}(A)}(\R^{0|0})\) is not a manifold in general.
But the restriction \(\underline{\mathcal{M}_{0,T}(\Set{A_\alpha})}\) to super stable maps of fixed tree type \(T\) and partition \(\Set{A_\alpha}\) of the homology class \(A\) on the vertices of the tree \(T\) is the orbit functor of a superorbifold.
The dimension of the orbifold decreases with the number of vertices.
We believe that this yields sufficient structure on \(\underline{\overline{\mathcal{M}}^*_{0,k}(A)}\) to define virtual fundamental cycle and supergeometric Gromov--Witten invariants as a Berezin integral to be defined in an upcoming work.

The moduli spaces \(\underline{\mathcal{M}_{0,T}(\Set{A_\alpha})}\) are constructed as quotient of a configuration space by the action of the automorphism group.
While it is known that the quotient of a supermanifold by a proper and free action of a super Lie group is a supermanifold, see~\cite{AH-IBIHS}, this action of the automorphism group is not free.
Using geodesics on Riemannian supermanifolds we prove a supergeometric analogue of the slice theorem, compare~\cite{P-ESANCLG}:
\begin{restatable*}{theorem}{ExistenceOfSlice}\label{thm:ExistenceOfSlice}
	Let \(M\) be a supermanifold over \(B\) with Riemannian metric \(m\), \(a\colon G\times_B M\to M\) a proper group action and \(p\colon B\to M\) a \(B\)-point of \(M\) with isotropy group \(H_p\).
	Suppose that
	\begin{itemize}
		\item
			\(H_p\) is a super Lie subgroup of \(G\), and
		\item
			\(H_p\) acts on \(M\) by isometries.
	\end{itemize}
	Then there exists a slice at \(p\).
\end{restatable*}
A slice \(S\) is a local complement to the orbits of the group action around the point \(p\) and it follows that the manifold \(M\) is locally, around \(p\), isomorphic to \(G\times_{H_p} S\).

For the investigation of the superorbifold structure of the moduli spaces we develop a systematic approach to superorbifolds via Morita equivalence of super Lie groupoids.
If a proper action of a super Lie group on a supermanifold has finite isotropy groups the quotient is a superorbifold.
Local charts for the quotient superorbifold are given by slices of the action divided by the isotropy group.
This yields the second main result:
\begin{restatable*}{theorem}{QuotientOrbifoldTheorem}\label{thm:QuotientByFiniteIsotropyProperGroupAction}
	Let \(G\) be a super Lie group that acts properly and with finite isotropy groups on a supermanifold \(M\) of dimension \(m|2n\).
	The Morita equivalence class of the transformation groupoid \(G\ltimes M\) is a superorbifold of dimension \(\dim M - \dim G\).
	We denote this superorbifold by \(\faktor{M}{G}\).
\end{restatable*}

We apply Theorem~\ref{thm:QuotientByFiniteIsotropyProperGroupAction} twice in this article.
First, we use Theorem~\ref{thm:QuotientByFiniteIsotropyProperGroupAction} to construct the moduli space of stable supercurves of genus zero with marked points.
A precise discussion of the superconformal automorphisms of \(\ProjectiveSpace[\C]{1|1}\) shows that any such automorphism of \(\ProjectiveSpace[\C]{1|1}\) is determined by the image of three points up to a sign.
This yields a description of the moduli space of stable supercurves of genus zero, fixed tree type, and with markings as a superorbifold obtained by the quotient of a power of \(\ProjectiveSpace[\C]{1|1}\) by the tree automorphisms in Proposition~\ref{prop:ModuliSpaceOfSRS0k}.
Similarly, a second application of Theorem~\ref{thm:QuotientByFiniteIsotropyProperGroupAction} shows that the functor \(\underline{\mathcal{M}^*_{0,T}(\Set{A_\alpha})}\) of equivalence classes of simple super stable maps of genus zero, with fixed tree type \(T\) and partition \(\Set{A_\alpha}\) of the homology class \(A\) is the orbit functor of a quotient superorbifold:
\begin{restatable*}{theorem}{ModuliSpaceOfSimpleStableMapsFixedTreeType}\label{thm:ModuliSpaceOfSimpleStableMapsFixedTreeType}
Let \(N\) be an almost Kähler manifold, \(T\) a \(k\)-labeled tree with \(\#E\) edges and \(\Set{A_\alpha}\) a partition of the homology class \(A\) on the tree \(T\).
	Assume that
	\begin{itemize}
		\item
			all the moduli spaces \(\mathcal{M}_0^*(A_\alpha)\) of simple super \(\targetACI\)-holomorphic curves of genus zero are supermanifolds,
		\item
			the edge evaluation map \(\ev^T\colon Z^T\times M^T\to N^{2\#E}\) is transversal to \(\Delta^T\) for all simple stable curves of genus zero modeled over the tree \(T\).
	\end{itemize}
	Then \(\underline{\mathcal{M}_{0,T}^*(\Set{A_\alpha})}\) is the orbit functor of the global quotient superorbifold
	\begin{equation}
		\mathcal{M}_{0,T}^*(\Set{A_\alpha}) = \faktor{Z^T\times M^T}{G^T}.
	\end{equation}
	The orbifold \(\mathcal{M}_{0,T}^*(\Set{A_\alpha})\) has dimension
	\begin{equation}
		2n + 2\left<A, c_1(TN)\right> - 2\#E + 2k - 6 | 2\left<A, c_1(TN)\right>  + 2k - 4.
	\end{equation}
	and isotropy group \(\Z_2^{\#E+1}\) on the \(\R^{0|0}\)-points, generated by the maps \(\Xi_-^\alpha\) which act by reflection of the odd direction on the node \(\alpha\in T\) and identity on the others.
\end{restatable*}
Here \(Z^T\) is an open sub-supermanifold of a power of \(\ProjectiveSpace[\C]{1|1}\) and \(M^T\) an open sub-supermanifold of the product of the moduli spaces \(\mathcal{M}^*_0(A_\alpha)\).
The conditions of Theorem~\ref{thm:ModuliSpaceOfSimpleStableMapsFixedTreeType} can in certain cases be satisfied by generic perturbation of the almost complex structure on the target almost Kähler manifold \((N, \omega, \targetACI)\).

We point out that we only consider stable supercurves and super stable maps of genus zero with Neveu--Schwarz punctures and nodes.
It was shown in~\cites{D-LaM} that in the local deformation theory of super Riemann surfaces two types of nodes, Neveu--Schwarz type nodes and Ramond type nodes, may appear.
But it is consistent to restrict to Neveu--Schwarz type nodes and punctures only because the degree of Ramond divisors has to be even on every irreducible component, see Remark~\ref{rem:NSPuncturesAndRPunctures}.
Neveu--Schwarz punctures can be added to a given super Riemann surface and are compatible with the principle of functoriality explained above.
The case of Ramond punctures which requires modification of the definition of super Riemann surfaces and  of the methods developed in~\cite{KSY-SJC} is left for later.

The paper is organized as follows:
In Chapter~2, we discuss group actions and quotients in supergeometry.
Besides recalling the notation, the main goal of Chapter~2 is to show that the quotient of a supermanifold of dimension \(m|2n\) by a proper action of a super Lie group with finite isotropy groups is a superorbifold.
To this end, we define geodesics and the exponential map on Riemannian supermanifolds and prove a slice theorem for proper group actions on Riemannian supermanifolds of dimension \(m|2n\).
We give a rigorous discussion of superorbifolds via super Lie groupoids up to Morita equivalence.

In Chapter~3, we discuss stable supercurves of genus zero with marked points.
In a first step we show that the superconformal automorphisms of \(\ProjectiveSpace[\C]{1|1}\) that fix three points consist only of the identity and the reflection of the odd directions.
We then define stable curves of genus zero as a stable tree of copies of \(\ProjectiveSpace[\C]{1|1}\) and show that for a fixed tree type their moduli space is a superorbifold.

Super stable maps are discussed in Chapter~4.
We define super stable maps of genus zero as maps from marked nodal supercurves that allow only a finite number of automorphisms and give the definition of the point functor of the moduli space of super stable curves with a fixed number of marked points.
We show that this point functor can be refined to take values in the category of topological spaces in a way that extends classical Gromov topology and such that its restriction to simple super stable curves of fixed tree type yields the orbit functor of a superorbifold.

\subsection*{Acknowledgments}
Enno Keßler was supported by a Deutsche Forschungsgemeinschaft Research Fellowship, KE2324/1--1.
Artan Sheshmani was supported partially by the NSF~DMS-1607871, NSF~PHY-1306313, the Simons~38558, and Laboratory of Mirror Symmetry NRU HSE, RF Government grant, ag. No~14.641.31.0001.
Shing-Tung Yau was partially supported by NSF~DMS-1607871, NSF~PHY-1306313, and Simons~38558.
We thank Tyler Jarvis, Jürgen Jost, Albrecht Klemm, Slava Matveev and Paolo Perrone for useful discussion and comments.
 
\section{Slice theorem and superorbifolds}
In this chapter we study the quotient of supermanifolds by proper actions of super Lie groups such that the isotropy groups of the action are finite.
This is the main technical tool for the construction of the moduli spaces of stable curves and maps in later chapters.
We proceed in several steps:
First we set the basic notations for supermanifolds in the ringed space approach and their functor of points in Sections~\ref{Sec:RingedSpaceApproachToSupermanifolds} and~\ref{sec:FunctorOfPoints}.
In Section~\ref{sec:ProductsAndQuotients} we recall quotients of supermanifolds by regular equivalence relations.
Geodesics and exponential maps on Riemannian supermanifolds are constructed in Section~\ref{sec:GeodesicsOnRiemannianSupermanifolds}.
We use the exponential map in Section~\ref{sec:GroupActions} to construct slices, that is, local complements to the orbit of a proper action of a super Lie group.
This yields a supergeometric analogue of the slice theorem by~\cite{P-ESANCLG} for proper actions of a super Lie group on a Riemannian supermanifold.
A proper and free group action yields a regular equivalence relation and the quotient is a supermanifold.
But if the group action has finite isotropy groups the quotient is a superorbifold instead.
While superorbifolds have appeared implicitly before, we give a systematic construction of superorbifolds using super Lie groupoids in Section~\ref{sec:Superorbifolds}.

\subsection{The ringed space approach to supermanifolds}\label{Sec:RingedSpaceApproachToSupermanifolds}
In this article we need two approaches to supermanifolds:
The Berezin--Kostant--Leites approach via ringed spaces and the Molotkov--Sachse approach via the functor of points.
For a detailed introduction to supermanifolds in the ringed space approach, we refer to~\cites{L-ITS}{DM-SUSY}{EK-SGSRSSCA}.
In this approach, a supermanifold~\(M\) consists of a topological space~\(\Top{M}\) together with a sheaf of supercommutative rings~\(\cO_M\) such that  \(M=(\Top{M}, \cO_M)\) is locally isomorphic to
\begin{equation}
	\R^{m|n} = \left(\R^m, \cO_{\R^{m|n}}=C^\infty(\R^m, \R)\otimes_\R {\bigwedge}_{n}\right)
\end{equation}
as a locally ringed, supercommutative space.
The superdomain \(\R^{m|n}\) consists of the euclidean topological space \(\R^m\) together with the sheaf of real-valued smooth functions in \(m\) variables twisted by a real Graßmann algebra in \(n\) generators.
For the standard coordinates \(x^a\), \(a=1, \dotsc, m\) on \(\R^m\) and \(\eta^\alpha\), \(\alpha=1, \dotsc, n\) generators of the Graßmann algebra \({\bigwedge}_n\) we say that \(X^A=(x^a, \eta^\alpha)\) are the standard supercoordinates on \(\R^{m|n}\).
Here and henceforth we use the convention that small Latin characters number even objects, small Greek letters number odd objects and capital Latin letters number both even and odd objects.

Maps \(f\colon M\to N\) between supermanifolds are given by morphisms of locally ringed spaces.
By a Theorem of Leites, see~\cite[Theorem~2.1.7]{L-ITS}, any map between open superdomains \(f\colon \R^{p|q}\supset U\to V\subset\R^{m|n}\) is completely determined by the image of the coordinates \(X^A=(x^a, \eta^\alpha)\) of \(\R^{m|n}\) under \(f^\#\colon \cO_{\R^{m|n}}\to\cO_{\R^{p|q}}\).
Maps between supercommutative rings preserve the \(\Z_2\)-grading which for elements of \(\cO_{\R^{m|n}}\) is induced from the \(\Z\)-grading of \({\bigwedge}_n\).
Consequently, the sheaf \(\cO_M\) of any supermanifold \(M\) possesses a \(\Z_2\)-grading \(\cO_M={\left(\cO_M\right)}_0\oplus{\left(\cO_M\right)}_1\).

For our purposes we need to work with families of supermanifolds instead of single supermanifold.
A family of supermanifolds over the supermanifold \(B\) is a submersion \(b_M\colon M\to B\).
We will assume that \(B\) is a superpoint, that is of the form \(B=\R^{0|s}\).
Then, the family \(b_M\) is locally on \(M\) a projection \(U\times B\to B\).
A map between two families of supermanifolds \(b_M\colon M\to B\) and \(b_N\colon N\to B\) is a map \(f\colon M\to N\) such that \(b_N\circ f = b_M\).
However, we do not assume that \(B\) is fixed and allow for arbitrary base change and do only consider geometric constructions that are relative to \(B\) and invariant under base change.
In particular, we will assume that any supermanifold can be seen as a, possibly trivial, family over \(B\) and that \(B\) is large enough, that is, contains as many odd directions as needed.
For a detailled discussion of families of supermanifolds we refer to~\cite[Chapter~3]{EK-SGSRSSCA}.

\subsection{The functor of points approach to supermanifolds}\label{sec:FunctorOfPoints}
The functor of points approach uses an idea from algebraic geometry:
Describing a supermanifold \(M\) in terms of maps \(C\to M\).
We use the formalism developed in~\cites{M-IDCSM}{S-GAASTS} which describes a supermanifold purely in terms of its superpoints.
Let \(\cat{SPoint}\) be the category of superpoints, that is, of supermanifolds isomorphic to \(\R^{0|s}\) for some \(s\) and smooth maps between them.
For a supermanifold \(M\) and a superpoint \(C\in\cat{SPoint}\) we write
\begin{equation}
	\underline{M}(C) = \Set{p\colon C\to M}
\end{equation}
for the set of \(C\)-points of \(M\).
Explicitly, for \(M=\R^{m|n}\) a map \(p\colon C\to \R^{m|n}\) is given by a tuple
\begin{align}
	p^a=p^\#x^a&\in {\left(\cO_C\right)}_0, &
	p^\alpha = p^\#\eta^\alpha &\in{\left(\cO\right)}_1,
\end{align}
where \((x^a, \eta^\alpha)\) are coordinates on \(\R^{m|n}\).
The functions \(p^a\) and \(p^\alpha\) combine to an even element \((p^a, p^\alpha)\) of the graded tensor product \(\R^{m|n}\otimes \cO_C\) of graded vector spaces.
It follows that
\begin{equation}
	\underline{\R^{m|n}}(C)
	= {\left(\R^{m|n}\otimes \cO_C\right)}_0
	= {\left(\R^{m|n}\otimes {\bigwedge}_s\right)}_0
	= {\left({\bigwedge}_s\right)}_0^m\oplus{\left({\bigwedge}_s\right)}_1^n
\end{equation}
is a manifold.
As every supermanifold is locally isomorphic to \(\R^{m|n}\), also \(\underline{M}(C)\) is a manifold.
The manifold \(\underline{M}(\R^{0|0})\) coincides with the reduced manifold \(\Red{M}\) of \(M\).

For every map \(c\colon C'\to C\) and \(p\in \underline{M}(C)\) the composition \(p\circ c\) is in \(\underline{M}(C')\).
Hence \(\underline{M}\) is a functor from the opposite of the category of superpoints to the category of manifolds:
\begin{equation}
	\underline{M}\colon \cat{SPoint}^{op}\to \cat{Man}
\end{equation}
Conversely, it was worked out in~\cites{M-IDCSM}{S-GAASTS} that a functor \(\underline{M}\colon \cat{SPoint}^{op}\to \cat{Man}\) is a supermanifold, if and only if it possesses a cover by open subfunctors which are isomorphic to \(\underline{\R^{m|n}}\).
Here a subfunctor \(\underline{U}\subset \underline{M}\) is open if for all \(C\) the subset \(\underline{U}(C)\) is open in \(\underline{M}(C)\).
Maps between supermanifolds are realized as functor morphisms \(\underline{M}\to \underline{N}\).

A \(C\)-point of a family \(b_M\colon M\to B\) of supermanifolds consists of a commutative triangle
\begin{diag}
	\matrix[dr](m){
		C&& M \\
		& B & \\
	} ;
	\path[pf]{
		(m-1-1) edge node{\(p\)} (m-1-3)
		edge node[swap]{\(C_B\)} (m-2-2)
		(m-1-3) edge node{\(b_M\)}(m-2-2)
	};
\end{diag}
The local picture is
\begin{diag}
	\matrix[dr](m){
		C&& \R^{m|n}\times B \\
		& B & \\
	} ;
	\path[pf]{
		(m-1-1) edge node{\(p\)} (m-1-3)
		edge node[swap]{\(C_B\)} (m-2-2)
		(m-1-3) edge node{\(b_M\)}(m-2-2)
	};
\end{diag}
For fixed \(C_B\) and \(b_M\) the set of \(C_B\)-points of \(\R^{m|n}\times B\) coincides with \(\underline{\R^{m|n}}(C)={\left(\R^{m|n}\otimes \cO_C\right)}_0\).
Consequently, a supermanifold \(M\) over \(B\) yields a functor
\begin{equation}
	\underline{M}\colon \cat{SPoint}_B^{op}\to \cat{Man}
\end{equation}
which is locally isomorphic to \(\underline{\R^{m|n}}\).
Here \(\cat{SPoint}_B\) is the category of superpoints over \(B\) where the objects consist of maps \(C_B\colon C\to B\) from a superpoint \(C\) to \(B\) and the morphisms \(C_B\to C'_B\) are commutative triangles of the form
\begin{diag}
	\matrix[dr](m){
		C&& C' \\
		& B & \\
	} ;
	\path[pf]{
		(m-1-1) edge (m-1-3)
		edge node[auto, swap]{\(C_B\)} (m-2-2)
		(m-1-3) edge node[auto]{\(C'_B\)}(m-2-2)
	};
\end{diag}
The category \(\cat{SPoint}_B\) contains an initial object \(\R^{0|0}_B\colon \R^{0|0}\to B\) and a terminal object \(\id_B\colon B\to B\).
Furthermore, as \(\R^{0|0}\) is also terminal object in \(\cat{SPoint}\) we have \(\cat{SPoint}_{\R^{0|0}}=\cat{SPoint}\).

\subsection{Products and quotients}\label{sec:ProductsAndQuotients}
In this section we recall the notions of fibered product and quotients of a supermanifold by a regular equivalence relation.
Here we follow the work of~\cites{A-STGM}{BBRHP-QS}{AH-IBIHS}.
Let \(f_1\colon M_1\to N\) and \(f_2\colon M_2\to N\) be two maps between supermanifolds.
The fibered product \(M_1\times_{f_1,f_2} M_2\) is a supermanifold together with two projections \(p_i\colon M_1\times_{f_1, f_2} M_2\to M_i\) such that \(f_1\circ p_1=f_2\circ p_2\).
The fiber product is characterized by the following universal property:
For any supermanifold \(S\) and maps \(g_i\colon S\to M_i\) such that \(f_1\circ g_1=f_2\circ g_2\) there exists a map \((g_1, g_2)\colon S\to M_1\times_{f_1, f_2} M_2\) such that \(g_i=p_i\circ (g_1, g_2)\).
\begin{diag}
	\matrix(m)[mat] {S &&\\
		& M_1\times_{f_1, f_2}M_2 & M_2\\
		& M_1 & N \\};
	\path[pf]
		(m-1-1)	edge [densely dotted] node[auto] {\((g_1,g_2)\)} (m-2-2)
			edge [bend left=20] node[auto] {\(g_2\)} (m-2-3)
			edge [bend right=20] node[auto] {\(g_1\)} (m-3-2)
		(m-2-2) edge node[auto,swap] {\(p_2\)} (m-2-3)
			edge node[auto] {\(p_1\)} (m-3-2)
		(m-2-3) edge node[auto] {\(f_2\)} (m-3-3)
		(m-3-2) edge node[auto] {\(f_1\)} (m-3-3);
\end{diag}
A sufficient condition for existence of the fibered product is that one of the maps \(f_i\) is a submersion.
If the fibered product exists, it is of dimension
\begin{equation}
	\dim M_1\times_{f_1, f_2} M_2 = \dim M_1 + \dim M_2 - \dim N
\end{equation}
and the projections \(p_i\) are surjective submersions.
If the maps \(f_i\) are clear from the context, for example for families \(f_i=b_{M_i}\colon M_i\to B\) we may also use the shorthands \(M_1\times_B M_2\) or \(M_1\times M_2\).

\begin{defn}\label{defn:EquivalenceRelation}
	Let \(M\) be a supermanifold and \(j\colon R\to M\times M\) a subsupermanifold.
	Let \(\Delta\colon M\to M\times M\) denote the diagonal morphism, \(p_1\), \(p_2\colon M\times M\to M\) the projection on the first and second factor respectively and \(\overline{p}_k=p_k\circ j\).
	The subsupermanifold \(j\colon R\to M\times M\) is called an equivalence relation in \(M\) if
	\begin{enumerate}
		\item\label{item:defn:EquivalenceRelation:Diagonal}
			there exists a morphism \(\delta\colon M\to R\) such that \(j\circ \delta = \Delta\),
		\item\label{item:defn:EquivalenceRelation:Associativity}
			there is a morphism \(c\colon R\times_{\overline{p}_2, \overline{p}_1} R\to R\) such that
				\begin{equation}
					\begin{tikzpicture}[commutative diagrams/.cd, every diagram, baseline=(current bounding box.center)]
						\matrix[mat](m){
							R\times_{\overline{p}_2,\overline{p}_1} R & R\\
							R & M\\
						} ;
						\path[pf]{
							(m-1-1) edge node[auto]{\(c\)} (m-1-2)
								edge node[auto]{\(\pi_k\)} (m-2-1)
							(m-1-2) edge node[auto]{\(\overline{p}_k\)}(m-2-2)
							(m-2-1) edge node[auto]{\(\overline{p}_k\)}(m-2-2)
						};
					\end{tikzpicture}
					\hspace{5em}
					\begin{tikzpicture}[commutative diagrams/.cd, every diagram, baseline=(current bounding box.center)]
						\matrix[mat, column sep=large](m){
							R\times_{\overline{p}_2,\overline{p}_1}R\times_{\overline{p}_2,\overline{p}_1}R & R\times_{\overline{p}_2,\overline{p}_1}R \\
							R\times_{\overline{p}_2,\overline{p}_1}R & R\\
						} ;
						\path[pf]{
								(m-1-1) edge node[auto]{\(\id\times_{\overline{p}_2,\overline{p}_1} c\)} (m-1-2)
								edge node[auto]{\(c\times_{\overline{p}_2,\overline{p}_1} \id\)} (m-2-1)
							(m-1-2) edge node[auto]{\(c\)}(m-2-2)
							(m-2-1) edge node[auto]{\(c\)}(m-2-2)
						};
					\end{tikzpicture}
				\end{equation}
				are commutative.
				Here \(\pi_1\), \(\pi_2\colon R\times_{\overline{p}_2,\overline{p}_1}R\to R\) are the projections on the first and second factor respectively.
			\item\label{item:defn:EquivalenceRelation:Symmetry}
				there is a map \(i\colon R\to R\) such that \(i\circ i=\id_R\) and \(\overline{p}_1\circ i = \overline{p}_2\).
	\end{enumerate}
	The equivalence relation is called regular if \(j\colon R\to M\times M\) is a closed subsupermanifold and \(\overline{p}_1\colon R\to M\) is a submersion.
\end{defn}

\begin{defn}\label{defn:QuotientEquivalenceRelation}
	Let \(j\colon R\to M\times M\) be an equivalence relation in \(M\).
	A submersion \(q\colon M\to Q\) is called quotient of \(M\) by \(R\) if \(q\circ\overline{p}_1 = q\circ\overline{p}_2\) and any map \(f\colon M\to N\) such that \(f\circ\overline{p}_1 = f\circ \overline{p}_2\) factors over \(Q\), that is there exists a map \(\overline{f}\colon Q\to N\) such that \(f=\overline{f}\circ q\).
	\begin{diag}
		\matrix[mat](m){
			R & M & Q\\
				& & N\\
		} ;
		\path[pf]{
			(m-1-1) edge[bend left=40] node[auto]{\(\overline{p}_1\)} (m-1-2)
				edge[bend right=40] node[auto, swap]{\(\overline{p}_2\)} (m-1-2)
			(m-1-2) edge node[auto]{\(q\)} (m-1-3)
				edge node[auto]{\(f\)} (m-2-3)
			(m-1-3) edge[commutative diagrams/dashrightarrow]  node[auto]{\(\overline{f}\)} (m-2-3)
		};
	\end{diag}
	Such a quotient, if it exists, is unique up to superdiffeomorphism and will be denoted by~\(\faktor{M}{R}\).
\end{defn}

The following supergeometric analogue of a theorem of Godement's theorem, see~\cite[Section~5.9.5]{B-VDA}, is proved in~\cite[Theorem~2.6]{AH-IBIHS} based on the work of~\cites{A-STGM}{BBRHP-QS}:
\begin{prop}\label{prop:QuotientByRegularEquivalenceRelation}
	The quotient \(\faktor{M}{R}\) exists if and only if the equivalence relation \(R\) is regular.
	In that case the map
	\begin{equation}
		\overline{p}_1\times_{\faktor{M}{R}}\overline{p}_2\colon R \to M\times_{\faktor{M}{R}} M
	\end{equation}
	is a superdiffeomorphism.
\end{prop}
\begin{rem}
	The existence of a superdiffeomorphism identifying \(R=M\times_{\faktor{M}{R}} M\) implies that for both even and odd dimension we have
	\begin{equation}\label{eq:DimensionQuotientEquivalenceRelation}
		\dim R = 2\dim M - \dim \faktor{M}{R}.
	\end{equation}
\end{rem}
\begin{rem}\label{rem:EquivalenceRelationSuperPoints}
	The Definition~\ref{defn:EquivalenceRelation} of equivalence relation and the Definition~\ref{defn:QuotientEquivalenceRelation} of quotient reduce to their classical counterpart for point sets when the supermanifold \(M\) has odd dimension zero.
	More generally, if \(M\) is a supermanifold over \(B\) and \(C_B\colon C\to B\) a \(B\)-point the set \(\underline{R}(C_B)\) of \(C_B\)-points of \(R\) form an equivalence relation over the set~\(\underline{M}(C_B)\) of \(C_B\) points of \(M\).
	If the quotient \(Q\) exists as a supermanifold, it holds \(\underline{Q}(C_B)=\faktor{\underline{M}(C_B)}{\underline{R}(C_B)}\) as manifolds, see~\cite[Section~6.4]{M-IDCSM}.

	If the quotient does not exist, one can still define a functor with values in the category of topological spaces
	\begin{equation}
		\begin{split}
			\cat{SPoint}^{op}_B &\to \cat{Top} \\
			C_B &\mapsto \faktor{\underline{M}(C_B)}{\underline{R}(C_B)}
		\end{split}
	\end{equation}
	which represents the equivalence classes of superpoints of the quotient, but does not represent a supermanifold.
	In Section~\ref{sec:Superorbifolds} below we will see how to give this quotient the structure of a superorbifold in certain cases where the equivalence relation is not regular.
\end{rem}

\subsection{Geodesics on Riemannian supermanifolds}\label{sec:GeodesicsOnRiemannianSupermanifolds}
In this section we construct geodesics and exponential maps on supermanifolds with a Riemannian metric.
While geodesics on supermanifolds have appeared before in~\cites{G-RSG}{GW-GFRS}, the variant we propose here allows for an exponential map that identifies an open subset of the superorbifold around a point with an open subset of the tangent space.

Consider a curve \(\gamma\), that is, a map from an interval \(I\subset\R\) to \(M\) parametrized by the superpoint \(C_B\):
\begin{equation}
	\gamma\colon I\times C\to M\times_B C.
\end{equation}
For any time parameter \(t\in I\) the evaluation \(\gamma(t)\) is a \(C_B\)-point of \(M\), \(\gamma(t)\colon C\to M\times_B C\) or \(\gamma(t)\in \underline{M}(C_B)\).
The tangent vector of the curve at time \(t\) is given by an even element
\begin{equation}
	\left.\differential{\gamma}\left(\partial_t\right)\right|_{t}\in\VSec{C_B^*p^*\tangent{M}} = \VSec{p^*\tangent{M}}\otimes_{\cO_B} \cO_C.
\end{equation}

In order to simplify the notation we introduce the following:
\begin{defn}
	Let \(M\) be a supermanifold over \(B\), \(E\to M\) a super vector bundle and \(p\colon B\to M\) be a \(B\)-point of \(M\).
	We define the functor \(\underline{\VSec[p]{E}}\) of sections of \(E\) at \(p\) by
	\begin{equation}
		\begin{split}
			\underline{\VSec[p]{E}}\colon \cat{SPoint}^{op}_B &\to \cat{Mod}_{{\left(\cO_B\right)}_0} \\
			\left(C_B\colon C\to B\right) &\mapsto {\VSec{C_B^*p^*E}}_0 = {\left(\VSec{p^*E}\otimes_{\cO_B}\cO_C\right)}_0
		\end{split}
	\end{equation}
\end{defn}
Using this notation we write \(\left.\differential{\gamma}\left(\partial_t\right)\right|_{t}\in \underline{\VSec[p]{\tangent{M}}}(C_B)\).
The tangent map \(\tangent{f}\colon \tangent{M}\to \tangent{N}\) of a map \(f\colon M\to N\) induces functor transformations \(\underline{\VSec[p]{\tangent{M}}}\to \underline{\VSec[f\circ p]{\tangent{N}}}\).

Let now \(m\) a Riemannian metric on \(\tangent{M}\) and \(\nabla\) the corresponding Levi-Civita connection.
A Riemannian metric on the supermanifold \(M\to B\) is an even, symmetric, non-degenerate \(\cO_M\)-bilinear form \(m\colon\VSec{\tangent{M}}\times\VSec{\tangent{M}}\to\cO_M\) such that its reduction~\(\Red{m}\) is a positive definite metric on \(\Red{\left(\tangent{M}\right)}\to \Red{M}\), see~\cite[Definition~6.9.1]{EK-SGSRSSCA}.
In particular, the symmetry condition requires that for all \(X\), \(Y\in\VSec{\tangent{M}}\)
\begin{equation}
	m(X, Y) = {\left(-1\right)}^{\p{X}\p{Y}}m(Y, X).
\end{equation}
The symmetry condition together with the non-degeneracy requires that the number of odd dimensions is even.
Conversely, Riemannian metrics can be shown to exist on any supermanifold of dimension \(p|2q\) by using a partition of unity as in classical differential geometry.
Given an isomorphism of supermanifolds \(f\colon N\to M\) we denote the induced metric on \(N\) by \(m_f\).
For the theory of connections on vector bundles and principal bundles in supergeometry see~\cite[Chapter~4.4, Chapter~6]{EK-SGSRSSCA}.
The existence of a Levi-Civita connection~\(\nabla\) on supermanifolds with Riemannian metric has been shown in~\cites[Lemma~6.9.6]{EK-SGSRSSCA}[Theorem~4.1]{G-RSG}.

\begin{prop}\label{prop:ExistenceOfGeodesics}
For every \(B\)-point \(p\colon B\to M\) and every \(v\in\underline{\VSec[p]{\tangent{M}}}(C_B)\) there exists an open interval \(I=(-T, T)\subset\R\) and a unique map over \(C\)
	\begin{equation}
		\gamma_{p,v}\colon I\times C\to M\times_B C
	\end{equation}
	such that
	\begin{align}
		\nabla_{\partial_t} \differential{\gamma_{p,v}}\left(\partial_t\right) &=0, &
		\gamma(0) &= p\circ C_B, &
		\left.\differential{\gamma_{p,v}}\left(\partial_t\right)\right|_{t=0} &= v.
	\end{align}
	We call \(\gamma_{p,v}\) the geodesic going through the point \(p\) in direction of \(v\).
\end{prop}
\begin{proof}
	Let \((X^A)=(x^a, \eta^\alpha)\) be relative coordinates in a neighborhood of \(p\) and denote the Christoffel symbols of the Levi-Civita connection on \(M\) by \(\nabla_{\partial_{X^D}}\partial_{X^E}= \Gamma_{DE}^F\partial_{X^F}\).
	Suppose that \(p_C=p\circ C_B\) is given in those coordinates by \(p_C^{\#}X^A = p^A\in{\left(\cO_C\right)}_{\p{A}}\) and
	\begin{equation}
		v = v^a\partial_{x^a} + v^\alpha\partial_{\eta^\alpha}
	\end{equation}
	for \(v^a\in{\left(\cO_C\right)}_0\) and \(v^\alpha\in{\left(\cO_C\right)}_1\).
	Writing \(\gamma^A(t)=\gamma_{p,v}^{\#}X^A\) as well as \(\dot{\gamma}^A(t)\) and \(\ddot{\gamma}^A(t)\) for its first and second time derivative respectively, the differential equations for geodesics read
	\begin{align}
		\ddot{\gamma}^A(t) + \dot{\gamma}^E(t)\dot{\gamma}^D(t){\gamma(t)}^{\#}\Gamma_{DE}^A &= 0, &
		\gamma^A(0) &= p^A, &
		\dot{\gamma}^A(0) &= v^A.
	\end{align}
	This is a set of ordinary differential equations of second order with coefficients in \(\cO_C\) that can be solved uniquely by expanding recursively with respect to the odd generators of \(\cO_C\).
	Notice that the equation of degree zero corresponds to the equation for \(C_B=\R^{0|0}_B\) and is hence the usual geodesic equation, that is, a non-linear ordinary differential equation.
	The differential equations for the higher order coefficients are linear ordinary differential equations depending on the reduced geodesic.
	As linear ordinary differential equations can be solved for all times, super geodesics can be extended in time-direction as far as the reduced geodesic.
\end{proof}
Geodesics in supergeometry share some features with their classical counterpart:
\begin{itemize}
	\item
		Geodesics are parametrized proportionally to arc-length.
		Indeed,
		\begin{equation}
			\frac{\d}{\d{t}}\norm{\differential{\gamma_{p,v}}(\partial_t)}^2
			= m\left(\nabla_{\partial_t}\differential{\gamma_{p,v}}(\partial_t), \differential{\gamma_{p,v}}(\partial_t)\right) + m\left(\differential{\gamma_{p,v}}(\partial_t), \nabla_{\partial_t}\differential{\gamma_{p,v}}(\partial_t)\right)
			= 0.
		\end{equation}
		Consequently, \(\norm{\differential{\gamma_{p,v}}(\partial_t)}^2=\norm{v}^2\) for all times \(t\).
	\item
		Let \(\gamma_{p,v}\colon (-T, T)\times C\to M\times_B C\) be the geodesic going through \(p\) and in direction \(v\in\underline{\VSec[p]{\tangent{M}}}(C_B)\).
		For any \(\lambda\in\R\) the map \(\gamma\colon (-\frac{T}{\lambda}, \frac{T}{\lambda})\times C\to M\times_B C\) given by \(\gamma(t) = \gamma_{p,v}(\lambda t)\) is also a geodesic and satisfies \(\gamma(0)=p\) and \(\left.\differential{\gamma}(\partial_t)\right|_{t=0}=\lambda v\).
		Hence \(\gamma_{p,\lambda v}(t) = \gamma_{p, v}(\lambda t)\) for all \(\lambda\in\R\).
	\item
		Let \(f\colon M\to M\) be an isometry, that is \(m_f=m\).
		Then \(f\) maps geodesics to geodesics.
\end{itemize}
\begin{prop}
	There is an open subfunctor of \(\underline{U}\subset\underline{\VSec[p]{\tangent{M}}}\) such that for all \(C_B\colon C\to B\) and \(v\in \underline{U}(C_B)\) the geodesic \(\gamma_{p,v}\) is defined on an open interval containing \([-1,1]\).
	We define the exponential map \(\underline{\exp_p}\colon\underline{U}\to\underline{M}\) at \(p\colon B\to M\) by
	\begin{equation}
		\begin{split}
			\underline{\exp_p}(C_B)\colon \underline{U}(C_B)\subset\underline{\VSec[p]{\tangent{M}}}(C_B)&\to \underline{M}(C_B) \\
			v&\mapsto \gamma_{p,v}(1)
		\end{split}
	\end{equation}
	The exponential map is an isomorphism from an open subfunctor \(\underline{U}'\subset \underline{U}\) onto its image, an open subfunctor of \(\underline{M}\) around \(p\).
\end{prop}
\begin{proof}
	It is known that there is an open subset \(U\subset\underline{\VSec[p]{\tangent{M}}}(\R^{0|0}_B)\) such that for every \(v\in U\) the geodesic \(\gamma_{p,v}\) is defined on an open interval containing \([-1,1]\).
	Let \(\underline{U}=\left.\underline{\VSec[p]{\tangent{M}}}\right|_U\).
	As explained in the proof of Proposition~\ref{prop:ExistenceOfGeodesics} this implies that for every \(C_B\in\cat{SPoint}_B^{op}\) and \(v\in\underline{U}(C_B)\) the geodesic \(\gamma_{p,v}\) can be defined on an open interval containing \([-1,1]\).
	Consequently, the exponential map \(\underline{\exp_p}\) is well defined on \(\underline{U}\).

	By construction \(\underline{\exp_p}(C_B)(0)=p_C\).
	We will show, in analogy to classical differential geometry, that \(\differential{\underline{\exp_p}}=id\colon \underline{\VSec[p]{\tangent{M}}}\to\underline{\VSec[p]{\tangent{M}}}\), where we identify \(\underline{\VSec[p]{\tangent{U}}}\) with \(\underline{\VSec[p]{\tangent{M}}}\).
	Indeed, for \(v\in\underline{\VSec[p]{\tangent{M}}}(C_B)\), we have
	\begin{equation}
		\left(\left(\differential{\underline{\exp_p}}\right)(C_B)\right)v
		= \left(\differential{\left(\underline{\exp_p}(C_B)\right)}\right)v
		= \left.\frac{\d}{\d{t}}\right|_{t=0}\gamma_{p, tv}(1)
		= \left.\frac{\d}{\d{t}}\right|_{t=0}\gamma_{p, v}(t)
		= v.
	\end{equation}

	From the invertibility of the differential of \(\underline{\exp_p}\) it follows that \(\underline{\exp_p}\) is a local isomorphism by the supergeometric inverse function theorem, see, for example,~\cite[Proposition~6.3.1]{M-IDCSM}.
\end{proof}

\subsection{Group actions}\label{sec:GroupActions}
In this section, we first recall the notions of a super Lie groups and actions of super Lie groups on supermanifolds.
Then we give an explicit construction of the quotient of a supermanifold by a proper and free group action.
The main new contribution in this section is a generalization of the slice theorem to certain group actions on Riemannian supermanifolds in Proposition~\ref{prop:SliceFoliation}:
In particular, any supermanifold \(M\) with an even number of odd dimensions and a proper group action by \(G\) with finite isotropy group at a \(B\)-point \(p\) can be written in a neighborhood of the point \(p\) as \(G\times_{H_p} S\), where \(H_p\) is the isotropy group at \(p\).
This is in preparation for Section~\ref{sec:Superorbifolds} where we will see that the quotient of a manifold by a proper group action with finite isotropy groups is a superorbifold.

A super Lie group is a supermanifold \(G\) over \(B\) together with maps \(m\colon G\times_B G \to G\), \(i\colon G\to G\) and \(e\colon B\to G\) such that for any \(C_B\colon C\to B\) the maps \(\underline{m}(C_B)\), \(\underline{i}(C_B)\) and \(\underline{e}(C_B)\) equip \(\underline{G}(C_B)\) with a multiplication, inverse and neutral element satisfying the group laws.
Examples relevant to the work here are subgroups of the general linear groups \(\GL(r|s)\) and finite groups as super Lie groups of dimension \(0|0\).
The module of sections \(\underline{\VSec[e]{\tangent{G}}_0}\) can be identified with the \(\cO_B\)-module of left-invariant vector fields and inherits the structure of a Lie superalgebra which we denote by \(\mathfrak{g}\).
For details we refer to~\cite[Chapter~5]{EK-SGSRSSCA} and~\cites{DM-SUSY}{CCF-MFS}.

A  left-action of a super Lie group \(G\) over \(B\) on a supermanifold \(M\) over \(B\) is given by a map \(a\colon G\times_B M\to M\) such that for any \(C_B\colon C\to B\) the map \(\underline{a}(C_B)\colon \underline{G}(C_B)\times \underline{M}(C_B)\to \underline{M}(C_B)\) is an action of the group \(\underline{G}(C_B)\) on \(\underline{M}(C_B)\).

\begin{ex}[Action of \(\GL(p|q)\) on the supermanifold \(\R^{p|q}\)]
	The group of even, invertible linear maps from the super vector space \(\R^{p|q}\) to itself can be equipped with the structure of a super Lie group acting on the supermanifold \(\R^{p|q}\) as follows:
	We define its point functor by setting \(\underline{\GL(p|q)}(C) = \GL_{\cO_C}(\R^{p|q}\otimes \cO_C)\), that is, the \(C\)-points of \(\GL(p|q)\) are \(\cO_C\)-linear even automorphisms of \(\R^{p|q}\otimes \cO_C\).
	The multiplication map multiplies \(C\)-points, and the inverse and identity are likewise defined on \(C\)-points.

	Elements of \(\underline{\GL(p|q)}(C)\) act on \(\underline{\R^{p|q}}(C)\) which defines the action of the super Lie group~\(\GL(p|q)\) on the supermanifold \(\R^{p|q}\).
	For a description of \(\GL(p|q)\) in the ringed space approach, see~\cite{EK-SGSRSSCA} and references therein.
	Using subgroups of \(\GL(p|q)\) and homomorphisms of super Lie groups one can construct further interesting examples, see, for example,~\cite{FG-CS}.
\end{ex}
\begin{ex}[Action of \(\R^{r|s}\) on \(\R^{m|n}\)]\label{ex:ActionTranslation}
	Let \((X^B)=(x^b, \eta^\beta)\) be coordinates on \(\R^{r|s}\).
	The supermanifold \(\R^{r|s}\) can be equipped with an additive super Lie group structure as follows:
	\begin{align}
		\begin{split}
			m\colon \R^{r|s}\times \R^{r|s}&\to \R^{r|s} \\
			m^\#x^b &= x^b + \overline{x}^b \\
			m^\#\eta^\beta &= \eta^\beta + \overline{\eta}^\beta  \\
		\end{split}&&
		\begin{split}
			i\colon \R^{r|s}&\to \R^{r|s}\\
			i^\#x^b &= -x^b \\
			i^\#\eta^\beta &= -\eta^\beta \\
		\end{split}&&
		\begin{split}
			e\colon \R^{0|0}&\to \R^{r|s} \\
			e^\#x^b &= 0 \\
			e^\#\eta^\beta &= 0 \\
		\end{split}
	\end{align}
	Here \((\overline{x}^b, \overline{\eta}^\beta)\) denote the coordinates on the second factor \(\R^{r|s}\).
	If we denote the \(C\)-points of \(\R^{r|s}\) by \((p^a, p^\alpha)\in{\left(\cO_C\right)}_0^r\oplus{\left(\cO_C\right)}_1^s\) we see that the additive group structure on \(\R^{r|s}\) corresponds precisely to the addition of its \(C\)-points in the vector space \(\underline{\R^{r|s}}(C)\).

	Let now \(L\in\Mat_{\cO_B}(r|s\times m|n)\) be an even matrix with entries in \(\cO_B\).
	With respect to the standard basis we express \(L\) as the block matrix
	\begin{equation}
		L = \left(\tensor[_C]{L}{^B}\right) =
		\begin{pmatrix}
			\tensor[_c]{L}{^b} & \tensor[_c]{L}{^\beta} \\
			\tensor[_\gamma]{L}{^b} & \tensor[_\gamma]{L}{^\beta} \\
		\end{pmatrix},
	\end{equation}
	where the entries \(\tensor[_c]{L}{^b}\) and \(\tensor[_\gamma]{L}{^\beta}\) are even entries and \(\tensor[_c]{L}{^\beta}\) and \(\tensor[_\gamma]{L}{^b}\) are odd entries.
	The supermanifold \(\R^{m|n}\times B\) over \(B\) can be equipped with an action \(a_L\colon \left(\R^{r|s}\times B\right)\times_B\left(\R^{m|n}\times B\right)\to \R^{m|n}\times B\) parametrized by \(B\) of the additive super Lie group \(\R^{r|s}\times B\) over \(B\).
	In coordinates \((y^b, \theta^\beta)\) on \(\R^{m|n}\) the action \(a_L\) is given by.
	\begin{equation}
		\begin{split}
			a_L^\# y^b &= y^b + X^R\tensor[_R]{L}{^b}
			= y^b + x^c\tensor[_c]{L}{^b} + \eta^\beta\tensor[_\gamma]{L}{^b}\\
			a_L^\#\theta^\beta &= \theta^\beta + X^R\tensor[_R]{L}{^\beta}
			= \theta^\beta + x^c\tensor[_c]{L}{^\beta} + \eta^\beta\tensor[_\gamma]{L}{^\beta}\\
		\end{split}
	\end{equation}
	This action corresponds to the translation of the \(C\)-point \((p^a, p^\alpha)\in\underline{\R^{m|n}}(C)\) by the vector \(Q\cdot L\) for the \(C\)-point \(Q=(q^b, q^\beta)\in\underline{\R^{r|s}}(C)\).
\end{ex}
\begin{ex}[Action of \(\Z_2\) on \(\R^{m|n}\)]\label{ex:Z2ActionsOnRmn}
	We define two different actions of \(\Z_2\) on \(\R^{m|n}\) on the level of \(C\)-points by
	\begin{align}
		\begin{split}
			a_1\colon \Z_2\times \underline{\R^{m|n}}(C) &\to\underline{\R^{m|n}}(C) \\
			(z, (p^a, p^\alpha)) &\mapsto \left({\left(-1\right)}^z p^a, p^\alpha\right)\\
		\end{split} &&
		\begin{split}
			a_2\colon \Z_2\times \underline{\R^{m|n}}(C) &\to\underline{\R^{m|n}}(C) \\
			(z, (p^a, p^\alpha)) &\mapsto \left(p^a, {\left(-1\right)}^z p^\alpha\right)\\
		\end{split}
	\end{align}
	We call the first action \(a_1\) reflection of the even directions at the origin and the second action \(a_2\) reflection of the odd directions at the origin.
	Note that the reflection of the odd directions at the origin can be defined for any supermanifold over \(\R^{0|0}\).
\end{ex}

In order to formulate the first quotient theorem, we need to give the supergeometric generalization of certain properties of actions of super Lie groups:
Let \(a\colon G\times M\to M\) be an action of the super Lie group \(G\) on the supermanifold \(M\).
\begin{itemize}
	\item
		We say that \(a\) is a \emph{proper} group action if the reduction \(\underline{a}(\R^{0|0}_B)\colon \underline{G}(\R^{0|0}_B)\times \underline{M}(\R^{0|0}_B)\to \underline{M}(\R^{0|0})\) is proper.
	\item
		The action \(a\) is \emph{free} if for all \(p\in\underline{M}(C_B)\) the equality \(\underline{a}(C_B)(g, p)=p\) implies \(g=e\).
	\item
		A map \(f\colon M\to N\) between supermanifolds is called \emph{\(G\)-invariant}, if \(f\circ a = f\circ p_M\).
	\item
		Let \(\gamma\colon G\to H\) a homomorphism of super Lie groups, that is \(m_H\circ \gamma\times \gamma = \gamma \circ m_G\) and suppose the super Lie group \(H\) acts on a supermanifold \(N\) via \(a_N\colon H\times N\to N\).
		The map \(f\colon M\to N\) is called \emph{\(\gamma\)-equivariant} if \(a_N\circ \gamma\times f = f\circ a\).
		In the case \(\gamma=\id_G\), we say that \(f\) is \(G\)-equivariant.
\end{itemize}

The following Proposition~\ref{prop:QuotientByFreeProperGroupAction} shows that the quotient by a free and proper action of a super Lie group on a supermanifold is again a supermanifold.
It has been proved in~\cite[Theorem~2.12]{AH-IBIHS} but we will nevertheless give a proof in order to introduce notation and concepts used in Section~\ref{sec:Superorbifolds}.
\begin{prop}\label{prop:QuotientByFreeProperGroupAction}
	Let \(G\) act on \(M\) freely and properly.
	Then there exists supermanifold~\(\faktor{M}{G}\) and a \(G\)-invariant submersion \(q\colon M\to \faktor{M}{G}\) such that any \(G\)-invariant map \(f\colon M\to N\) factors over \(q\).
	That is, there exists a map \(\overline{f}\colon \faktor{M}{G}\to N\) such that \(f=\overline{f}\circ q\).
	The supermanifold~\(\faktor{M}{G}\) and the map~\(q\) are unique up to superdiffeomorphism.
\end{prop}
\begin{proof}
	We define a map \(j\colon R\to M\times M\), where \(R=G\times M\) by setting \(j=(p_M, a)\).
	The proof proceeds by showing that \(j\) defines a regular equivalence and that the quotient~\(\faktor{M}{R}\) satisfies the universal property.

	To see that \(j\) gives an equivalence relation notice first that \(j\) is indeed the embedding of a subsupermanifold because the action is free.
	By the definition of \(j\) the first projection \(\overline{p}_1=p_1\circ j\colon R\to M\) coincides with the projection \(p_M\colon G\times M\to M\) on the first factors and \(\overline{p}_2=p_2\circ j\colon R\to M\) coincides with the group action \(a\colon G\times M\to M\).
	Then \(\delta=(e, \id_M)\colon M\to R=G\times M\) satisfies \(j\circ \delta=\Delta\) and hence condition~\ref{item:defn:EquivalenceRelation:Diagonal} of Definition~\ref{defn:EquivalenceRelation}.

	Condition~\ref{item:defn:EquivalenceRelation:Associativity} reduces to the associativity of the action as follows:
	One checks that the fiber product \(R\times_{\overline{p}_2, \overline{p}_1} R\) is isomorphic to \(G\times G\times M\), where the projections are given by
	\begin{align}
		\pi_1= t_{23}\colon G\times G\times M &\to R=G\times M, &
		\pi_2= \id_G\times a\colon G\times G\times M &\to R=G\times M.
	\end{align}
	where \(t_{23}\colon G\times G\times M\to G\times M\) is the projection on second and third factor.
	With this characterization of \(R\times_{\overline{p}_2, \overline{p}_1} R\) we define
	\begin{equation}
		c=m\times \id_M\colon R\times_{\overline{p}_2, \overline{p}_1} R = G\times G\times M\to G\times M = R
	\end{equation}
	and verify
	\begin{align}
		\overline{p}_1\circ c &= \overline{p}_1 \circ \pi_1, \\
		\overline{p}_2\circ c &= a\circ \left(m\times \id_M\right) = a\circ \left(\id_G\times a\right) = \overline{p}_2 \circ \pi_2.
	\end{align}
	While the first equation is straightforward, the second uses the compatibility of the action \(a\) with the multiplication \(m\) of the super Lie group \(G\).
	Both equations together show the left-hand diagram in~\ref{item:defn:EquivalenceRelation:Associativity} of Definition~\ref{defn:EquivalenceRelation}.
	The diagram on the right-hand side of condition~\ref{item:defn:EquivalenceRelation:Associativity} can analogously be reduced to associativity properties of the action using the identification \(R\times_{\overline{p}_2,\overline{p}_1} R\times_{\overline{p}_2,\overline{p}_1} R = G\times G\times G\times M\).

	The inverse map \(i\colon G\to G\) gives a map \(\sigma\colon R\to R\) by setting \(\sigma=(i\circ p_G, a)\) which satisfies \(\sigma\circ \sigma=\id_R\) and \(\overline{p}_1\circ \sigma = \overline{p}_2\).
	Consequently, \(\sigma\) fulfills the symmetry condition~\ref{item:defn:EquivalenceRelation:Symmetry} of Definition~\ref{defn:EquivalenceRelation}.

	The properness of the group action implies that \(j\) gives a closed subsupermanifold and the first projection \(\overline{p}_1\) is obviously a submersion.
	Consequently, the equivalence relation~\(j\colon R\to M\times M\) is a regular equivalence relation and the quotient \(\faktor{M}{R}\) exists by Proposition~\ref{prop:QuotientByRegularEquivalenceRelation}.

	By definition of the quotient with respect to an equivalence relation, see Definition~\ref{defn:QuotientEquivalenceRelation}, the quotient map \(q\colon M\to \faktor{M}{R}\) is a submersion satisfying
	\begin{equation}
		q\circ p_M = q\circ\overline{p}_1 = q\circ\overline{p}_2 = q \circ a,
	\end{equation}
	that is, \(q\) is \(G\)-invariant.
	It remains to show that the quotient \(q\colon M\to \faktor{M}{R}\) fulfills the universal property of the quotient by the group action:
	If \(f\colon M\to N\) is a \(G\)-invariant map, by definition it satisfies
	\begin{equation}
		f\circ \overline{p}_1= f\circ p_M = f\circ a = f\circ \overline{p}_2.
	\end{equation}
	Consequently, there exists a map \(\overline{f}\colon \faktor{M}{R}\to N\) such that \(f=\overline{f}\circ q\).
\end{proof}
\begin{rem}
	The dimension formula~\eqref{eq:DimensionQuotientEquivalenceRelation} for the quotient \(\faktor{M}{R}\) implies that for both even and odd dimension
	\begin{equation}
		\dim\faktor{M}{G}
		= 2\dim M - \left(\dim M + \dim G\right)
		= \dim M - \dim G.
	\end{equation}
\end{rem}

In the previous Proposition~\ref{prop:QuotientByFreeProperGroupAction} the condition that the action is free is strictly necessary.
Indeed, as can be seen in the Example~\ref{ex:Z2ActionsOnRmn}, the quotient of a supermanifold by a non-free action cannot be a supermanifold:
The quotient of the reduced action of \(a_1\) of Example~\ref{ex:Z2ActionsOnRmn} would be an orbifold instead of a manifold and in the case of the action~\(a_2\) a similar phenomenon appears for the odd directions.

The first step to understand non-free group actions is to define the stabilizer or isotropy group.
The isotropy group \(H_p\) of a \(B\)-point \(p\in\underline{M}(\id_B)\) is given by the functor
\begin{equation}
	\begin{split}
		\underline{H_p}\colon \cat{SPoint}_B^{op}&\to \cat{Set} \\
		C_B&\mapsto \Set{g\in\underline{G}(C_B)\given \underline{a}(C_B)(g, p)=p}
	\end{split}
\end{equation}
In the case of \(B=\R^{0|0}\), that is of an action of a super Lie group over \(\R^{0|0}\) on a supermanifold over \(\R^{0|0}\), it is proved in~\cite{BCC-SGS} that \(H_p\) is a super Lie subgroup of~\(G\).
If \(B\) has odd directions, \(H_p\) is not always a super Lie group, as the following example shows:
\begin{ex}\label{ex:ActionTranslationNoLieSubGroup}
	Let \(\lambda\in{\left(\cO_B\right)}_1\) be an odd element and define the action \(a_\lambda\colon \left(\R^{0|1}\times B\right)\times_B\left(\R^{1|0}\times B\right)\to\R^{1|0}\) as in Example~\ref{ex:ActionTranslation} by
	\begin{equation}
		a_\lambda^\# y = y + \eta\lambda.
	\end{equation}
	The isotropy group \(H_0\) at \(p=0\) is given by
	\begin{equation}
		\begin{split}
			\underline{H_0}(C_B)
			&= \Set{g\in\underline{\left(\R^{0|1}\times B\right)}(C_B)\given \underline{a_\lambda}(C_B)(g, 0) = g\cdot\left(C_B^*\lambda\right)= 0} \\
			&= \Set{g\in {\left(\cO_C\right)}_1\given g\cdot\left(C_B^*\lambda\right)=0}.
		\end{split}
	\end{equation}
	That is the \(C_B\)-points of \(H_0\) are those odd elements of \({\left(\cO_C\right)}_1\) that annihilate \(C_B^*\lambda\).
	While \(H_0\) is a subgroup of the group \(\R^{0|1}\times B\) it is not a subsupermanifold of \(\R^{0|1}\times B\).
\end{ex}

In order to prove that \(H_p\) is a super Lie subgroup one uses the constant rank theorem.
For a discussion of the constant rank theorem in supergeometry, we refer to~\cite[Chapter~5.2]{CCF-MFS}.
The constant rank theorem is applied to \(a_p\colon G\to M\), the map obtained from \(a\) by fixing the second factor to \(p\).
Its tangent at the identity gives a map  \(\underline{\tangent[e]{a_p}}\colon \underline{\LieAlg{g}}\to\underline{\VSec[p]{\tangent{M}}}\).
The linear map \(\underline{\tangent[e]{a_p}}\) is said to have rank \(r|s\) if we can pick a basis of \(\underline{\LieAlg{g}}(\id_B)\) and a basis of \(\underline{\VSec[p]{\tangent{M}}}(\id_B)\) such that with respect to this basis \(\underline{\tangent[e]{a_p}}(\id_B)\) is given by
\begin{equation}
	\begin{pmatrix}
		\id_r & 0 & 0 & 0 \\
		0 & 0 & 0 & 0 \\
		0 & 0 & \id_s & 0\\
		0 & 0 & 0 & 0\\
	\end{pmatrix}.
\end{equation}
Not every matrix with entries in \(\cO_B\) has a rank.
For example the differential of the restriction of \(a_\lambda\) from Example~\ref{ex:ActionTranslationNoLieSubGroup} to zero is given by multiplication with the odd \(\lambda\).
As \(\lambda\) is not invertible, the differential cannot be brought in standard form.
However, with the assumption of constant rank, the proof of~\cite[Proposition~5.2]{BCC-SGS} generalizes to families over \(B\) and yields:
\begin{lemma}\label{lemma:IsotropyGroupIsSuperLieGroup}
	Let \(a\colon G\times_B M\to M\) be the action of a super Lie group \(G\) over \(B\) on the supermanifold \(M\) over \(B\) and \(p\in \underline{M}(B)\) a \(B\)-point.
	The following are equivalent
	\begin{itemize}
		\item
			\(H_p\) is a super Lie subgroup of \(G\).
		\item
			The map \(\underline{\tangent[e]{a_p}}\colon \underline{\LieAlg{g}}\to\underline{\VSec[p]{\tangent{M}}}\) has a rank.
		\item
			The kernel of \(\underline{\tangent[e]{a_p}}(\id_B)\) is a free submodule of \(\underline{\LieAlg{g}}(\id_B)\).
		\item
			The image of \(\underline{\tangent[e]{a_p}}(\id_B)\) is a free submodule of \(\underline{\VSec[p]{\tangent{M}}}(\id_B)\).
	\end{itemize}
\end{lemma}

\begin{rem}
	Note that if \(B=\R^{0|0}\) the differential~\(\underline{\tangent[e]{a_p}}(\id_{\R^{0|0}})\) is a matrix with entries in \(\R\) and consequently its rank is always defined.
	This is the reason that the proof that the isotropy group is a super Lie group in~\cite{BCC-SGS} does not require the rank hypothesis.

	We want to reiterate that every statement in this article is functorial under base change.
	At the example of Lemma~\ref{lemma:IsotropyGroupIsSuperLieGroup} this yields:
	Suppose that the isotropy group \(H_p\) of the action \(a\colon G\times_B M\to M\) of a super Lie group \(G\) over \(B\) on a supermanifold \(M\) over \(B\) at the \(B\)-point \(p\colon B\to M\) is a super Lie group over \(B\).
	For any map \(b\colon B'\to B\), the action \(b^*a\colon b^*G\times_{B'} b^*M\to b^*M\) is an action of the super Lie group \(b^*G\) over \(B'\) on the supermanifold \(b^*M\) over \(B'\) and \(b^*p\colon B'\to b^*M\) is a \(B'\)-point of \(b^*M\).
	Then also the isotropy group \(H_{b^*p}\) of \(b^*a\) at \(b^*p\) is a super Lie group (over \(B'\)).
	Most of the time we do not write out this functoriality and drop the explicit \(b^*\) from the notation.
	Instead we assume that \(B\) can be changed as necessary.
\end{rem}

The next step in the study of non-free actions is to investigate the structure of a supermanifold with an action of a super Lie group around a \(B\)-point \(p\) such that the isotropy group \(H_p\) is a super Lie group.
We will show that one can find a local slice for the group action around every \(B\)-point, that is, a submanifold transversal to the orbits and invariant under the action of the isotropy group.
The corresponding classical results go back to~\cites{K-SCGTL}{P-ESANCLG}.
For a modern textbook see~\cite[Chapter~3]{AB-LGGAIA}.

\begin{defn}\label{defn:Slice}
	Let \(p\colon B\to M\) be a \(B\)-point of the supermanifold \(M\) with \(G\)-action~\(a\) and denote by \(H_p\) the isotropy group at \(p\).
	A slice at \(p\) is an embedded submanifold \(S\) in \(M\) such that
	\begin{enumerate}
		\item\label{item:DefnSlice:pinS}
			\(p\in\underline{S}(B)\).
		\item\label{item:DefnSlice:Transversal}
			\(\underline{\VSec[p]{\tangent{M}}}=\left(\underline{\tangent[e]{a_p}}\right)\underline{\LieAlg{g}}\oplus\underline{\VSec[p]{\tangent{S}}}\)
		\item\label{item:DefnSlice:InvarianceHp}
			\(S\) is invariant under \(H_p\).
			That is for all \(C_B\colon C\to B\), \(q\in\underline{S}(C_B)\) and \(g\in\underline{H_p}(C_B)\) then \(\underline{a}(C_B)(g, q)\in\underline{H_p}(C_B)\).
		\item\label{item:DefnSlice:SingleOrbits}
			If \(q\in\underline{S}(C_B)\) and \(g\in\underline{G}\) such that \(\underline{a}(C_B)(g, q)\in \underline{S}(C_B)\) then \(g\in\underline{H_p}(C_B)\).
	\end{enumerate}
\end{defn}

Notice that the condition~\ref{item:DefnSlice:Transversal} requires that \(\left(\underline{\tangent[e]{a_p}}\right)\underline{\LieAlg{g}}\) has a complement that is free and hence needs to be free itself.
It follows that the existence of a slice at \(p\) requires that \(H_p\) is a super Lie group.
The even and odd dimension of the slice \(S\) at \(p\) can be calculated as
\begin{equation}
	\dim S = \dim M - \dim G + \dim H_p.
\end{equation}

\ExistenceOfSlice{}
\begin{proof}
	We generalize the proof of~\cite[Theorem~3.49]{AB-LGGAIA}.
	Let \(\underline{N}\) be the \(m\)-orthogonal complement of \(\left(\underline{\tangent[e]{a_p}}\right)\underline{\LieAlg{g}}\) in \(\underline{\VSec[p]{\tangent{M}}}\).
	Pick an open subfunctor \(\left.\underline{N}\right|_U\) of~\(\underline{N}\) and define the submanifold
	\begin{equation}
		\underline{S} = \exp_p \underline{N}|_U \subset \underline{M}
	\end{equation}
	as a candidate for the slice.
	Here the open subset \(U\subset\underline{N}(\R^{0|0})\) is chosen such that \(\exp_p\) is defined on \(\underline{N}|_U\) but might need to be restricted further in the remainder of the argument.

	By definition the submanifold \(S\subset M\) satisfies the conditions~\ref{item:DefnSlice:pinS} and~\ref{item:DefnSlice:Transversal} of Definition~\ref{defn:Slice}.
	To see that \(S\) satisfies also condition~\ref{item:DefnSlice:InvarianceHp} remark that \(H_p\) acts via isometries on \(M\) and isometries map geodesics to geodesics.

	It remains to prove that \(S\) satisfies condition~\ref{item:DefnSlice:SingleOrbits} when \(U\) is chosen small enough.
	We will argue by contradiction and use the properness of the action.
	Suppose that there is no \(U\subset\underline{N}(\R^{0|0})\) such that \(\underline{S}=\exp_p \underline{N}|_U\) satisfies condition~\ref{item:DefnSlice:SingleOrbits}.
	Then for every \(C_B\colon C\to B\) we can find sequences \(p_n\in\underline{S}(C_B)\) and \(g_n\in\underline{G}(C_B)\) such that
	\begin{itemize}
		\item
			\(g_n\not\in\underline{H_p}(C_B)\)
		\item
			\(\underline{a}(C_B)(g_n, p_n)\in\underline{S}(C_B)\)
		\item
			the reduced points \(\overline{g}_n=r^*g_n\in\underline{G}(\R^{0|0}_B)\) and \(\overline{p}_n=r^*x_n\in\underline{S}(\R^{0|0}_B)\) where \(r\colon \R^{0|0}_B\to C_B\) converge to
			\begin{align}
				\lim_{n\to\infty} \overline{p}_n &=  \overline{p} &
				\lim_{n\to\infty} \underline{a}(\R^{0|0}_B)(\overline{g}_n, \overline{p}_n) = \overline{p}
			\end{align}
	\end{itemize}
	Now as \(a\) is a proper group action, \(\overline{g}_n\) converges to an element \(h\in\underline{H}(\R^{0|0}_B)\).
	By multiplying the sequence \(g_n\) by \(h^{-1}\) we can assume that \(\overline{g}_n\) converges to \(\underline{e}(\R^{0|0}_B)\).
	By condition~\ref{item:DefnSlice:Transversal} the restriction of the action \(a\colon G\times M\to M\) to \(\tilde{a}\colon G\times S\to M\) is submersive at \((e, p)\).
	Consequently, up to shrinking \(U\) there is a subsupermanifold \(V\subset G\) through \(e\) and an open subsupermanifold \(M_p\subset M\) containing \(p\) such that \(a\colon V\times S\to M_0\) is a diffeomorphism and \(\tangent[e]{G}=\LieAlg{h}\oplus \tangent[e]{V}\).
	Here \(\LieAlg{h}\) is the Lie algebra of \(H_p\).
	Furthermore, again by implicit function theorem, the multiplication in \(G\) induces a local diffeomorphism between \(V\times H_p\) and \(G\).
	Consequently, for \(n\) sufficiently large, every \(g_n\) can be written uniquely as \(v_n\cdot h_n\) for some \(h_n\in\underline{H_p}(C_B)\) and \(v_n\in\underline{V}(C_B)\) and \(v_n\neq e\).
	It follows that \(\underline{a}(C_B)(h_n, p_n)\in \underline{S}(C_B)\) and hence
	\begin{equation}
		\underline{a}(C_B)(v_n, \underline{a}(C_B)(h_n, p_n)) \not\in \underline{S}(C_B)
	\end{equation}
	because of the diffeomorphism \(V\times S\to M_0\).
	But this contradicts \(\underline{a}(C_B)(g_n, p_n)\in\underline{S}(C_B)\) and it follows that~\ref{item:DefnSlice:SingleOrbits} must be satisfied.
\end{proof}

\begin{cor}
	Let \(M\) be a supermanifold over \(B\) with an even number of odd dimensions, \(a\colon G\times_{B} M\to M\) a proper group action and \(p\colon B\to M\) a \(B\)-point of \(M\) such that \(H_p\) is a finite group.
	Then there exists a slice at \(p\).
\end{cor}
\begin{proof}
	Any finite isotropy group \(H_p\) is a super Lie subgroup of \(G\) and we can construct an \(H_p\)-invariant metric by averaging over \(H_p\) as follows:
	Choose an arbitrary Riemannian metric \(\tilde{m}\) on \(M\) and for \(h\in H_p\) denote by \(\tilde{m}_h\) the metric obtained by pullback along the diffeomorphism induced by \(h\).
	Then
	\begin{equation}
		m = \sum_{h\in H_p} \tilde{m}_h
	\end{equation}
	is an \(H_p\)-invariant metric.
	Hence, we can apply Theorem~\ref{thm:ExistenceOfSlice}.
\end{proof}

For any closed Lie subgroup \(H\) in the super Lie group \(G\) the quotient \(\faktor{G}{H}\) with respect to the right action of \(H\) on \(G\) exists.
Moreover, it is proved in~\cite[Proposition~3.6]{AH-IBIHS} that the quotient map \(q\colon G\to\faktor{G}{H}\) is a principal bundle with typical fiber \(H\).
For further information on principal bundles on supermanifolds we also refer to~\cite[Chapter~6]{EK-SGSRSSCA}.
The slice constructed in Theorem~\ref{thm:ExistenceOfSlice} carries an \(H_p\) action.
Consequently, one can form the associated bundle \(G\times_{H_p} S\) by replacing the fibers of the \(H_p\)-principal bundle \(G\to \faktor{G}{H_p}\) by \(S\), see~\cite[Section~6.2]{EK-SGSRSSCA}.
Alternatively, in~\cite[Section~3.2]{AH-IBIHS}, the associated bundle is described as
\begin{equation}
	G\times_{H_p} S = \faktor{\left(G\times S\right)}{H_p}
\end{equation}
where \(H_p\) acts freely on \(G\times S\) such that \(h\in H_p\) maps on \(C_B\)-points as follows:
\begin{equation}\label{eq:AssociatedBundleAsQuotient}
	\begin{split}
		h\colon \underline{G}(C_B)\times \underline{S}(C_B) &\to \underline{G}(C_B)\times \underline{S}(C_B) \\
		(g, s) &\mapsto (gh^{-1}, hs)
	\end{split}
\end{equation}
With this characterization we can prove the following supergeometric analogue of the tubular neighborhood theorem:

\begin{prop}\label{prop:SliceFoliation}
	Let \(p\colon B\to M\) be a \(B\)-point of a supermanifold \(M\) over \(B\) with group action \(a\colon G\times_B M\to M\) such that the isotropy group \(H_p\) is a super Lie subgroup and there is a slice \(S\) around \(p\).
	Then there exists an open neighborhood \(T\) of \(p\) in \(M\) and a \(G\)-equivariant diffeomorphism \(G\times_{H_p} S\to T\), where \(S\) is the slice at \(p\).
\end{prop}
\begin{proof}
	Let \(a|_S\colon G\times S\to M\) be the restriction of the action \(a\colon G\times M\to M\) to \(S\).
	The map~\(a|_S\) is \(G\)-equivariant with respect to the \(G\)-action on \(G\times S\) consisting of multiplication from the left and the \(H_p\)-invariant with respect to the action described in~\eqref{eq:AssociatedBundleAsQuotient}.
	Consequently, there is a well defined \(G\)-equivariant map \(\tilde{a}\colon G\times_{H_p} S\to M\).
	By the classical tubular neighborhood theorem, see~\cite[Theorem~3.57]{AB-LGGAIA}, its reduction is a diffeomorphism onto its image \(T\subset \underline{M}(\R^{0|0})\).

	It remains to show that \(\tilde{a}\colon G\times_{H_p} S\to M|_T\) is a superdiffeomorphism.
	We will do so by showing that \(\differential{\tilde{a}}\) is bijective everywhere.
	But as the dimension of \(G\times_{H_p} S\) and \(M\) agree, it is sufficient to show surjectivity of \(\differential{\tilde{a}}\) which in turn follows from surjectivity of \(\differential{a}\) because the map \(G\times S\to G\times_{H_p} S\) is submersive.
	We have already seen in the proof of Theorem~\ref{thm:ExistenceOfSlice} that \(\differential{a}\) is surjective at \((e, s)\) for all \(s\in\underline{S}(B)\).
	As \(a(g, s)=a(g\cdot e, s)\) and left-multiplication by \(g\) is a diffeomorphism of \(G\), the result follows.
\end{proof}

\subsection{Superorbifolds}\label{sec:Superorbifolds}
In this section we introduce superorbifolds as a generalization of supermanifolds such that quotients of supermanifolds by proper group actions with finite isotropy groups can be defined.
The concept of orbifold goes back to~\cite{S-GNM} where orbifolds are defined as topological spaces that are locally isomorphic to quotients of \(\R^m\) by finite group actions.
As topological spaces are not sufficient to describe supermanifolds we will use here a reformulation in terms of equivalence classes of groupoids, see~\cite{M-OAGI}.

The following definition is adapted from~\cite{NLAB-InternalCategory}, but super Lie groupoids have appeared before in~\cite{AHW-SO}.
\begin{defn}\label{defn:SuperLieGroupoid}
	A super Lie groupoid \(\groupoid{G}\) consists of two supermanifolds \(\groupoid{G}_0\) and \(\groupoid{G}_1\) together with
	\begin{itemize}
		\item
			a submersive source morphism \(s\colon \groupoid{G}_1\to \groupoid{G}_0\),
		\item
			a submersive target morphism \(t\colon \groupoid{G}_1\to \groupoid{G}_0\),
		\item
			\(e\colon \groupoid{G}_0\to \groupoid{G}_1\) called identity assigning morphism, such that source and target of the identity are given by the identity map:
			\begin{equation}
				\begin{tikzpicture}[commutative diagrams/.cd, every diagram, baseline=(current bounding box.center)]
					\matrix[dr](m){
						& \groupoid{G}_1 &\\
						\groupoid{G}_0 && \groupoid{G}_0\\
					} ;
					\path[pf]{
						(m-2-1) edge node[auto]{\(e\)} (m-1-2)
						edge node[auto]{\(\id_{\groupoid{G}_0}\)} (m-2-3)
						(m-1-2) edge node[auto]{\(s\)}(m-2-3)
					};
				\end{tikzpicture}
				\hspace{5em}
				\begin{tikzpicture}[commutative diagrams/.cd, every diagram, baseline=(current bounding box.center)]
					\matrix[dr](m){
						& \groupoid{G}_1 &\\
						\groupoid{G}_0 && \groupoid{G}_0\\
					} ;
					\path[pf]{
						(m-2-1) edge node[auto]{\(e\)} (m-1-2)
						edge node[auto]{\(\id_{\groupoid{G}_0}\)} (m-2-3)
						(m-1-2) edge node[auto]{\(t\)}(m-2-3)
					};
				\end{tikzpicture}
			\end{equation}
		\item
			a composition morphism \(c\colon \groupoid{G}_1\times_{t,s} \groupoid{G}_1\to \groupoid{G}_1\) that satisfies
			\begin{itemize}
				\item
					that the source and target of the composite are given by the source of the first factor and the target of the second:
					\begin{equation}
						\begin{tikzpicture}[commutative diagrams/.cd, every diagram, baseline=(current bounding box.center)]
							\matrix[mat](m){
								\groupoid{G}_1\times_{t,s} \groupoid{G}_1 & \groupoid{G}_1 \\
								\groupoid{G}_1 & \groupoid{G}_0\\
							} ;
							\path[pf]{
								(m-1-1) edge node[auto]{\(c\)} (m-1-2)
								edge node[auto]{\(p_1\)} (m-2-1)
								(m-1-2) edge node[auto]{\(s\)}(m-2-2)
								(m-2-1) edge node[auto]{\(s\)}(m-2-2)
							};
						\end{tikzpicture}
						\hspace{5em}
						\begin{tikzpicture}[commutative diagrams/.cd, every diagram, baseline=(current bounding box.center)]
							\matrix[mat](m){
								\groupoid{G}_1\times_{t,s} \groupoid{G}_1 & \groupoid{G}_1 \\
								\groupoid{G}_1 & \groupoid{G}_0\\
							} ;
							\path[pf]{
								(m-1-1) edge node[auto]{\(c\)} (m-1-2)
								edge node[auto]{\(p_2\)} (m-2-1)
								(m-1-2) edge node[auto]{\(t\)}(m-2-2)
								(m-2-1) edge node[auto]{\(t\)}(m-2-2)
							};
						\end{tikzpicture}
					\end{equation}
				\item
					an associativity law for iterated composition:
					\begin{diag}
						\matrix[mat](m){
							\groupoid{G}_1\times_{t,s} \groupoid{G}_1 \times_{t,s} \groupoid{G}_1 & \groupoid{G}_1\times_{t,s} \groupoid{G}_1 \\
							\groupoid{G}_1\times_{t,s} \groupoid{G}_1 & \groupoid{G}_1\\
						} ;
						\path[pf]{
							(m-1-1) edge node[auto]{\(c\times_{t,s}\id_{\groupoid{G}_1}\)} (m-1-2)
							edge node[auto]{\(\id_{\groupoid{G}_1}\times_{t,s} c\)} (m-2-1)
							(m-1-2) edge node[auto]{\(c\)}(m-2-2)
							(m-2-1) edge node[auto]{\(c\)}(m-2-2)
						};
					\end{diag}
				\item
					unit laws for composition with the identity from the left and the right:
					\begin{diag}
						\matrix[mat,column sep=huge](m){
							\groupoid{G}_0\times_{\id_{\groupoid{G}_0},s} \groupoid{G}_1 & \groupoid{G}_1\times_{t,s} \groupoid{G}_1 & \groupoid{G}_1\times_{t,\id_{\groupoid{G}_0}} \groupoid{G}_0 \\
							& \groupoid{G}_1 &\\
						} ;
						\path[pf]{
							(m-1-1) edge node[auto]{\(e\times_{\id_{\groupoid{G}_0},s}\id_{\groupoid{G}_1}\)} (m-1-2)
								edge node[auto, swap]{\(p_2\)} (m-2-2)
							(m-1-2) edge node[auto]{\(c\)}(m-2-2)
							(m-1-3) edge node[auto, swap]{\(\id_{\groupoid{G}_1}\times_{t,\id_{\groupoid{G}_0}} e\)} (m-1-2)
								edge node[auto]{\(p_1\)} (m-2-2)
						};
					\end{diag}
			\end{itemize}
		\item
			an inverse morphism \(i\colon \groupoid{G}_1\to \groupoid{G}_1\) that
			\begin{itemize}
				\item
					flips the source and target morphisms:
					\begin{equation}
						\begin{tikzpicture}[commutative diagrams/.cd, every diagram, baseline=(current bounding box.center)]
							\matrix[dr](m){
								\groupoid{G}_1 && \groupoid{G}_1\\
								& \groupoid{G}_0 &\\
							} ;
							\path[pf]{
								(m-1-1) edge node[auto]{\(i\)} (m-1-3)
								edge node[auto, swap]{\(s\)} (m-2-2)
								(m-1-3) edge node[auto]{\(t\)}(m-2-2)
							};
						\end{tikzpicture}
						\hspace{5em}
						\begin{tikzpicture}[commutative diagrams/.cd, every diagram, baseline=(current bounding box.center)]
							\matrix[dr](m){
								\groupoid{G}_1 && \groupoid{G}_1\\
								& \groupoid{G}_0 &\\
							} ;
							\path[pf]{
								(m-1-1) edge node[auto]{\(i\)} (m-1-3)
								edge node[auto, swap]{\(t\)} (m-2-2)
								(m-1-3) edge node[auto]{\(s\)}(m-2-2)
							};
						\end{tikzpicture}
					\end{equation}
				\item
					and acts as a left- and right-inverse to the composition:
					\begin{equation}
						\begin{tikzpicture}[commutative diagrams/.cd, every diagram, baseline=(current bounding box.center)]
							\matrix[mat](m){
								\groupoid{G}_1 & \groupoid{G}_1 \times_{t,s} \groupoid{G}_1\\
								\groupoid{G}_0 & \groupoid{G}_1\\
							} ;
							\path[pf]{
								(m-1-1) edge node[auto]{\((i,\id_{\groupoid{G}_1})\)} (m-1-2)
									edge node[auto]{\(t\)} (m-2-1)
								(m-1-2) edge node[auto]{\(c\)}(m-2-2)
								(m-2-1) edge node[auto]{\(e\)}(m-2-2)
							};
						\end{tikzpicture}
						\hspace{5em}
						\begin{tikzpicture}[commutative diagrams/.cd, every diagram, baseline=(current bounding box.center)]
							\matrix[mat](m){
								\groupoid{G}_1 & \groupoid{G}_1 \times_{t,s} \groupoid{G}_1\\
								\groupoid{G}_0 & \groupoid{G}_1\\
							} ;
							\path[pf]{
								(m-1-1) edge node[auto]{\((\id_{\groupoid{G}_1},i)\)} (m-1-2)
									edge node[auto]{\(s\)} (m-2-1)
								(m-1-2) edge node[auto]{\(c\)}(m-2-2)
								(m-2-1) edge node[auto]{\(e\)}(m-2-2)
							};
						\end{tikzpicture}
					\end{equation}
			\end{itemize}
	\end{itemize}
\end{defn}
\begin{ex}[translation groupoid]
	The prototypical example of a super Lie groupoid is the translation groupoid.
	Assume that the super Lie group \(G\) acts on the supermanifold~\(M\) via \(a\colon G\times M\to M\).
	Then we define the super Lie groupoid \(G\ltimes M\) by setting \({\left(G\ltimes M\right)}_0 = M\) and \({\left(G\ltimes M\right)}_1= G\times M\) where the source map \(s=p_M\colon G\times M\to M\) coincides with the projection on \(M\) and the target map coincides with the action \(t=a\colon G\times M\to M\).
	Both source and target maps are submersions.

	If we denote the identity and inverse of the super Lie group as
	\begin{align}
		e_G\colon B&\to G &
		i_G\colon G\to G
	\end{align}
	we define the identity assigning morphism and inversem morphism of the groupoid as
	\begin{align}
		e=(e_G, \id_M) \colon M&\to G\times M &
		i=(i_G\circ p_G, a) \colon G\times M\to G\times M
	\end{align}
	where \(p_G\colon G\times M\to G\) is the projection on \(G\).

In order to define the composition morphism, we identify \({\left(G\ltimes M\right)}_1\times_{t,s} {\left(G\ltimes M\right)}_1\) with \(G\times G\times M\) where the projections are given by \(p_1=t_{23}\) and \(p_2=\id_G\times a\), like in the proof of Proposition~\ref{prop:QuotientByFreeProperGroupAction}.
	The composition map is then given in terms of the multiplication \(m_G\colon G\times G\to G\) by
	\begin{equation}
		c= m\times\id_M\colon G\times G\times M\to G\times M.
	\end{equation}
	The verification of the commutative diagrams of the definition of super Lie groupoid reduces to properties of the group action and has in part been done in the proof of Proposition~\ref{prop:QuotientByFreeProperGroupAction}.
\end{ex}
\begin{ex}[supermanifolds as super Lie groupoids]
	Any supermanifold gives rise to the identity groupoid \(\groupoid{G}\) over \(M\).
	Set \(\groupoid{G}_1=\groupoid{G}_0=M\) and let the source and target maps be the identity, that is, \(s=t=\id_M\).
	Also the identity morphism, composition and the inverse map are given by the identity.
	This can be seen as a special case of the translation groupoid when one considers the trivial action of the trivial group on \(M\).
\end{ex}
\begin{ex}[super Lie groups as groupoids]
	Any super Lie group \(G\) can be seen as a super Lie groupoid with \(\groupoid{G}_0=\R^{0|0}\) and \(\groupoid{G}_1=G\).
	The identity, composition and inverse of the group then yield the corresponding maps of the groupoid.
	This is another special case of the translation groupoid with \(M=\R^{0|0}\).
\end{ex}
\begin{ex}[regular equivalence relation groupoid and pair groupoid]
	Let \(R\subset M\times M\) be a regular equivalence relation as in Definition~\ref{defn:EquivalenceRelation} with projections~\(\overline{p}_1\) and \(\overline{p}_2\).
	We can define a groupoid \(\groupoid{R}\) by setting \(\groupoid{R}_0=M\) and \(\groupoid{R}_1=R\), \(s=\overline{p}_1\), \(t=\overline{p}_2\) and using the same composition map for the groupoid as for the equivalence relation.
	The symmetry condition of the equivalence relation yields the inverse map.

	A special case is the trivial equivalence relation \(R=M\times M\).
	In this case the resulting groupoid is called the pair groupoid.
\end{ex}
\begin{rem}
	The usual definition of a groupoid, see for example~\cite[Section~1.1]{M-OAGI}, defines a groupoid as a set of objects \(\groupoid{G}_0\) and a set of arrows \(\groupoid{G}_1\) together with maps as in Definition~\ref{defn:SuperLieGroupoid}.
	The arrows are thought as originating and ending on points of \(\groupoid{G}_0\) and need to be invertible.
	A Lie groupoid is then defined as a groupoid with additional manifold structure and requiring maps to be smooth.

	As supermanifolds are not sufficiently described as a set of points with additional structure, our definition of super Lie groupoid describes a groupoid internal to the category of supermanifolds.
	For details on groupoids internal to a category see~\cite{NLAB-InternalCategory}.
	We need the even stronger condition that the source and target maps are submersions so that all fiber products necessary for Definition~\ref{defn:SuperLieGroupoid} exist.
\end{rem}
\begin{defn}
	A super Lie groupoid \(\groupoid{G}\) is called proper if the reduction of \((s,t)\colon \groupoid{G}_1 \to \groupoid{G}_0\times \groupoid{G}_0\) is a proper map of topological spaces.
	A super Lie groupoid \(\groupoid{G}\) is called étale if the source and target maps \(s\), \(t\colon \groupoid{G}_1\to \groupoid{G}_0\) are local diffeomorphisms.
\end{defn}
\begin{ex}
	Let \(a\colon G\times M\to M\) be a group action.
	The translation groupoid \(G\ltimes M\) is a proper groupoid if the group action is proper.
	If \(G\) is a discrete group, the translation groupoid is étale.
	If \(G\) is a finite group, the translation groupoid is proper and étale.
\end{ex}
\begin{defn}\label{defn:LieGroupoidHomomorphism}
	A homomorphism \(F\colon \groupoid{G}\to \groupoid{H}\) between super Lie groupoids \(\groupoid{G}\) and \(\groupoid{H}\) consists of two smooth maps between supermanifolds \(F_0\colon \groupoid{G}_0\to \groupoid{H}_0\) and \(F_1\colon \groupoid{G}_1\to \groupoid{H}_1\) such that the following commutative diagram commutes:
	\begin{diag}\label{diag:HomomorphismSuperLieGroupoid}
		\matrix[mat](m){
			\groupoid{G}_1 & \groupoid{H}_1\\
			\groupoid{G}_0\times \groupoid{G}_0 & \groupoid{H}_0\times \groupoid{H}_0\\
		} ;
		\path[pf]{
			(m-1-1) edge node[auto]{\(F_1\)} (m-1-2)
				edge node[auto]{\((s,t)\)} (m-2-1)
				(m-1-2) edge node[auto]{\((s,t)\)}(m-2-2)
			(m-2-1) edge node[auto]{\(F_0\times F_0\)}(m-2-2)
		};
	\end{diag}
	The composition \(F'\circ F\) of homomorphisms of super Lie groupoids \(F\colon\groupoid{G}\to \groupoid{H}\) with \(F'\colon \groupoid{H}\to\groupoid{K}\) is defined by \({\left(F'\circ F\right)}_0=F'_0\circ F_0\) and \({\left(F'\circ F\right)}_1=F'_1\circ F_1\).

	A weak equivalence between super Lie groupoids is a homomorphism \(F\colon \groupoid{G}\to \groupoid{H}\) of super Lie groupoids such that in addition
	\begin{enumerate}
		\item\label{item:defn:LieGroupoidEquivalence:SurjectiveSubmersion}
			the map
			\begin{diag}
				\matrix[seq](m){
					\groupoid{H}_1\times_{s,F_0} \groupoid{G}_0 & \groupoid{H}_1 & \groupoid{H}_0\\
				} ;
				\path[pf]{
					(m-1-1) edge node[auto]{\(p_1\)} (m-1-2)
					(m-1-2) edge node[auto]{\(t\)}(m-1-3)
				};
			\end{diag}
			is a surjective submersion.
		\item\label{item:defn:LieGroupoidEquivalence:FiberProduct}
			the square~\eqref{diag:HomomorphismSuperLieGroupoid} is a fiber product, that is, the induced map
			\begin{equation}
				\groupoid{G}_1 \to \groupoid{H}_1\times_{(s,t),F_0\times F_0} \left(\groupoid{G}_0\times \groupoid{G}_0\right)
			\end{equation}
			is a superdiffeomorphism.
	\end{enumerate}
\end{defn}
\begin{ex}
	Let \(a_M\colon G\times M\to M\) and \(a_N\colon H\times N\to N\) be actions of the super Lie groups \(G\) and \(H\) on \(M\) and \(N\) respectively.
	A homomorphism of super Lie groups \(\rho\colon G\to H\) together with a \(\rho\)-equivariant map \(f\colon M\to N\) yield a homomorphism \(F\colon G\ltimes M\to H\ltimes N\) by setting \(F_0=0\) and \(F_1=\rho\times f\).
\end{ex}
\begin{ex}[Homomorphism between translation groupoid and equivalence relation groupoid]\label{ex:HomomorphismTranslationGroupoidEquivalenceRelationGroupoid}
	Let \(a\colon G\times M\to M\) be a proper free group action and let \(R\) be the equivalence relation constructed in the proof of Proposition~\ref{prop:QuotientByFreeProperGroupAction} by setting \(j=(p_M, a)\colon G\times M\to M\times M\).
	Then the map \(j\) induces a homomorphism \(F\colon G\ltimes M\to \groupoid{R}\) from the translation groupoid to the regular equivalence relation groupoid.
	Here \(F_0=\id_M\colon M\to M\) and \(F_1=j\colon G\times M\to R\).
\end{ex}
\begin{ex}[Weak equivalence between the regular equivalence relation groupoid and the identity groupoid over its quotient]\label{ex:EquivalenceRelationGroupoidAndQuotient}
	Let \(R\) be a regular equivalence relation on \(M\) with quotient \(q\colon M\to Q\), \(\groupoid{R}\) the equivalence relation groupoid and \(\groupoid{Q}\) the identity groupoid of \(Q\).
	The homomorphism \(F\colon\groupoid{R}\to \groupoid{Q}\) be given by \(F_0=q\colon \groupoid{R}_0=M\to Q=\groupoid{Q}_0\) and \(F_1=q\circ\overline{p}_1\colon \groupoid{R}_1=R\to Q=\groupoid{Q}_1\) is a weak equivalence.
	Indeed,
	\begin{diag}
		\matrix[mat, column sep=large](m){
			\groupoid{Q}_1\times_{s, F_0} \groupoid{R}_0 = Q\times_{\id_Q, j} M = M & \groupoid{Q}_1=Q & \groupoid{Q}_0=Q\\
		} ;
		\path[pf]{
			(m-1-1) edge node[auto]{\(p_1=q\)} (m-1-2)
			(m-1-2) edge node[auto]{\(t=\id_Q\)}(m-1-3)
		};
	\end{diag}
	is a surjective submersion, showing~\ref{item:defn:LieGroupoidEquivalence:SurjectiveSubmersion} of Definition~\ref{defn:LieGroupoidHomomorphism}.
	In order to show~\ref{item:defn:LieGroupoidEquivalence:FiberProduct}, we show that \(\groupoid{R}_1=R\) fulfills the universal property of the fiber product.
	That is, we need to show that for any two maps \(f\) and \((g_1, g_2)\) from the supermanifold \(N\) satisfying the following commutative diagram
	\begin{diag}
		\matrix[mat, column sep=large](m){
			N & & \\
			& \groupoid{R}_1=R & \groupoid{Q}_1=Q \\
			& \groupoid{R}_0\times\groupoid{R}_0=M\times M & \groupoid{Q}_0\times\groupoid{Q}_0=Q\times Q \\
		};
		\path[pf]{
			(m-1-1) edge[bend left=20] node[auto]{\(f\)} (m-2-3)
				edge[bend right=20] node[auto, swap]{\((g_1, g_2)\)} (m-3-2)
				edge[dotted] node[auto]{\((f, (g_1,g_2))\)} (m-2-2)
			(m-2-2) edge node[auto]{\(F_1=q\circ\overline{p}_1\)} (m-2-3)
				edge node[auto]{\((s,t)=(\overline{p}_1,\overline{p}_2)\)} (m-3-2)
			(m-2-3) edge node[auto]{\((s,t)=(\id_Q, \id_Q)\)} (m-3-3)
			(m-3-2) edge node[auto]{\(F_0\times F_0=q\times q\)} (m-3-3)
		};
	\end{diag}
	there exists a map \((f,(g_1, g_2))\) making the above diagram commutative.
	One can read of that \((g_1, g_2)\) can be seen as a map taking values in \(M\times_Q M=R\) and satisfies the conditions.
	Hence, \(F\) is a weak equivalence.
\end{ex}
\begin{ex}\label{ex:TranslationGroupoidAndQuotient}
	Let \(a\colon G\times M\to M\) be a free proper group action with quotient \(q\colon M\to \faktor{M}{G}\) and denote the identity groupoid over the quotient \(\faktor{M}{G}\) by \(\groupoid{Q}\).
	We construct a weak equivalence \(F\colon \groupoid{G}\to\groupoid{Q}\) by \(F_0\colon \groupoid{G}_0=M\to\groupoid{Q}_0=Q\) and \(F_1\colon q\circ p_M\colon \groupoid{G}_1\to \groupoid{Q}_1=Q\).
	The proof that \(F\) is a weak equivalence proceeds analogously to Example~\ref{ex:EquivalenceRelationGroupoidAndQuotient} and is omitted here.
	Alternatively, one can show that the homomorphism between the translation groupoid and the regular equivalence relation groupoid in Example~\ref{ex:HomomorphismTranslationGroupoidEquivalenceRelationGroupoid} is an equivalence and use associativity of the equivalence relation.
\end{ex}
\begin{lemma}\label{lemma:FiniteIsotropyActionProperEtaleGroupoid}
	Let \(G\) be a super Lie group that acts properly and with finite isotropy groups on a supermanifold \(M\) of dimension \(m|2n\).
	There exists a proper étale groupoid~\(\groupoid{G}\) and a weak equivalence \(F\colon \groupoid{G}\to G\ltimes M\).
\end{lemma}
\begin{proof}
	Choose a set of \(B\)-points \(p_\alpha\colon B\to M\), \(\alpha\in A\) with isotropy groups \(H_{p_\alpha}\) and slices \(S_\alpha\) for the action such that \(G\times_{H_{p_\alpha}} S_\alpha\) covers \(M\).
	We define a groupoid \(\groupoid{G}\) with \(\groupoid{G}_0=\coprod_{\alpha\in A} S_\alpha\).
	The inclusions \(S_\alpha\to M\) yield a map \(F_0\colon \groupoid{G}_0\to M\).
	Setting \(\groupoid{G}_1={\left(G\ltimes M\right)}_1\times_{(s,t), F_0\times F_0}\left(\groupoid{G}_0\times\groupoid{G}_0\right)\) and \(F_1\colon \groupoid{G}_1\to {\left(G\ltimes M\right)}_1\) to be the projection on the first factor we obtain by construction an equivalence of groupoids \(F\colon\groupoid{G}\to G\ltimes M\).

	In order to show that \(\groupoid{G}\) is a proper, étale groupoid, we note that the dimension of \(\groupoid{G}_1={\left(G\ltimes M\right)}_1\times_{(s,t), F_0\times F_0}\left(\groupoid{G}_0\times\groupoid{G}_0\right)\) coincides with the dimension of the slices.
	As the source and target maps are surjective submersions, they need to be local diffeomorphisms.
\end{proof}

\begin{defn}
	The super Lie groupoids \(\groupoid{H}\) and \(\groupoid{H}'\) are called Morita equivalent if there exists a super Lie groupoid \(\groupoid{G}\) and weak equivalences \(F\colon \groupoid{G}\to \groupoid{H}\) and \(F'\colon \groupoid{G}\to\groupoid{H}'\).

	A superorbifold is the Morita equivalence class of a proper and étale super Lie groupoid.
\end{defn}
\begin{rem}
	In order to see that Morita equivalence is an equivalence relation on the set of super Lie groupoids we need to verify associativity.
	Let \(\groupoid{H}\) and \(\groupoid{H}'\) as well as \(\groupoid{H}'\) and \(\groupoid{H}''\) be Morita equivalent, that is, there exist weak equivalences \(F^{\groupoid{G}\groupoid{H}}\colon \groupoid{G}\to\groupoid{H}\), \(F^{\groupoid{G}\groupoid{H}'}\colon \groupoid{G}\to\groupoid{H}'\), \(F^{\groupoid{G}'\groupoid{H}'}\colon\groupoid{G}'\to \groupoid{H}'\) and \(F^{\groupoid{G}'\groupoid{H}''}\colon \groupoid{G}'\to \groupoid{H}''\).
	We need to show that there is a super Lie groupoid~\(\groupoid{K}\) and weak equivalences \(F^{\groupoid{K}\groupoid{H}}\colon \groupoid{K}\to \groupoid{H}\) and \(F^{\groupoid{K}\groupoid{H}''}\colon\groupoid{K}\to\groupoid{H}''\).
	\begin{diag}
		\matrix[mat](m){
			&& \groupoid{K} && \\
			&\groupoid{G} && \groupoid{G}' & \\
			\groupoid{H} && \groupoid{H}' && \groupoid{H}'' \\
		} ;
		\path[pf]{
			(m-1-3) edge[dotted] node[auto]{\(F^{\groupoid{K}\groupoid{G}}\)} (m-2-2)
				edge[dotted] node[auto]{\(F^{\groupoid{K}\groupoid{G}'}\)} (m-2-4)
				edge[dashed, bend right=70] node[auto,swap]{\(F^{\groupoid{K}\groupoid{H}}\)} (m-3-1)
				edge[dashed, bend left=70] node[auto]{\(F^{\groupoid{K}\groupoid{H}''}\)} (m-3-5)
			(m-2-2) edge node[auto]{\(F^{\groupoid{G}\groupoid{H}}\)} (m-3-1)
				edge node[auto]{\(F^{\groupoid{G}\groupoid{H}'}\)} (m-3-3)
			(m-2-4) edge node[auto]{\(F^{\groupoid{G}'\groupoid{H}'}\)} (m-3-3)
				edge node[auto]{\(F^{\groupoid{G}'\groupoid{H}''}\)} (m-3-5)
		};
	\end{diag}
	As in~\cite[Section~2.3]{M-OAGI} there are two steps to verify:
	\begin{itemize}
		\item
			There exists a groupoid \(\groupoid{K}\) together with weak equivalences \(F^{\groupoid{K}\groupoid{G}}\colon \groupoid{K}\to \groupoid{G}\) and \(F^{\groupoid{K}\groupoid{G}'}\colon \groupoid{K}\to \groupoid{G}'\).
			The super Lie groupoid \(\groupoid{K}\) is constructed by setting
			\begin{align}
				\groupoid{K}_0 &= \groupoid{G}_0\times _{F^{\groupoid{G}\groupoid{H}'}_0, s} \groupoid{H}'_1 \times_{t, F^{\groupoid{G}'\groupoid{H}'}_0} \groupoid{G}'_0 &
				\groupoid{K}_1 &= \groupoid{G}_1\times _{F^{\groupoid{G}\groupoid{H}'}_0\circ s, s} \groupoid{H}'_1 \times_{t, F^{\groupoid{G}'\groupoid{H}'}_0\circ t} \groupoid{G}'_1 &
			\end{align}
			and extend the source, target and composition maps component-wise.
			The projections on the first and last factor \(F^{\groupoid{K}\groupoid{G}}\colon \groupoid{K}\to\groupoid{G}\) and \(F^{\groupoid{K}\groupoid{G}'}\colon \groupoid{K}\to \groupoid{G}'\) can be shown to be weak equivalences using that \(F^{\groupoid{G}\groupoid{H}'}\) and \(F^{\groupoid{G}'\groupoid{H}'}\) are weak equivalences.
			This construction of \(\groupoid{K}\) is called weak fiber product and satisfies a universal property.
			For more details see~\cite[Chapter~5.3]{MM-IFLG}.
		\item
			The composition of weak equivalences is again a weak equivalence.
			Consequently, \(F^{\groupoid{G}\groupoid{H}}\circ F^{\groupoid{K}\groupoid{G}}\colon \groupoid{K}\to \groupoid{H}\) and \(F^{\groupoid{G}'\groupoid{H}''}\circ F^{\groupoid{K}\groupoid{G}'}\colon \groupoid{K}\to \groupoid{H}''\) are weak equivalences and yield a Morita equivalence between \(\groupoid{H}\) and \(\groupoid{H}''\).
	\end{itemize}
\end{rem}
The idea of Morita equivalence is that an equivalence class represents a quotient.
We have already seen instances of this principle in Example~\ref{ex:EquivalenceRelationGroupoidAndQuotient} where we have shown that a regular equivalence relation groupoid and its quotient lie in the same Morita equivalence class.
Similarly, Example~\ref{ex:TranslationGroupoidAndQuotient} shows that the translation groupoid of a proper free groupoid and its quotient are Morita equivalent.
We want to apply this principle to more general situations, where the quotient cannot be realized as a supermanifold, in particular the situation of proper group actions with finite isotropy groups.

In order to better motivate the definition of Morita equivalence let us look at the orbits of the groupoid.
Note that the definition of super Lie groupoid implies that for all superpoints \(C_B\colon C\to B\) the image of \(\underline{(s,t)}(C_B)\) defines an equivalence relation \(\underline{R_{\groupoid{G}}}(C_B)\) on the set \(\underline{\groupoid{G}_0}(C_B)\).
\begin{defn}
	We define the orbit functor of the groupoid \(\groupoid{G}\) as
	\begin{equation}
		\begin{split}
			\underline{\groupoid{G}}\colon\cat{SPoint}^{op}_B &\to \cat{Top} \\
			C_B &\mapsto \faktor{\underline{\groupoid{G}_0}(C_B)}{\underline{R_\groupoid{G}}(C_B)}
		\end{split}
	\end{equation}
\end{defn}

The following lemma shows that Morita equivalent groupoids have the same orbit functor.
Compare also Remark~\ref{rem:EquivalenceRelationSuperPoints}.
\begin{lemma}
	A weak equivalence \(F\colon \groupoid{G}\to\groupoid{H}\) induces a natural isomorphism of functors \(\underline{F}\colon \underline{\groupoid{G}}\to \underline{\groupoid{H}}\).
\end{lemma}
\begin{proof}
	By the definition of homomorphism of super Lie groupoids it holds
	\begin{equation}
		\underline{F_0\times F_0}(C_B)(\underline{R_{\groupoid{G}}}(C_B))\subset \underline{R_{\groupoid{H}}}(C_B)
	\end{equation}
	Consequently, the map \(\underline{F_0}(C_B)\) descends to a map between the quotients:
	\begin{equation}
		\underline{F}{(C_B)}\colon \faktor{\underline{\groupoid{G}_0}(C_B)}{\underline{R_{\groupoid{G}}}(C_B)} \to \faktor{\underline{\groupoid{H}_0}(C_B)}{\underline{R_{\groupoid{H}}}(C_B)}
	\end{equation}
	The map \(\underline{F}(C_B)\) is surjective because, by condition~\ref{item:defn:LieGroupoidEquivalence:SurjectiveSubmersion}, for any \(h\in\underline{\groupoid{H}_0}(C_B)\) there exists an \(H\in\underline{\groupoid{H}_1}(C_B)\) and a \(g\in\underline{\groupoid{G}_0}(C_B)\) such that \(\underline{F_0}(C_B)(g)=\underline{s}(C_B)(H)\) and \(\underline{t}(C_B)=h\).
	That is, \((\underline{F_0}(C_B)(g), h)\) lies in \(\underline{R_{\groupoid{H}}}\) and hence the equivalence class of \(h\) lies in the image of \(\underline{F}(C_B)\).
	The map \(\underline{F}(C_B)\) is injective because if \((h, h')\in\underline{R_{\groupoid{H}}}(C_B)\) then there exists \((g, g')\in\underline{R_\groupoid{G}}(C_B)\) such that \(\underline{F_0\times F_0}(C_B)((g, g'))=(h, h')\) by condition~\ref{item:defn:LieGroupoidEquivalence:FiberProduct} of Definition~\ref{defn:LieGroupoidHomomorphism}.

	Continuity of \(\underline{F}(C_B)\) follows from continuity of \(\underline{F_0}(C_B)\) and the continuity of the projections.
	To see that \(\underline{F}(C_B)\) is also open we consider the following commutative diagram:
	\begin{diag}
		\matrix[mat](m){
			\underline{\groupoid{G}_1}(C_B)\simeq \underline{\groupoid{H}_1\times_{(s,t),F_0\times F_0} \left(\groupoid{G}_0\times \groupoid{G}_0\right)}(C_B) & \underline{\groupoid{H}_1}(C_B) \\
			\underline{\groupoid{G}_0}(C_B) & \underline{\groupoid{H}_0}(C_B) \\
			\underline{\groupoid{G}}(C_B)=\faktor{\underline{\groupoid{G}_0}(C_B)}{\underline{R_{\groupoid{G}}}(C_B)} & \underline{\groupoid{G}}(C_B)=\faktor{\underline{\groupoid{H}_0}(C_B)}{\underline{R_{\groupoid{H}}}(C_B)} \\
		} ;
		\path[pf]{
			(m-1-1) edge node[auto]{\(\underline{p_1}(C_B)\)} (m-1-2)
				edge node[auto]{\(\underline{t}(C_B)\)} (m-2-1)
			(m-1-2) edge node[auto]{\(\underline{t}(C_B)\)} (m-2-2)
			(m-2-1) edge node[auto]{\(\underline{F_0}(C_B)\)} (m-2-2)
			edge node[auto]{\(\underline{p_\groupoid{G}}(C_B)\)} (m-3-1)
			(m-2-2) edge node[auto]{\(\underline{p_\groupoid{H}}(C_B)\)}(m-3-2)
			(m-3-1) edge node[auto]{\(\underline{F}(C_B)\)} (m-3-2)
		};
	\end{diag}
	Let now \(U\subset\underline{\groupoid{G}}(C_B)\) be open.
	To show that \(\underline{F}(C_B)(U)\) is open it suffices to show that the saturation of
	\begin{equation}
		\tilde{V}=\underline{F_0}(C_B)\left({\left(\underline{p_\groupoid{G}}(C_B)\right)}^{-1}(U)\right)
	\end{equation}
	is open in \(\underline{\groupoid{H}_0}(C_B)\).
	The set \(\tilde{V}\) is open because \({\left(\underline{p_\groupoid{G}}(C_B)\circ\underline{t}(C_B)\right)}^{-1}(U)\) is open and \(\underline{t}(C_B)\circ\underline{p_1}(C_B)\) is a surjective submersion and hence open.
	The saturation of \(\tilde{V}\) is the union
	\begin{equation}
		\underline{s}(C_B)\left({\left(\underline{t}(C_B)\right)}^{-1}(\tilde{V})\right) \cup\underline{t}(C_B)\left({\left(\underline{s}(C_B)\right)}^{-1}(\tilde{V})\right)
	\end{equation}
	and hence open because \(\underline{s}(C_B)\) and \(\underline{t}(C_B)\) are open.

	Functoriality of \(\underline{F}(C_B)\) follows from functoriality of the functors of points of \(\groupoid{G}_0\), \(\groupoid{G}_1\), \(\groupoid{H}_0\) and \(\groupoid{H}_1\) as well as all constructions involved.
\end{proof}
\begin{ex}
	Let \(\groupoid{G}\) be a proper étale groupoid and \(j\colon V\to\groupoid{G}_0\) an open subsupermanifold.
	We construct the restriction \(\groupoid{G}|_{V}=\groupoid{V}\) by setting \(\groupoid{V}_0=V\) and \(\groupoid{V}_1=\groupoid{G}_1\times_{(s,t),j\times j} \left(V\times V\right)\) and obtain a homomorphism of super Lie groupoids \(\groupoid{V}\to \groupoid{G}\).
	Up to restricting \(V\) further we have that \(\groupoid{V}_1\) is a trivial finite cover of \(\groupoid{V}_0=V\) and hence \(\groupoid{V}\) is the translation groupoid of a finite group action on \(V\).
	The orbit functor \(\underline{\groupoid{G}}\) of \(\groupoid{G}\) is locally isomorphic to the orbit functor \(\underline{\groupoid{V}}\) of \(\groupoid{V}\).

	We can find an open cover \(\Set{U_\alpha}_{\alpha\in A}\) of \(\groupoid{G}_0\) such that all restrictions \(\groupoid{G}|_{U_\alpha}\) are translation groupoids \(G_\alpha\ltimes U_\alpha\) for the action \(a_\alpha\colon G_\alpha\times U_\alpha\to U_\alpha\) of some finite group \(G_\alpha\) on \(U_\alpha\).
	We define a groupoid \(\groupoid{U}\) by
	\begin{align}
		\groupoid{U}_0 &= \coprod_{\alpha\in A} U_\alpha, &
		\groupoid{U}_1 &= \groupoid{G}_1\times_{(s,t),j\times j} \left(\groupoid{U}_0\times\groupoid{U}_0\right),
	\end{align}
	where \(j\colon \groupoid{U}_0 = \coprod U_\alpha\to \groupoid{G}\) is the map induced from the inclusions \(U_\alpha\subset\groupoid{G}_0\).
	By construction the resulting super Lie groupoid homomorphism \(\groupoid{U}\to\groupoid{G}\) is a weak equivalence.
	Up to refining the cover \(\Set{U_\alpha}\) we can assume that if \(U_\alpha\cap U_\beta\neq \emptyset\) then \(U_\alpha\cap U_\beta\in\Set{U_\alpha}\) and that any \(U_\alpha\) is an open subsupermanifold of \(\R^{m|n}\).
	Any embedding \(\lambda_{\alpha\beta}\colon U_\beta\to U_\alpha\) yields a homomorphism between super Lie groupoids \(G_\alpha\ltimes U_\alpha\to G_\beta\ltimes U_\beta\), that is a group homomorphism \(\rho_{\alpha\beta}\colon G_\alpha\to G_\beta\) such that \(\lambda_{\alpha\beta}\) is \(\rho_{\alpha\beta}\)-equivariant.

	We have thus arrived at the notions of superorbifold charts and atlas:
	Let \(\underline{Q}\colon \cat{SPoint}^{op}_B\to \cat{Top}\) be a functor.
	\begin{itemize}
		\item
			A superorbifold chart on \(\underline{Q}\) is an open subsupermanifold \(U\subset\R^{m|n}\times B\) over \(B\) together with an action \(a\colon G\times U\to U\) of a finite group \(G\) on \(U\) and a natural transformation \(\underline{q}\colon \underline{G\ltimes U}\to\underline{Q}\) such that \(\underline{G\ltimes U}\) is naturally isomorphic to its image in \(\underline{Q}\).
			\item
				An embedding \(\Lambda\colon (U, G, a, \underline{q}) \to (U', G', a', \underline{q}')\) between two superorbifold charts consists of a group homomorphism \(\rho\colon G\to G\) and a smooth embedding \(\lambda\colon U\to U'\) which is \(\rho\)-equivariant.
				Furthermore the induced map \(\underline{\lambda}\colon \underline{G\ltimes U}\to \underline{G'\ltimes U'}\) satisfies \(\underline{q}'\circ \underline{\lambda} = \underline{q}\).
			\item
				A superorbifold atlas on \(\underline{Q}\) consists of a set of superorbifold charts
				\begin{equation}
					\Set{(U_\alpha, G_\alpha, a_\alpha, \underline{q}_\alpha)}_{\alpha\in A}
				\end{equation}
				and a set of embeddings
				\begin{equation}
					\Set{\Lambda_{\alpha\beta}=(\lambda_{\alpha\beta}, \rho_{\alpha\beta}, \underline{\lambda}_{\alpha\beta}) \colon (U_\alpha, G_\alpha, a_\alpha, \underline{q}_{\alpha})\to (U_\beta, G_\beta, a_\beta, \underline{q}_\beta)}_{(\alpha, \beta)\in A'}
				\end{equation}
				such that whenever there is a \(p\in\underline{Q}(C_B)\) such that \(p\in\im \underline{q}_\alpha(C_B)\cap\im \underline{q}_\beta(C_B)\) there is another superorbifold chart \((U_\gamma, G_\gamma, a_\gamma, \underline{q}_\gamma)\) such that \(p\in\im \underline{q}_\gamma(C_B)\) as well as embeddings \(\Lambda_{\gamma\alpha}\) and \(\Lambda_{\gamma\beta}\).
	\end{itemize}
	We have shown above, how a proper étale super Lie groupoid yields a superorbifold atlas.
	The data of a superorbifold atlas can be seen as an alternative definition of a superorbifold that is closer to the original ideas of~\cite{S-GNM}.

	We briefly indicate how to construct a super Lie groupoid \(\groupoid{G}\) out of an superorbifold atlas such that \(\underline{\groupoid{G}}=\underline{Q}\).
	Let
	\begin{align}
		\groupoid{G}_0 &= \coprod_{\alpha\in A} U_\alpha, &
		\groupoid{G}_1 &= \coprod_{\alpha\in A} \groupoid{G}_1^\alpha \cup \coprod_{(\alpha, \beta)\in A'} \groupoid{G}_1^{\alpha\beta} \cup \coprod_{(\alpha, \beta)\in A'} \groupoid{G}_1^{\beta\alpha},
	\end{align}
	where
	\begin{align}
		\groupoid{G}_1^\alpha &= G_\alpha\times U_\alpha, &
		\groupoid{G}_1^{\alpha\beta} &= G_\beta\times U_\alpha, &
		\groupoid{G}_1^{\beta\alpha}= \left(G_\beta\times U_\beta\right)\times_{a_\beta, \lambda_{\alpha\beta}}U_\alpha.
	\end{align}
	All the maps necessary for the definition of the groupoid will be defined on the disjoint open subsets separately:
	\begin{align}
		s|_{\groupoid{G}_1^\alpha} &= p_{U_\alpha} &
		s|_{\groupoid{G}_1^{\alpha\beta}} &= p_{U_\alpha} &
		s|_{\groupoid{G}_1^{\beta\alpha}} &= p_{U_\beta} \\
		t|_{\groupoid{G}_1^\alpha} &= a_\alpha &
		t|_{\groupoid{G}_1^{\alpha\beta}} &= a_\beta\circ\left(\id_{G_\beta}\times \lambda_{\alpha\beta}\right) &
		t|_{\groupoid{G}_1^{\beta\alpha}} &= p_{U_\alpha}
	\end{align}
	In particular \(\groupoid{G}_1\) contains all arrows of the local translation groupoids \(\groupoid{G}_1^\alpha={\left(G_\alpha\ltimes U_\alpha\right)}_1\) and we use their identity, composition and inverse maps.
	The composition maps
	\begin{align}
		c|_{\groupoid{G}_1^\alpha\times_{t,s}\groupoid{G}_1^{\alpha\beta}}\colon \groupoid{G}_1^\alpha\times_{t,s}\groupoid{G}_1^{\alpha\beta}&\to \groupoid{G}_1^{\alpha\beta} &
		c|_{\groupoid{G}_1^{\alpha\beta}\times_{t,s}\groupoid{G}_1^\beta} \colon \groupoid{G}_1^{\alpha\beta}\times_{t,s}\groupoid{G}_1^\beta &\to \groupoid{G}_1^{\alpha\beta} \\
		c|_{\groupoid{G}_1^\beta\times_{t,s}\groupoid{G}_1^{\beta\alpha}}\colon \groupoid{G}_1^\beta\times_{t,s}\groupoid{G}_1^{\beta\alpha}&\to \groupoid{G}_1^{\beta\alpha} &
		c|_{\groupoid{G}_1^{\beta\alpha}\times_{t,s}\groupoid{G}_1^\alpha} \colon \groupoid{G}_1^{\beta\alpha}\times_{t,s}\groupoid{G}_1^\alpha &\to \groupoid{G}_1^{\beta\alpha}
	\end{align}
	are defined by using the group homomorphism \(\rho_{\alpha\beta}\) and the group multiplication in \(G_\beta\) in the only possible way.
	The \(\rho_{\alpha\beta}\)-equivariance of \(\lambda_{\alpha\beta}\) assures that all maps are well defined.
	The inverse maps
	\begin{align}
		i|_{\groupoid{G}_1^{\alpha\beta}} \colon \groupoid{G}_1^{\alpha\beta} &\to \groupoid{G}_1^{\beta\alpha} &
		i|_{\groupoid{G}_1^{\beta\alpha}} \colon \groupoid{G}_1^{\beta\alpha} &\to \groupoid{G}_1^{\alpha\beta}
	\end{align}
	are constructed with the inverse in the group \(G_\beta\).
\end{ex}
\begin{rem}
	We do not attempt to define smooth maps between superorbifolds here.
	Certainly, a homomorphism between super Lie groupoids induces a smooth map between the corresponding Morita equivalence classes.
	However, the question which of the super Lie groupoid homomorphisms induce equivalent smooth maps between the Morita equivalence classes is subtle.
	For a discussion of this point in the setting of non-super orbifolds we refer to~\cite[Section~3.1]{L-OAS}.
\end{rem}
\begin{defn}
	Let \(\groupoid{G}\) be a super Lie groupoid representing a superorbifold.
	We say that the dimension of the superorbifold \(\groupoid{G}\) is
	\begin{equation}
		\dim\groupoid{G} = 2\dim \groupoid{G}_0 - \dim\groupoid{G}_1.
	\end{equation}
\end{defn}
To see that the dimension is well defined, suppose \(F\colon \groupoid{G}\to \groupoid{H}\) is a weak equivalence.
Hence,
\begin{equation}
	\groupoid{G}_1 = \groupoid{H}_1\times_{(s,t), F_0\times F_0}\left(\groupoid{G}_0\times \groupoid{G}_0\right)
\end{equation}
and consequently
\begin{equation}
	\dim \groupoid{G}_1 = \dim \groupoid{H}_1 + 2\dim \groupoid{G}_0 - 2 \dim\groupoid{H}_0.
\end{equation}
If \(\groupoid{G}\) is a proper étale groupoid, we have \(\dim \groupoid{G} = \dim \groupoid{G}_0\).
If \(R\) is a regular equivalence relation on \(M\) the dimension of the regular equivalence relation groupoid is \(2\dim M - \dim R\) and if \(a\colon G\times M\) is a group action, the dimension of the translation groupoid is \(\dim M - \dim G\).

The following Theorem is now a reformulation of Lemma~\ref{lemma:FiniteIsotropyActionProperEtaleGroupoid}:
\QuotientOrbifoldTheorem{}
 
\section{Stable supercurves of genus zero}
In this section we discuss several moduli spaces of super Riemann surfaces of genus zero with marked points.
Recall that the classical moduli space \(M_{0,k}\) of Riemann surfaces of genus zero with \(k\) marked points possesses a compactification by stable curves.
This idea goes back to Deligne and Mumford using the methods of algebraic geometry, see for example~\cite{HM-MC}, and the algebro-geometric approach has been generalized to super Riemann surfaces in~\cites{D-LaM}{OV-SMSGZSUSYCRP}{FKP-MSSCCLB}.

As a preparation for Section~\ref{Sec:SuperStableMapsOfGenusZero}, we will use here a more differential geometric approach:
Stable Riemann surfaces of genus can be modeled as trees, where each vertex represents a copy of \(\ProjectiveSpace[\C]{1}\) and edges represent nodes, see~\cite[Appendix~D]{McDS-JHCST}.
If every copy has at least three special points, that is nodes or marked points, the resulting nodal Riemann surface is stable and has no automorphisms.

In the supergeometric generalization proposed here, a stable super Riemann surfaces is consequently modeled as a tree of super Riemann surfaces of genus zero with marked points.
Using a detailed study of their automorphisms and the methods of the last chapter we show that the moduli space of stable super Riemann surfaces of genus zero and fixed tree type is a superorbifold.
Taking the union over all tree types yields a realization of a compact moduli space of stable super Riemann surfaces of genus zero.

In Section~\ref{sec:AutomorphismsOfP11}, we discuss the supergeometric analogue of Möbius transformations on the only super Riemann surface of genus zero, the projective superspace~\(\ProjectiveSpace[\C]{1|1}\), and show that a super Möbius transformation is completely determined by the image of three points.
In Section~\ref{sec:MarkedPointsOnP11} we discuss the moduli space of \(k\)-marked points on a single copy of \(\ProjectiveSpace[\C]{1|1}\) and in Section~\ref{sec:StableSRSofGenusZero} we discuss the moduli space of a stable super Riemann surface of fixed tree type.

\subsection{Superconformal automorphisms of \texorpdfstring{\(\ProjectiveSpace[\C]{1|1}\)}{PC1|1}}\label{sec:AutomorphismsOfP11}
In this section, we recall the super Riemann surface structure on \(\ProjectiveSpace[\C]{1|1}\) and investigate its automorphisms.
Recall that the supermanifold structure on \(\ProjectiveSpace[\C]{1|1}\) is given by an atlas with two open sets \(U_1\simeq\C^{1|1}\) and \(U_2\simeq\C^{1|1}\) with holomorphic coordinates \((z_1, \theta_1)\) and \((z_2, \theta_2)\) respectively.
The coordinates are identified away from zero by
\begin{align}
	z_2 &= \frac1{z_1}, &
	\theta_2 &= \frac{\theta_1}{z_1}.
\end{align}
The superconformal structure on \(\ProjectiveSpace[\C]{1|1}\) is generated by \(\partial_{\theta_1}+\theta_1\partial_{z_1}\) and \(\partial_{\theta_2} - \theta_2\partial_{z_2}\) respectively.

The \(B\)-points of \(\ProjectiveSpace[\C]{1|1}\) can be described in terms of projective coordinates, that is, triples \([Z_1:Z_2:\Theta]\), where at least one of \(Z_i\in{\left(\cO_B\otimes\C\right)}_0\) is invertible and \(\Theta\in{\left(\cO_B\otimes\C\right)}_1\).
The triple \([Z_1:Z_2:\Theta]\) describes the \(B\)-point of \(U_1\) given by
\begin{align}
	z_1 &= \frac{Z_1}{Z_2} &
	\theta_1 &= \frac{\Theta}{Z_2}
\end{align}
if \(Z_2\) is invertible and the \(B\)-point of \(U_2\) given by
\begin{align}
	z_2 &= \frac{Z_2}{Z_1} &
	\theta_2 &= \frac{\Theta}{Z_1}
\end{align}
if \(Z_1\) is invertible.
Notice that \([Z_1:Z_2:\Theta]\) and \([\lambda Z_1: \lambda Z_2 : \lambda\Theta]\) describe the same \(B\)-point of \(\ProjectiveSpace[\C]{1|1}\) for every invertible \(\lambda\in\cO_B\otimes\C\).
We denote by \(0\), \(1\) and \(\infty\) the \(B\)-points \([0:1:0]\), \([1:1:0]\) and \([1:0:0]\) respectively.
Furthermore, for an odd \(\epsilon\in\cO_B\otimes\C\) we denote by \(1_\epsilon\) the point \([1:1:\epsilon]\).

Any \(L\in\GL_{\cO_B\otimes\C}(2|1)\) induces an automorphism \(l\) of \(\ProjectiveSpace[\C]{1|1}\times B\) over \(B\) via matrix multiplication on the projective coordinates:
\begin{equation}
	[\tilde{Z}_1:\tilde{Z}_2:\tilde{\Theta}]
	= [Z_1:Z_2:\Theta]
	\begin{pmatrix}
		a & c & \gamma \\
		b & d & \delta \\
		\alpha & \beta & e\\
	\end{pmatrix}
\end{equation}
If the matrix \(L\) is an element of \(\SpGL_{\cO_B\otimes\C}(2|1)\), that is,
\begin{equation}\label{eq:Sp21}
	\begin{aligned}
		ad - bc - \gamma\delta &= 1, &
		a\beta - c\alpha + e\gamma &= 0, \\
		e^2 + 2\alpha\beta &= 1, &
		b\beta - d\alpha + e\delta &= 0,
	\end{aligned}
\end{equation}
the automorphism \(l\) is superconformal.
A superconformal automorphism is an automorphism of the super Riemann surface \(\ProjectiveSpace[\C]{1|1}\).
That is, it preserves in addition to the complex structure also the distribution~\(\cD\).
In the coordinates \((z_1, \theta_1)\) the superconformal automorphism \(l\) induced by \(L\in\SpGL_{\cO_B\otimes\C}(2|1)\) is given by
\begin{align}
	l^\#z_1
	&= \frac{a z_1 + b + \theta_1\alpha}{cz_1 + d + \theta_1\beta}
	= \frac{az_1 + b}{cz_1+ d} \pm \theta_1\frac{\gamma z_1 + \delta}{{\left(c z_1 + d\right)}^2} \\
	l^\#\theta_1
	&= \frac{\gamma z_1 + \delta + \theta_1 e}{c z_1 + d + \theta_1 \beta}
	= \frac{\gamma z_1 + \delta}{c z_1 + d} \pm \theta_1\frac1{cz_1 + d}
\end{align}
Expressions in the coordinates \((z_2, \theta_2)\) can be obtained analogously.
For details of the calculation see~\cite[Examples~9.4.4 and 2.10.11]{EK-SGSRSSCA}.
Earlier sources are~\cites{H-MSRS}{CR-SRSUTT}[Chapter~2.1]{M-TNCG}.

More abstractly, the super Lie group \(\SpGL_\C(2|1)\) given by \(\underline{\SpGL_\C(2|1)}(B)=\SpGL_{\cO_B\otimes\C}(2|1)\) acts on \(\ProjectiveSpace[\C]{1|1}\).
Any \(B\)-point \(L\in\underline{\SpGL_\C(2|1)}(B)\) describes a superconformal automorphisms \(l\colon \ProjectiveSpace[\C]{1|1}\times B\to \ProjectiveSpace[\C]{1|1}\) and acts on \(B\)-points \(p\colon B\to \ProjectiveSpace[\C]{1|1}\) by composition \(l\circ p\).
The algebraic super Lie group structure of \(\SpGL_{\C}(2|1)\) has been investigated in~\cite{FK-PLSGSPAP11}.
In the following we will show that any superconformal automorphism of \(\ProjectiveSpace[\C]{1|1}\times B\) over \(B\) is induced from an element \(L\in\underline{\SpGL_{\C}(2|1)}(B)\) but first we will list some examples.

\begin{ex}[{Lift of Möbius transformations to \(\ProjectiveSpace[\C]{1|1}\)}]
	Any Möbius transformation \(\frac{a_0z + b_0}{c_0z+d_0}\) of \(\ProjectiveSpace[\C]{1}\) with complex coefficients \(a_0\), \(b_0\), \(c_0\) and \(d_0\) such that \(a_0 d_0 - b_0 c_0=1\) lifts to an automorphism of \(\ProjectiveSpace[\C]{1|1}\) given by
	\begin{align}
		l^\# z &= \frac{a_0z + b_0}{c_0z + d_0} &
		l^\# \theta &= \theta\frac1{c_0z + d_0}
	\end{align}
	The corresponding matrix in \(\SpGL_{\cO_B\otimes\C}(2|1)\) is given by
	\begin{equation}
		\begin{pmatrix}
			a_0 & c_0 & 0 \\
			b_0 & d_0 & 0 \\
			0 & 0 & 1 \\
		\end{pmatrix}
	\end{equation}
\end{ex}
\begin{ex}[Reflection of the odd directions]\label{ex:ReflectionOddDirections}
	Another important example, the reflection of the odd direction \(\Xi_-\) is given by the matrix
	\begin{equation}
		\begin{pmatrix}
			-1 & 0 & 0 \\
			0 & -1 & 0 \\
			0 & 0 & 1 \\
		\end{pmatrix}.
	\end{equation}
	In the coordinates \((z_1, \theta_1)\) the automorphism \(\Xi_-\) is given by
	\begin{align}
		\Xi_-^\#z_1 &= z_1, &
		\Xi_-^\#\theta_1 &= -\theta_1.
	\end{align}
\end{ex}
\begin{ex}
	For odd elements \(\gamma\), \(\delta\in\cO_B\otimes\C\) and an even nilpotent \(\sigma\in\cO_B\) let \(l\) be the automorphism induced by the matrix
	\begin{equation}
		\begin{pmatrix}
			1 & \sigma & \gamma \\
			0 & 1 + \gamma\delta & \delta \\
			\delta  & \sigma\delta - \gamma  & 1-\gamma\delta \\
		\end{pmatrix}
	\end{equation}
	In the coordinates \((z_1, \theta_1)\) the automorphism \(l\) is given by
	\begin{align}
		l^\# z_1 &= \frac{z_1}{\sigma z_1 + 1 + \gamma\delta} + \theta_1\frac{\gamma z_1 + \delta}{{\left(\sigma z_1 + 1 + \gamma\delta\right)}^2} &
		l^\#\theta_1 &= \frac{\gamma z_1 + \delta}{\sigma z_1 + 1 + \gamma\delta} + \theta_1\frac1{\sigma z_1 + 1 + \gamma\delta}
	\end{align}
	The denominators can be rewritten as follows
	\begin{equation}
		\begin{split}
			\frac1{\sigma z_1 + 1 + \gamma\delta}
			&= \sum_{n=0}^k {\left(- \sigma z_1 - \gamma\delta\right)}^n
			= \sum_{n=0}^k{\left(-1\right)}^n \left(\sigma^n z_1^n + n\sigma^{n-1}\gamma\delta z^{n-1}\right) \\
			&= \sum_{n=0}^k{\left(-\sigma\right)}^n\left(1 + (n+1)\gamma\delta\right) z_1^n
		\end{split}
	\end{equation}
	where \(k\) is the smallest integer such that \(\sigma^{k+1}=0\).
	This leads to
	\begin{align}
		\begin{split}
			l^\# z_1
			={}& \sum_{n=0}^k{\left(-\sigma\right)}^n\left(1 + (n+1)\gamma\delta\right) z_1^{n+1} + \theta_1\left(\gamma z_1 + \delta\right){\left(\sum_{n=0}^k{\left(-\sigma\right)}^n z_1^n\right)}^2 \\
			={}& \left(1 + \gamma\delta\right)z_1 - \sigma\left(1 + 2\gamma\delta\right)z_1^2 + \sigma^2\left(1+3\gamma\delta\right)z_1^3 + \dotsb \\
			&+ \theta_1\left( \delta + \left(\gamma - \frac{(k+1)k}2\sigma\delta\right)z_1 + \dotsb \right)
		\end{split} \\
		\begin{split}
			l^\#\theta_1
			={}&  \left(\gamma z_1 + \delta\right) \sum_{n=0}^k{\left(-\sigma\right)}^n z_1^n+ \theta_1\sum_{n=0}^k{\left(-\sigma\right)}^n\left(1 + (n+1)\gamma\delta\right) z_1^n \\
			={}& \delta + \left(\gamma - \sigma\delta\right)z_1 + (-\gamma\sigma + \sigma^2\delta)z_1^2 + \dotsb \\
			&+ \theta_1\left(1 + \gamma\delta - \sigma\left(1 + 2\gamma\delta\right)z_1 + \sigma^2\left(1+3\gamma\delta\right)z_1^2 + \dotsb \right)
		\end{split}
	\end{align}
	Those higher order polynomials might lead to the wrong impression that the automorphism~\(l\) is not induced by an \(L\in\SpGL_{\cO_B\otimes\C}(2|1)\).
	But as we will see below, in Corollary~\ref{Cor:AutomorphismsP11AllLinear}, all superconformal automorphisms of \(\ProjectiveSpace[\C]{1|1}\) are induced by \(B\)-points of \(\SpGL_{\C}(2|1)\).
\end{ex}

Recall that holomorphic automorphisms of \(\ProjectiveSpace[\C]{1}\) are Möbius transformations.
A Möbius transformation is completely determined by the image of three points.
The generalization of this statement to super Riemann surfaces is the following:

\begin{prop}\label{prop:LinearMapPrescribed3Points}
	Let \(p_i\colon B\to\ProjectiveSpace[\C]{1|1}\), \(i=1, \dotsc, 3\) be three \(B\)-points such that the reduction gives three distinct points of \(\ProjectiveSpace[\C]{1}\).
	There exists a matrix \(L\in\underline{\SpGL_{\C}(2|1)}(B)\) and an odd \(\epsilon\in\cO_B\otimes\C\) such that the induced superconformal automorphism maps \(0\) to \(p_1\), \(1_\epsilon\) to~\(p_2\) and \(\infty\) to \(p_3\).
\end{prop}
\begin{proof}
	Denote the projective coordinates of \(p_i\) by \([p_{i1}:p_{i2}:\pi_i]\).
	We are looking for a matrix \(L\in\SpGL_{\cO_B\otimes\C}(2|1)\), an odd \(\epsilon\in\cO_B\otimes\C\) and invertible \(\lambda_i\in\cO_B\otimes\C\) such that
	\begin{equation}\label{eq:AutomorphismsOfP11Def}
		\begin{split}
			\begin{pmatrix}
				0 & 1 & 0\\
				1 & 1 & \epsilon \\
				1 & 0 & 0 \\
			\end{pmatrix}
			\begin{pmatrix}
				a & c & \gamma \\
				b & d & \delta \\
				\alpha & \beta & e\\
			\end{pmatrix}
			&=
			\begin{pmatrix}
				b & d & \delta \\
				a + b + \epsilon\alpha & c + d + \epsilon\beta & \gamma + \delta + \epsilon e \\
				a & c & \gamma \\
			\end{pmatrix} \\
			&=
			\begin{pmatrix}
				\lambda_1 p_{11} & \lambda_1 p_{12} & \lambda_1 \pi_1 \\
				\lambda_2 p_{21} & \lambda_2 p_{22} & \lambda_2 \pi_2 \\
				\lambda_3 p_{31} & \lambda_3 p_{32} & \lambda_3 \pi_3 \\
			\end{pmatrix}.
		\end{split}
	\end{equation}
	From the first and last row we read off
	\begin{align}
		a &= \lambda_3 p_{31} &
		b &= \lambda_1 p_{11} &
		c &= \lambda_3 p_{32} &
		d &= \lambda_1 p_{12} &
		\gamma &= \lambda_3 \pi_3 &
		\delta &= \lambda_1 \pi_1
	\end{align}
	By~\cite[Example~2.10.11]{EK-SGSRSSCA}, the remaining entries of the matrix \(L\) are determined up to a sign by the Equations~\eqref{eq:Sp21}:
	\begin{align}
		e &= \pm\left(1-\gamma\delta\right) = \pm\left(1-\lambda_1\lambda_3\pi_3\pi_1\right), \\
		\alpha &= \pm\left(b\gamma - a\delta\right) = \pm\lambda_1\lambda_3\left(p_{11}\pi_3 - p_{31}\pi_1\right), \\
		\beta &= \pm\left(d\gamma - c\delta\right) = \pm\lambda_1\lambda_3\left(p_{12}\pi_3 - p_{32}\pi_1\right).
	\end{align}
	Now the second line of~\eqref{eq:AutomorphismsOfP11Def} yields
	\begin{align}
		\lambda_3 p_{31} + \lambda_1 p_{11} \pm \epsilon\lambda_1\lambda_3\left(p_{11}\pi_3 - p_{31}\pi_1\right) &= \lambda_2 p_{21}
		\label{eq:AutomorphismsOfP11Def21}\\
		\lambda_3 p_{32} + \lambda_1 p_{12} \pm \epsilon\lambda_1\lambda_3\left(p_{12}\pi_3 - p_{32}\pi_1\right) &= \lambda_2 p_{22}
		\label{eq:AutomorphismsOfP11Def22}\\
		\lambda_3 \pi_3 + \lambda_1\pi_1 \pm \epsilon\left(1 - \lambda_1\lambda_3\pi_3\pi_1\right) &= \lambda_2\pi_2
		\label{eq:AutomorphismsOfP11DefEpsilon}
	\end{align}
	As \(e=\pm\left(1-\lambda_1\lambda_3\pi_3\pi_1\right)\) is invertible, we can solve~\eqref{eq:AutomorphismsOfP11DefEpsilon} for \(\epsilon\):
	\begin{equation}
		\epsilon
		= \mp\left(1+\lambda_1\lambda_3\pi_3\pi_1\right)\left(\lambda_1\pi_1 - \lambda_2\pi_2 + \lambda_3\pi_3\right)
		= \mp\left(\lambda_1\pi_1 - \lambda_2\pi_2 + \lambda_3\pi_3 + \lambda_1\lambda_2\lambda_3\pi_1\pi_2\pi_3\right).
	\end{equation}
	The Equations~\eqref{eq:AutomorphismsOfP11Def21} and~\eqref{eq:AutomorphismsOfP11Def22} are linear in \(\lambda_1\), \(\lambda_3\) up to nilpotent perturbation and its reduction is solvable because \([\Red{p_{11}} : \Red{p_{12}}]\) and \([\Red{p_{21}} : \Red{p_{22}}]\) are different points of~\(\ProjectiveSpace[\C]{1}\).
		Consequently, the Equations~\eqref{eq:AutomorphismsOfP11Def21} and~\eqref{eq:AutomorphismsOfP11Def22} can be solved for \(\lambda_1\), \(\lambda_3\) as functions of~\(\lambda_2\) using expansions in the odd generators of \(\cO_B\).
	It remains to solve
	\begin{equation}
		ad - bc - \gamma\delta = 1
	\end{equation}
	to fix \(\lambda_2\).
	After all substitutions this is a quadratic function in \(\lambda_2\) up to nilpotent terms and can be solved, again, by recursion over the odd generators of \(\cO_B\).
\end{proof}
\begin{rem}
	The order of the points \(p_1\), \(p_2\) and \(p_3\) does matter in Proposition~\ref{prop:LinearMapPrescribed3Points}.
	Permutation of the order of the points leads to a multiplication of the parameter \(\epsilon\) by a power of \(\ic\).
	Indeed, there is an automorphism of \(\ProjectiveSpace[\C]{1|1}\) that is determined by
	\begin{align}
		0 &\mapsto 1_{\ic\epsilon} &
		1_\epsilon &\mapsto 0 &
		\infty &\mapsto \infty
	\end{align}
	as can be seen by the following matrix equation:
	\begin{equation}\label{eq:01inftyTo10Infty}
		\begin{pmatrix}
			0 & 1 & 0\\
			1 & 1 & \epsilon \\
			1 & 0 & 0 \\
		\end{pmatrix}
		\begin{pmatrix}
			\ic & 0 & 0 \\
			-\ic & -\ic & -\epsilon \\
			\ic\epsilon & 0 & 1\\
		\end{pmatrix}
		=
		\begin{pmatrix}
			-\ic & -\ic & -\epsilon \\
			0 & \ic & 0\\
			\ic & 0 & 0 \\
		\end{pmatrix}
	\end{equation}
	Similarly, there is an automorphism of \(\ProjectiveSpace[\C]{1|1}\) such that
	\begin{align}
		0 &\mapsto 0 &
		1_\epsilon &\mapsto \infty &
		\infty &\mapsto 1_{\ic\epsilon}
	\end{align}
	given by
	\begin{equation}\label{eq:01inftyTo0Infty1}
		\begin{pmatrix}
			0 & 1 & 0\\
			1 & 1 & \epsilon \\
			1 & 0 & 0 \\
		\end{pmatrix}
		\begin{pmatrix}
			-\ic & -\ic & -\epsilon \\
			0 & \ic & 0 \\
			0 & -\ic\epsilon & 1\\
		\end{pmatrix}
		=
		\begin{pmatrix}
			0 & \ic & 0\\
			-\ic & 0 & 0 \\
			-\ic & -\ic & -\epsilon \\
		\end{pmatrix}
	\end{equation}
	More automorphisms of \(\ProjectiveSpace[\C]{1|1}\) that map \(\Set{0, 1_\epsilon, \infty}\) to \(\Set{0, 1_{\epsilon'}, \infty}\) can be obtained by composition of the linear maps given in~\eqref{eq:01inftyTo10Infty} and~\eqref{eq:01inftyTo0Infty1} and their inverses.
\end{rem}
\begin{rem}
	The number \(\epsilon\) such that three given \(B\)-points \(p_1\), \(p_2\), \(p_3\) can be mapped to \(0\), \(1_\epsilon\), \(\infty\) is called pseudoinvariant of the triple \((p_1, p_2, p_3)\) in~\cite[Chapter~2,~2.12]{M-TNCG}.
\end{rem}

The next step is to study the set of all superconformal automorphisms of \(\ProjectiveSpace[_\C]{1|1}\) that preserve the three points \(0\), \(1_\epsilon\) and \(\infty\), or more generally map \(0\mapsto 0\), \(1_\epsilon\mapsto1_{\epsilon'}\) and \(\infty\mapsto\infty\).
\begin{prop}\label{prop:UniqueAutomorphismFixing3Points}
	Let \(\Xi\colon \ProjectiveSpace[\C]{1|1}\times B\to \ProjectiveSpace[\C]{1|1}\times B\) be a superconformal automorphism over \(B\) such that
	\begin{align}
		[0:1:0] &\mapsto [0:1:0] &
		[1:1:\epsilon] &\mapsto [1:1:\epsilon'] &
		[1:0:0] &\mapsto [1:0:0]
	\end{align}
	for some odd \(\epsilon\), \(\epsilon'{\left(\in\cO_B\otimes\C\right)}_1\).
	Then \(\epsilon=\pm\epsilon'\).
	If \(\epsilon=\epsilon'=0\) the automorphism \(\Xi\) is either the identity or \(\Xi_-\) from Example~\ref{ex:ReflectionOddDirections}.
	If \(\epsilon\neq 0\) and \(\epsilon=\epsilon'\) the automorphism \(\Xi\) is the identity.
	If \(\epsilon\neq 0\) and \(\epsilon= -\epsilon'\) the automorphism \(\Xi\) is \(\Xi_-\).
\end{prop}
\begin{proof}
	We write the automorphism \(\Xi\) in the superconformal coordinates \((z_1, \theta_1)\) as follows:
	\begin{align}
		\Xi^\# z_1 &= f(z_1) + \theta_1\zeta(z_1), &
		\Xi^\# \theta_1 &= \xi(z_1) + \theta_1g(z_1).
	\end{align}
	Here \(f\), \(g\) are even holomorphic functions and \(\xi\), \(\zeta\) are odd holomorphic functions satisfying
	\begin{align}
		f' &= g^2 - \xi\xi', &
		\zeta = \xi g.
	\end{align}
	A coordinate change to \((z_2, \theta_2)\) yields the holomorphic maps
	\begin{align}
		\begin{split}
			\Xi^\#z_2
			&= {\left(\Xi^\# z_1\right)}^{-1}
			= {\left( f(z_1) \left( 1+ \theta_1 \frac{\zeta(z_1)}{f(z_1)}\right)\right)}^{-1} \\
			&= {\left( f(\frac1{z_2}) \left( 1+ \theta_2 \frac{\zeta(\frac1{z_2})}{z_2 f(\frac1{z_2})}\right)\right)}^{-1}
			= \frac1{f(\frac1{z_2})} - \theta_2\frac{\zeta(\frac1{z_2})}{z_2{\left(f(\frac1{z_2})\right)}^2},
		\end{split} \\
		\begin{split}
			\Xi^\#\theta_2
			&= \frac{\Xi^\#\theta_1}{\Xi^\# z_1}
			= \frac{\xi(z_1) + \theta_1g(z_1)}{f(z_1) + \theta_1\zeta(z_1)} \\
			&= \left(\xi(\frac1{z_2}) + \theta_2\frac{g(\frac1{z_2})}{z_2}\right)\left(\frac1{f(\frac1{z_2})} - \theta_2\frac{\zeta(\frac1{z_2})}{z_2{\left(f(\frac1{z_2})\right)}^2}\right) \\
			&= \frac{\xi(\frac1{z_2})}{f(\frac1{z_2})} + \theta_2\frac{g(\frac1{z_2})}{z_2 f(\frac1{z_2})}.
		\end{split}
	\end{align}
	In the last step we have used that \(\xi\zeta= \xi\xi g = 0\).
	In order to show that \(f(z_1)=z_1\), let us expand the holomorphic function~\(f\) in power series around \(z_1=0\):
	\begin{equation}
		f(z_1) = \sum_{n=0}^\infty f_n z_1^n
	\end{equation}
	where the coefficients \(f_n\) are even sections of \(\cO_B\otimes\C\).
	The condition that \(\Xi\) preserves the point \([0:1:0]\) implies \(f(0)=0\), that is, \(f_0=0\).
	Furthermore, as the reduction \(\Red{\Xi}\colon \ProjectiveSpace[\C]{1}\to\ProjectiveSpace[\C]{1}\) preserves the three points \(0\), \(1\) and \(\infty\), the reduction of \(f(z_1)\) has to be~\(z_1\).
	That is, all coefficients \(f_n\) of \(f\) are even nilpotent with the exception of \(f_1\) which is of the form \(f_1=1+{f_1}_{nil}\) for an even nilpotent \({f_1}_{nil}\).
	Consequently, \(f_1\) is invertible and the Laurent series expansion of \({f(z_1)}^{-1}\) is of the form
	\begin{equation}
		{f(z_1)}^{-1}
		= {\left(z_1\left(\sum_{n=0}^\infty f_{n+1}z^n\right)\right)}^{-1}
		= \frac1{z_1}\sum_{n=0}^\infty \overline{f}_{n+1}z^n,
	\end{equation}
	where the coefficients \(\overline{f}_n\) are obtained inductively as follows:
	\begin{align}
		\overline{f}_1 &= \frac1{f_1} &
		\overline{f}_2 &= -\frac{f_2}{f_1^2} &
		\overline{f}_{n+1} &= -\frac1{f_1}\sum_{k=1}^{n} f_{k+1}\overline{f}_{n+1-k}.
	\end{align}
	The function
	\begin{equation}
		\frac1{f(\frac1{z_2})}
		= z_2\sum_{n=0}^\infty\overline{f}_{n+1} z_2^{-n}
		= \frac1{f_1}z_2 - \frac{f_2}{f_1^2} + \overline{f}_3\frac1{z_2} + \overline{f}_4\frac1{z_2^2} + \dotsb
	\end{equation}
	is holomorphic around \(z_2=0\), that is, the \(\overline{f}_{n+1}\) have to vanish for \(n\geq2\).
	The automorphism \(\Xi\) preserves the point \([1:0:0]\) and hence \(f_2=0\) which implies \(f_n=0\) for \(n>2\).

	A similar argument applies to \(\xi\).
	Let
	\begin{equation}
		\xi(z_1) = \sum_{n=0}^\infty \xi_n z_1^n
	\end{equation}
	be the power series expansion of \(\xi\).
	All coefficients \(\xi_n\) are odd sections of \(\cO_B\otimes\C\).
	The automorphism \(\Xi\) preserves the points \([0:1:0]\) and \([1:0:0]\) and hence \(\xi_0=0\) and
	\begin{equation}
		\frac{\xi(\frac1{z_2})}{f(\frac1{z_2})}
		= \left(\sum_{n=1}^\infty \xi_n z^{-n}\right)\frac1{f_1}z_2
		= \frac{\xi_1}{f_1} + \frac{\xi_2}{f_1}\frac1{z_2} + \frac{\xi_3}{f_1}\frac1{z_2^2} + \dotsb
	\end{equation}
	is holomorphic and vanishes at \(z_2=0\).
	This implies \(\xi=0\) and hence \(\zeta=0\).

	As \(f'(z_1)=f_1 = g^2\) we have \(g=\pm\sqrt{f_1}\).
	The fact that \(\Xi\) maps \([1:1:\epsilon]\) to \([1:1:\epsilon']\) implies now
	\begin{align}
		1 &= f_1, &
		\epsilon' &= \pm\epsilon\sqrt{f_1},
	\end{align}
	and the claim follows.
\end{proof}

\begin{cor}\label{Cor:AutomorphismsP11AllLinear}
	Any superconformal automorphism of \(\ProjectiveSpace[\C]{1|1}\times B\) over \(B\) is induced from a matrix in \(\SpGL_{\cO_B\otimes\C}(2|1)\).
\end{cor}
\begin{proof}
	Denote the preimages of the \(B\)-points \([0:1:0]\), \([1:1:0]\) and \([1:0:0]\) under the automorphism \(\Xi\) by \(p_1\), \(p_2\) and \(p_3\).
	By Proposition~\ref{prop:LinearMapPrescribed3Points} there is an automorphism \(l\) induced from \(L\in\SpGL_{\cO_B\otimes\C}(2|1)\) such that \(l\) maps
	\begin{align}
		[0:1:0] &\mapsto p_1, &
		[1:1:0] &\mapsto p_2, &
		[1:0:0] &\mapsto p_3.
	\end{align}
	By Proposition~\ref{prop:UniqueAutomorphismFixing3Points}, the map \(l\circ \Xi\) is either the identity or the reflection of the odd direction~\(\Xi_-\).
	Consequently, \(\Xi\) is either \(l^{-1}\) or \(l^{-1}\circ\Xi_-\) which are both induced by a matrix from \(\SpGL_{\cO_B\otimes\C}(2|1)\).
\end{proof}
Corollary~\ref{Cor:AutomorphismsP11AllLinear} is a special case of the result that any holomorphic, not necessarily superconformal, automorphism of \(\ProjectiveSpace[\C]{m|1}\) is induced by a linear map, see~\cite[Proposition~4.4]{FK-PLSGSPAP11}.

\subsection{Marked points on \texorpdfstring{\(\ProjectiveSpace[\C]{1|1}\)}{PC11}}\label{sec:MarkedPointsOnP11}
In this section we will describe the open moduli space \(\mathcal{M}_{0,k}\) of super Riemann surfaces of genus zero with \(k\geq 3\) distinct marked points as a superorbifold using the language developed in Section~\ref{sec:Superorbifolds}.

By the uniformization of super Riemann surfaces, see~\cites{CR-SRSUTT}[Theorem~9.4.1]{EK-SGSRSSCA}, we know that \(\ProjectiveSpace[\C]{1|1}\) is the only super Riemann surface of genus zero.
In particular, any family of super Riemann surfaces over \(B\) is the trivial family \(B\times \ProjectiveSpace[\C]{1|1}\to B\).
A marked point or Neveu--Schwarz puncture on \(\ProjectiveSpace[\C]{1|1}\) is a \(B\)-point \(B\to \ProjectiveSpace[\C]{1|1}\).
Hence, the moduli space of super Riemann surfaces of genus zero with \(k\) distinct marked points is the space of \(k\) distinct points on \(\ProjectiveSpace[\C]{1|1}\) up to automorphisms of \(\ProjectiveSpace[\C]{1|1}\).

More precisely, let
\begin{equation}
	Z_k \subset{\left(\ProjectiveSpace[\C]{1|1}\right)}^k
	=\underbrace{\ProjectiveSpace[\C]{1|1}\times\ProjectiveSpace[\C]{1|1}\times\dotsm \times\ProjectiveSpace[\C]{1|1}}_{k\text{ times}}
\end{equation}
be the open subsupermanifold where no two of the reduced projections coincide.
The \(\R^{0|0}\)-points \(\underline{Z_k}(\R^{0|0})\) are \(k\)-tuples of distinct points of \(\ProjectiveSpace[\C]{1}\).
The \(C\)-points \(\underline{Z_k}(C)\) are \(k\)-tuples of \(C\)-points of \(\ProjectiveSpace[\C]{1|1}\) such that their reduction are distinct.
The group \(\SpGL_{\C}(2|1)\) acts on \(Z_k\) diagonally.
This action is free and has as isotropy group \(\Z_2\) on all \(\R^{0|0}\)-points of \(Z_k\).
Consequently, by Theorem~\ref{thm:QuotientByFiniteIsotropyProperGroupAction}, we know that \(\faktor{Z_k}{\SpGL_{\C}(2|1)}\) is a superorbifold of real dimension
\begin{equation}
	\dim Z_k - \dim \SpGL_{\C}(2|1)
	= 2k|2k - 6|4
	= 2k-6|2k-4.
\end{equation}
The action of \(\SpGL_{\C}(2|1)\) on \(Z_k\) respects the complex structure and hence the quotient can be equipped with the structure of a complex orbifold.

To describe the orbifold \(\faktor{Z_k}{\SpGL_{\C}(2|1)}\) further, let \(\tilde{Z}_{k-3}\subset Z_{k-3}\) the open submanifold such that no reduced point coincides with \(0\), \(1\) or \(\infty\).
We claim that
\begin{equation}
	\faktor{Z_k}{\SpGL_{\C}(2|1)} \simeq \faktor{\left(\C^{0|1}\times \tilde{Z}_{k-3}\right)}{\Z_2}.
\end{equation}
This isomorphism is induced from the map \(f\colon \C^{0|1}\times \tilde{Z}_{k-3}\to Z_k\) given by
\begin{equation}
	\begin{split}
		\underline{f}(C)\colon \underline{\C^{0|1}\times \tilde{Z}_{k-3}}(C) &\to \underline{Z_k}(C) \\
		\left(\epsilon, (p_1, \dotsc, p_{k-3})\right) &\mapsto \left(0, 1_\epsilon, \infty, p_1, \dotsc, p_{k-3}\right)
	\end{split}
\end{equation}
Here, we parametrize \(C\)-points \(p\colon C\to \C^{0|1}\) by \(p^\#\eta=\epsilon\in{\left(\cO_C\otimes\C\right)}_1\), where \(\eta\) is a fixed coordinate of \(\C^{0|1}\).
The group \(\Z_2\) acts on \(\C^{0|1}\times \tilde{Z}_{k-3}\) by reflection of odd directions, see Example~\ref{ex:Z2ActionsOnRmn}, and the map \(f\) is \(\rho\)-equivariant, where \(\rho\colon \Z_2\to \SpGL_{\C}(2|1)\) is the group homomorphism that sends the generator of \(\Z_2\) to \(\Xi_-\).
Propositions~\ref{prop:LinearMapPrescribed3Points} and~\ref{prop:UniqueAutomorphismFixing3Points} show that the resulting homomorphism of super Lie groupoids
\begin{equation}
	F\colon \Z_2\ltimes\left(\C^{0|1}\times\tilde{Z}_{k-3}\right)\to \SpGL_{\C}(2|1)\ltimes Z_k
\end{equation}
is a weak equivalence.
This proves the following proposition:
\begin{prop}\label{prop:ModuliSpaceOfSRS0k}
	The moduli space of super Riemann surfaces of genus zero and \(k\geq 3\) marked points is an orbifold of real dimension \(2k-6|2k-4\) and given by
	\begin{equation}
		\mathcal{M}_{0,k}=\faktor{\left(\C^{0|1}\times\tilde{Z}_{k-3}\right)}{\Z_2}
	\end{equation}
	where \(\tilde{Z}_{k-3}\subset{\left(\ProjectiveSpace[\C]{1|1}\right)}^k\) is the open submanifold such that no two reductions of the projections coincide with each other nor with the points \(0\), \(1\) and \(\infty\).
	The \(\Z_2\)-action is given by the reflection in the odd directions.
\end{prop}
Note that by construction the space of \(\R^{0|0}\)-orbits \(\underline{\mathcal{M}_{0,k}}(\R^{0|0})\) is homeomorphic to~\(M_{0,k}\), the moduli space of Riemann surfaces of genus zero with \(k\) markings.
This ignores the \(\Z_2\)-isotropy group on the \(\R^{0|0}\)-points which multiplies the odd directions of a super Riemann surface by \(-1\).

\subsection{Stable supercurves of genus zero}\label{sec:StableSRSofGenusZero}
In this section, we will construct nodal supercurves of genus zero as trees of bubbles, similar to what is outlined in~\cite[Appendix~D.3]{McDS-JHCST}.
Every vertex of the tree represents one copy of \(\ProjectiveSpace[\C]{1|1}\) with marked points.
A sufficient number of markings guarantees that the nodal supercurve of genus zero is stable, that is, allows only a finite number of automorphisms.
We prove that the moduli space \(\mathcal{M}_{0,T}\) of stable supercurves modeled on a fixed tree \(T\) forms a superorbifold.

We will follow the notation of~\cite[Appendix~D]{McDS-JHCST}:
A tree is a connected graph without cycles.
We represent the tree by a set of vertices \(T\) and the edges as a subset \(E\subset T\times T\).
For \((\alpha, \beta)\in E\) we also write \(E_{\alpha\beta}\) and require that \(E_{\alpha\beta}\) if and only if \(E_{\beta\alpha}\), that is, the graph is undirected.
A tree satisfies \(\#T=\#E+1\).
A \(k\)-labeling is given by a map
\begin{equation}
	\begin{split}
		p\colon \Set{1, 2, \dotsc, k} &\to T \\
		i &\mapsto p_i
	\end{split}
\end{equation}
The \(k\)-labeled tree is called stable if for every vertex \(\alpha\) the number of outgoing edges plus the number of markings at \(\alpha\) is at least three:
\begin{equation}
	\#p^{-1}(\alpha) + \#\Set{\beta\in T\given E_{\alpha\beta}}\geq 3.
\end{equation}
A homomorphism of \(k\)-labeled trees \(f\colon (T, E, p)\to (\tilde{T}, \tilde{E}, \tilde{p})\) consist of a map \(f\colon T\to \tilde{T}\) such that for any edge \(E_{\alpha\beta}\) also \(\tilde{E}_{f(\alpha)f(\beta)}\) is an edge and \(\tilde{p}_i = f(p_i)\).
A homomorphism of \(k\)-labeled trees is called isomorphism if it is invertible.
In the following we will identify isomorphic labeled trees.
That is, when we speak of a \(k\)-labeled tree, we actually think of its isomorphism class.

\begin{defn}\label{defn:NodalSuperCurve}
	Let \(T\) be a \(k\)-labeled tree as above.
	A nodal supercurve of genus zero over \(B\), modeled on \(T\) is a tuple
	\begin{equation}
		\bm{z} = \left({\Set{z_{\alpha\beta}}}_{E_{\alpha\beta}}, {\Set{z_i}}_{1\leq i\leq k}\right)
	\end{equation}
	consisting of \(B\)-points \(z_{\alpha\beta}\colon B\to \ProjectiveSpace[\C]{1|1}\) and \(z_i\colon B\to \ProjectiveSpace[\C]{1|1}\) such that for every \(\alpha\in T\) the reduction of the points \(z_{\alpha\beta}\) and \(z_i\) for \(p(i)=\alpha\) are disjoint.
	The \(z_{\alpha\beta}\) are called nodal points and \(z_i\) are marked points.
	Together the nodal points and marked points form the set of special points
	\begin{equation}
		Y_\alpha = \Set{z_{\alpha\beta}\given E_{\alpha\beta}}\cup \Set{z_i\given p(i)=\alpha}.
	\end{equation}

	The nodal supercurve of genus zero modeled on \(T\) is called a stable supercurve modeled on \(T\) if \(T\) is a stable \(k\)-labeled tree.
\end{defn}
We think of a nodal supercurve of genus zero modeled on \(T\) as a tree of projective superspaces \(\ProjectiveSpace[\C]{1|1}\) such that every vertex corresponds to a separate copy of \(\ProjectiveSpace[\C]{1|1}\).
If \(E_{\alpha\beta}\) is an edge, the point \(z_{\alpha\beta}\) of the \(\ProjectiveSpace[\C]{1|1}\) touches the point \(z_{\beta\alpha}\) in the copy of \(\ProjectiveSpace[\C]{1|1}\) at \(\beta\).
The marked point \(p_i\) is in the copy of \(\ProjectiveSpace[\C]{1|1}\) sitting at \(\alpha=p(i)\).
Note that the condition of stability is equivalent to \(\#Y_\alpha\geq 3\) for all \(\alpha\in T\).

\begin{defn}\label{defn:ReparametrizationGroupOfNodalCurves}
	We call the super Lie group
	\begin{equation}
		G^T=\prod_{\alpha\in T} \SpGL_{\C}(2|1)
	\end{equation}
	the reparametrization group of nodal supercurves of genus zero modeled on~\(T\).
	An element \(\bm{g}=\Set{g_\alpha}\in\underline{G^T}(B)\) acts on a supercurve \(\bm{z}=(\Set{z_{\alpha\beta}}, \Set{z_i})\) of genus zero over \(B\) modeled on \(T\) by
	\begin{equation}
		\bm{g}(\bm{z}) = \left(\Set{g_\alpha(z_{\alpha\beta})}, \Set{g_{p(i)}(z_i)}\right).
	\end{equation}

	Two stable supercurves over \(B\) of genus zero modeled on \(T\) are considered to be equivalent if they differ by a reparametrization.
	We denote the set of equivalence classes by \(\underline{\mathcal{M}_{0,T}}(B)\).
\end{defn}

The stability condition implies that a stable supercurve of genus zero modeled on \(T\) has only finitely many automorphisms, that is, equivalence transformations fixing the special points.
More precisely, the automorphism group of a stable supercurve of genus zero modeled on \(T\) is a subgroup of \(\Z_2^{\#T}\) because every vertex contains at least three special points and hence the remaining automorphism group at that vertex is either \(\Z_2\) or the trivial group.

The equivalence classes \(\underline{\mathcal{M}_{0,T}}(B)\) of stable supercurves over \(B\) of genus zero modeled on \(T\) form a functor
\begin{equation}
	\underline{\mathcal{M}_{0,T}}\colon \cat{SPoint}^{op}\to \cat {Sets}.
\end{equation}
Moreover, as the group of equivalence transformations acts diagonally on the different nodes, we can equip this functor with the structure of a superorbifold
\begin{equation}
	\mathcal{M}_{0,T} = \prod_{\alpha\in T} \mathcal{M}_{0, \#Y_\alpha}.
\end{equation}
Here, every factor is a moduli space of supercurves of genus zero with \(\#Y_\alpha\geq 3\) markings.
Consequently, \(\mathcal{M}_{0,T}\) is a quotient superorbifold of dimension
\begin{equation}
	\begin{split}
		\dim \mathcal{M}_{0,T}
		&= \sum_{\alpha\in T} \dim\mathcal{M}_{0, \#Y_\alpha}
		= \sum_{\alpha\in T} \left(2\#Y_\alpha - 6\right) | \sum_{\alpha\in T} \left(2\#Y_\alpha - 4\right) \\
		&= 4\#E + 2k - 6\#T | 4\#E + 2k - 4\#T \\
		&= 2k - 6 - 2\#E| 2k-4
	\end{split}
\end{equation}
and isotropy groups of the form \(\Z_2^{\# T}\) at all \(\R^{0|0}\) points.
The topological space \(\underline{\mathcal{M}_{0,T}}(\R^{0|0})\) is homeomorphic to the moduli space \(M_{0,T}\) of Riemann surfaces of genus zero modeled on the labeled tree \(T\) as discussed in~\cite[Appendix~D.3]{McDS-JHCST}.

We now define the set-valued point functor \(\underline{\overline{\mathcal{M}}_{0,k}}\) of the moduli space of stable supercurves of genus zero by taking the union over all isomorphism classes of \(k\)-labeled trees:
\begin{equation}\label{eq:DefinitionM0kBar}
	\begin{split}
		\underline{\overline{\mathcal{M}}_{0,k}}\colon \cat{SPoint}^{op}&\to \cat{Sets} \\
		C &\mapsto \bigcup_{T\text{ \(k\)-labeled tree}} \underline{\mathcal{M}_{0,T}}(C)
	\end{split}
\end{equation}
Points \(p\in\underline{\overline{\mathcal{M}}_{0,k}}(C)\) correspond to stable supercurves over \(C\) of genus zero with \(k\) markings.
Any such stable supercurve, up to equivalence, is represented by a unique \(C\)-point.
By construction, \(\underline{\overline{\mathcal{M}}_{0,k}}(\R^{0|0})\) is in bijection to the moduli space~\(\overline{M}_{0,k}\) of stable Riemann surfaces of genus zero and \(k\) marked points.
Note also that the union in Equation~\eqref{eq:DefinitionM0kBar} is a finite union because the number of stable \(k\)-labeled trees is finite.

We now refine the functor \(\underline{\overline{\mathcal{M}}_{0,k}}\) to take values in topological spaces.
The following is a generalization of the Gromov convergence for stable curves, see~\cite[Definition~D.5.1]{McDS-JHCST} to the case of super stable curves:
\begin{defn}\label{defn:GromovConvergenceStableCurves}
	A sequence \(\bm{z}^\nu=(z_{\alpha\beta}^\nu, z_i^\nu)\) of stable supercurves over \(B\) of genus zero modeled on \(T^\nu\) is said to Gromov converge to to a stable supercurve \(\bm{z}=(z_{\alpha\beta}, z_i)\) over \(B\) of genus zero and modeled on \(T\) if for \(\nu\) sufficiently large there exists a tree homomorphism \(f^\nu\colon T\to T^\nu\) and a collection of reparametrizations \(\bm{g}^\nu = {\Set{g^\nu_\alpha}}_{\alpha\in T}\in\underline{G^{T}}(B)\) sucht that the following hold:
	\begin{description}
		\item[(Rescaling)]
			If \(\alpha\), \(\beta\in T\) are connected by an edge \(E_{\alpha\beta}\) and \(\nu_j\) is a subsequence such that \(f^{\nu_j}(\alpha) = f^{\nu_j}(\beta)\) then the sequence
			\begin{equation}
				\underline{{\left(g_\alpha^{\nu_j}\right)}^{-1}\circ g_\beta^{\nu_j}}(C)\colon \underline{\ProjectiveSpace[\C]{1|1}}(C)\to \underline{\ProjectiveSpace[\C]{1|1}}(C)
			\end{equation}
			converges uniformly on compact subsets to \(C_B^*z_{\alpha\beta}\in\underline{\ProjectiveSpace[\C]{1|1}}(C)\) for all \(C_B\colon C\to B\) and \(j\to\infty\).
		\item[(Nodal points)]
			If \(\alpha\), \(\beta\in T\) are connected by an edge \(E_{\alpha\beta}\) and \(\nu_j\) is a subsequence such that \(f^{\nu_j}(\alpha)\neq f^{\nu_j}(\beta)\) then
			\begin{equation}
				z_{\alpha\beta} = \lim_{j\to\infty}{\left(g_\alpha^{\nu_j}\right)}^{-1} (z^{\nu_j}_{f^{\nu_j}(\alpha)f^{\nu_j}(\beta)})
				\in\underline{\ProjectiveSpace[\C]{1|1}}(B).
			\end{equation}
		\item[(Marked Points)]
			For all~\(1\leq i\leq k\) it holds \(p^\nu(i)=f^\nu(p(i))\) and
			\begin{equation}
				z_i = \lim_{\nu\to\infty}{(g_{p(i)}^\nu)}^{-1}(z_i^\nu)
				\in \underline{\ProjectiveSpace[\C]{1|1}}(B).
			\end{equation}
	\end{description}
\end{defn}
Gromov convergence is defined up to reparametrizations.
That is, if \(\bm{z}^\nu\) Gromov converges to \(\bm{z}\) also \(\bm{g}^\nu\cdot\bm{z}^\nu\) converges to \(\bm{g}\cdot\bm{z}\) for arbitrary sequences \(\bm{g}^\nu\in \underline{G^{T^\nu}}(B)\) and \(\bm{g}\in\underline{G^T}(B)\).
Hence, we say that a subset \(U\in \underline{\overline{\mathcal{M}}_{0,k}}(C)\) is Gromov closed if the limit of any Gromov convergent sequence \(\bm{z}^\nu\in U\) also lies in \(U\).
The Gromov closed subsets define a topology, called Gromov topology, on \(\underline{\overline{\mathcal{M}}_{0,k}}(C)\) such that any Gromov converging sequence converges also with respect to the Gromov topology.
As the definition of Gromov convergence and hence the definition of Gromov topology is functorial, we obtain a functor
\begin{equation}
	\underline{\overline{\mathcal{M}}_{0,k}}\colon \cat{SPoint}^{op}\to \cat{Top} \\
\end{equation}
taking values in the category of topological spaces which satisfies:
\begin{itemize}
	\item
		By construction, the Gromov topology on \(\underline{\overline{\mathcal{M}}_{0,k}}(\R^{0|0})\) coincides with the Gromov topology constructed in~\cite[Appendix D.5]{McDS-JHCST}.
		Consequently, \(\underline{\overline{\mathcal{M}}_{0,k}}(\R^{0|0})\) is compact.
	\item
		The restriction of the Gromov topology on \(\underline{\overline{\mathcal{M}}_{0,k}}(C)\) to \(\underline{\mathcal{M}_{0,T}}(C)\) is equivalent to the topology obtained from the orbit functor of the superorbifold \(\mathcal{M}_{0,T}\) for all \(k\)-labeled trees \(T\) and superpoints \(C\).
\end{itemize}
We expect that \(\underline{\overline{\mathcal{M}}_{0,k}}\) is the orbit functor of a superorbifold.
The moduli space of stable supercurves of arbitrary genus has been constructed as a smooth super Deligne--Mumford stack in~\cites{D-LaM}{FKP-MSSCCLB} with methods of algebraic geometry.

\begin{rem}
	We conjecture uniqueness up to reparametrization of the limit of a Gromov converging sequence of super stable curves also in the case \(C\neq\R^{0|0}\).
	Uniqueness of the limit would imply that \(\underline{\overline{\mathcal{M}}_{0,k}}(C)\) is Hausdorff and that convergence with respect to the Gromov topology is equivalent to Gromov convergence, see~\cite[Lemma~5.6.4]{McDS-JHCST}.

	We do not expect the spaces \(\underline{\overline{\mathcal{M}}_{0,k}}(C)\) to be compact for \(C\neq\R^{0|0}\).
	The space of higher superpoints is not compact, even if the underlying topological space of the supermanifold in the ringed space approach is compact.
	This can already be seen for
	\begin{equation}
		\overline{\mathcal{M}}_{0,3}
		=\mathcal{M}_{0,3}
		=\faktor{\C^{0|1}}{\Z_2}.
	\end{equation}
\end{rem}
\begin{rem}[Neveu--Schwarz punctures and Ramond punctures]\label{rem:NSPuncturesAndRPunctures}
	In~\cites{D-LaM}{FKP-MSSCCLB} and also in the more physics oriented literature, for example in~\cite{W-NSRSTM}, two types of punctures or marked points are considered:
	Neveu--Schwarz punctures on a super Riemann surface \(M\) are points \(p\colon B\to M\) and given with respect to local superconformal coordinates \((z, \theta)\) by \(z=a\) and \(\theta=\alpha\) for some \(a\in{\left(\cO_{B}\right)}_0\) and \(\alpha\in {\left(\cO_B\right)}_1\).
	Ramond punctures are of a different type.
	They are divisors of complex codimension \(1|0\),  where the map \(\cD\otimes \cD\to \faktor{\tangent{M}}{\cD}\) is allowed to be not injective.
	More precisely it is required that there is an isomorphism
	\begin{equation}
		\cD\otimes \cD\simeq \faktor{\tangent{M}}{\cD}\left(-R\right),
	\end{equation}
	where \(R\) is the divisor given by the disjoint union of all Ramond punctures.
	Locally, a Ramond puncture is given by the equation \(z=a\) for some \(a\in{\left(\cO_B\right)}_0\) and the distribution \(\cD\) is generated by the vector field \(\partial_\theta + \theta\left(z-a\right)\partial_z\).
	In~\cites{D-LaM} it is shown that both Neveu--Schwarz nodes and Ramond nodes can arise under deformations of super Riemann surfaces.

	Let us give a brief argument why in the case of supercurves of genus zero with only Neveu--Schwarz punctures no Ramond nodes can arise.
	Assume that \(M\) a nodal supercurve of genus zero with \(k\) Neveu--Schwarz marked points where some of the nodes are Ramond nodes.
	Then \(M\) has a decomposition into a tree of irreducible components, all of which are of genus zero and some of them carry Ramond punctures as nodal points.
	As all nodes are simple nodes we must have a vertex of this tree that has a single Ramond puncture.
	But this is not possible because the number of Ramond punctures has to be even:
	Let \(M\) be an irreducible super Riemann surface of genus \(p\) with Ramond punctures and \(j\colon \Red{M}\to M\) the canonical map from the reduced space.
	Then,
	\begin{equation}
		j^*\left(\faktor{\tangent{M}}{\cD}\right) = \tangent{\Red{M}}
	\end{equation}
	and hence for \(L=j^*\cD\) it holds
	\begin{equation}\label{eq:TwistedSpinorBundle}
		L\otimes L = \tangent{\Red{M}}(-\Red{R})
	\end{equation}
	and
	\begin{equation}
		2\deg L = 2 - 2p - \deg \Red{R},
	\end{equation}
	that is the number of points in \(\Red{R}\) is even.
	For a precise discussion of the moduli space of pairs \((\Red{M}, L)\) satisfying the condition~\eqref{eq:TwistedSpinorBundle}, see~\cite{AJ-MTSC}.

	The moduli space of supercurves of genus zero with only Ramond punctures has been studied in~\cite{OV-SMSGZSUSYCRP}.
\end{rem}

\begin{rem}
	There are several essentially combinatoric constructions that can be extended to the case of stable supercurves of genus zero:
	There are projections \(\pi_k\colon \overline{\mathcal{M}}_{0,k}\to \overline{\mathcal{M}}_{0,k-1}\) which delete the \(k\)-th marked point and collapse the non-stable vertices.
	The map \(\pi_k\) has canonical sections \(\sigma_j\) which double the \(j\)-th marked point.
	Furthermore, if we have a supercurve of genus zero modeled on a tree~\(T\) which is not stable, it can be collapsed to a stable supercurve of genus zero, modeled on a different tree with less vertices.
	For details in the classical case, we refer again to~\cite[Appendix~D]{McDS-JHCST}.
\end{rem}

\section{Super stable maps of genus zero}\label{Sec:SuperStableMapsOfGenusZero}
In this section we generalize the notion of stable maps of genus zero to super Riemann surfaces and study their moduli spaces.
Super stable maps are realized as maps from nodal supercurves to an almost complex manifold, that is a tree of super \(\targetACI\)-holomorphic curves of genus zero with conditions on the nodes.
For the non-super theory, we refer to~\cite[Chapter~5 and Chapter~6]{McDS-JHCST}.

In Section~\ref{SSec:SuperJHolomorphicCurvesOfGenusZero} we recall the notion of super \(\targetACI\)-holomorphic curve and their properties from~\cite{KSY-SJC}.
Super stable maps and their equivalence classes are then defined in Section~\ref{SSec:DefinitionOfSuperStableMapsOfGenusZero}.
In Section~\ref{SSec:SuperStableMapsOfFixedTreeType} we show that the moduli space of simple super stable maps of genus zero forms a global quotient superorbifold under certain conditions on the target.
In general, sequences of super stable maps of genus zero may converge to a limit of different tree type.
This supergeometric extension of bubbling is captured in the Gromov topology on super stable maps, defined in Section~\ref{SSec:GromovTopologyOnM0kABar}.

\subsection{Super \texorpdfstring{\(\targetACI\)}{J}-holomorphic curves of genus zero}\label{SSec:SuperJHolomorphicCurvesOfGenusZero}
We fix a classical symplectic manifold~\(N\) of dimension~\(2n\) with symplectic form~\(\omega\) and compatible almost complex structure~\(\targetACI\).
In~\cite{KSY-SJC}, we have introduced the concept of super \(\targetACI\)-holomorphic curves from an arbitrary super Riemann surface \(M\) to \(N\):
A map \(\Phi\colon M\to N\) of supermanifolds over the base \(C\) is called a super \(\targetACI\)-holomorphic curve if
\begin{equation}
	\DJBar\Phi = \frac12\left.\left(1 + \ACI\otimes\targetACI\right)\right|_\cD \differential{\Phi}
\end{equation}
vanishes.
Here \(\cD\subset\tangent{M}\) is the completely non-integrable distribution of complex rank~\(0|1\) defining the super Riemann surface and \(\ACI\) the almost complex structure on \(M\).

In this paper we focus on the case of genus zero, that is \(M=\ProjectiveSpace[\C]{1|1}\).
There exists a holomorphic embedding \(i\colon \ProjectiveSpace[\C]{1}\to\ProjectiveSpace[\C]{1|1}\) that reduces to the identity on the topological spaces.
With respect to this embedding the super Riemann surface structure on \(\ProjectiveSpace[\C]{1|1}\) is described by the round metric on \(\ProjectiveSpace[\C]{1}=S^2\), the unique spinor bundle \(S\) and vanishing gravitino \(\chi=0\).
Decomposing the map \(\Phi\) into component fields
\begin{align}
	\varphi &= \Phi\circ i\colon \ProjectiveSpace[\C]{1}\times C \to N &
	\psi &= \left.i^*\differential{\Phi}\right|_\cD \in \VSec{\dual{S}\otimes \varphi^*\tangent{N}} &
	F &= i^*\DLaplace \Phi \in\VSec{\varphi^*\tangent{N}}
\end{align}
the condition for \(\Phi\) to be a super \(\targetACI\)-holomorphic curve can be rewritten as
\begin{align}
	0 &= \DelJBar\varphi + \frac14\Tr_{\dual{g}_S}\left(\gamma\otimes\mathfrak{j}\targetACI\right)\psi, &
	0 &= F, \\
	0 &= \left(1+\ACI\otimes\targetACI\right)\psi &
	0 &= \Dirac\psi - \frac13SR^N(\psi),
\end{align}
see~\cite[Corollary~2.5.3]{KSY-SJC}
Here \(\mathfrak{j}\) is the derivative of \(\targetACI\) in the direction of \(\cD\) and \(SR^N(\psi)\) is a term obtained from the Riemannian curvature on \(N\) that is cubic in \(\psi\).
Both \(\mathfrak{j}\) and \(SR^N(\psi)\) vanish when \(N\) is Kähler.
In that case, \(\varphi\) is a \(\targetACI\)-holomorphic curve with a holomorphic twisted spinor \(\psi\) parametrized over \(C\).
In any case, the reduction \(\phi=\Red{\Phi}\colon \ProjectiveSpace[\C]{1}\to N\) is a \(\targetACI\)-holomorphic curve.

Under certain transversality conditions on a super \(\targetACI\)-holomorphic curve \(\Phi\colon \ProjectiveSpace[\C]{1|1}\to N\) one can use the implicit function theorem to construct a family of super \(\targetACI\)-holomorphic curves in the neighborhood of \(\Phi\) as a subsupermanifold of the space of all maps \(\ProjectiveSpace[\C]{1|1}\to N\).
Here the transversality condition requires that both
\begin{align}
	\begin{split}\label{eq:DefnDphi}
		D_{\phi}\colon \VSec{\phi^*\tangent{N}}&\to {\VSec{\cotangent{\ProjectiveSpace[\C]{1}}\otimes\phi^*\tangent{N}}}^{0,1} \\
		\xi&\mapsto\frac12\left(1+\ACI\otimes\targetACI\right)\left(\nabla\xi - \frac12\id_{\cotangent{\ProjectiveSpace[\C]{1}}}\otimes\left(\targetACI\left(\nabla_\xi \targetACI\right)\right)\differential\phi\right)
	\end{split} \\
	\begin{split}\label{eq:DefnDirac01}
		\Dirac^{0,1}\colon {\VSec{\dual{S}\otimes\phi^*\tangent{N}}}^{0,1} &\to \VSec{\dual{S}\otimes_\C\phi^*\tangent{N}} \\
		\zeta^{0,1} &\mapsto \frac12\left(1 - \ACI\otimes\targetACI\right)\Dirac\zeta^{0,1}
	\end{split}
\end{align}
are surjective.
Here \(\phi=\Red{\Phi}\) is the reduced map, \(\nabla\) the Levi-Civita connection on \(\ProjectiveSpace[\C]{1}\) and \(\Dirac\) the twisted Dirac operator.

For certain almost complex manifolds \(N\), for example \(N=\ProjectiveSpace[\C]{n}\), one can show that the transversality conditions hold for all super \(\targetACI\)-holomorphic curves.
In that case one can construct the moduli space \(\mathcal{M}_0(A)\) of super \(\targetACI\)-holomorphic curves representing the homology class \(A\in H_2(N, \Integers)\).
Its functor of points is given by
\begin{equation}\label{eq:DimensionOfModuliSpaceSuperJHolomorphicCurves}
	\begin{split}
		\underline{\mathcal{M}_0(A)}\colon \cat{SPoint}^{op} &\to \cat{Man} \\
		C &\mapsto \Set{\Phi\in\Hom_C(\ProjectiveSpace[\C]{1|1}, N)\given \DJBar\Phi = 0} \\
		\left(c: C'\to C\right) &\mapsto \left(\Phi\mapsto c^*\Phi\right)
	\end{split}
\end{equation}
and represents a supermanifolds of dimension
\begin{equation}
	2n + 2\left<c_1(TN),A\right>| 2\left<c_1(TN), A\right>.
\end{equation}
Alternatively, if not all maps to a certain target manifold are transversal, one can restrict to simple super \(\targetACI\)-holomorphic curves.
Here we call a super \(\targetACI\)-holomorphic curve simple if its reduction \(\phi\) is simple, that is, \(\phi\) cannot be decomposed into a \(\targetACI\)-holomorphic curve \(\tilde{\phi}\) and a branched holomorphic covering \(h\colon \ProjectiveSpace[\C]{1}\to \ProjectiveSpace[\C]{1}\) such that \(\phi=\tilde{\phi}\circ h\).
The Bochner method shows that \(\Dirac^{0,1}\) is surjective under certain conditions on the curvature.
It is then known that a generic perturbation of the almost complex structure assures surjectivity of \(D_\phi\) for all simple \(\targetACI\)-holomorphic curves.
Consequently, in those cases one can construct the moduli space \(\mathcal{M}_0^*(A)\) of simple super \(\targetACI\)-holomorphic curves as a supermanifold with the same dimension~\eqref{eq:DimensionOfModuliSpaceSuperJHolomorphicCurves}.

The super Lie group \(\SpGL_\C(2|1)\) of superconformal automorphisms of \(\ProjectiveSpace[\C]{1|1}\) acts on the moduli space of super \(\targetACI\)-holomorphic curves of genus zero:
For \(g\in\underline{\SpGL_\C(2|1)}(C)\) and \(\Phi\in\underline{\mathcal{M}_0(A)}(C)\) we define \(g\cdot \Phi=\Phi\circ g^{-1}\).
This action is proper and if \(A\neq 0\) the only fix-points are the \(\R^{0|0}\)-points which are invariant under the reflection in the odd directions~\(\Xi_-\).
Every constant map is invariant under every automorphisms of \(\ProjectiveSpace[\C]{1|1}\).

Super \(\targetACI\)-holomorphic curves over \(C\) are critical points of the superconformal action
\begin{equation}
	\begin{split}
		A(\Phi)
		&= \int_{\ProjectiveSpace[\C]{1|1}/B} {\norm{\left.\differential{\Phi}\right|_{\cD}}}^2[\d{vol}] \\
		&= \int_{\ProjectiveSpace[\C]{1}} \norm{\differential{\phi}}^2 + \left<\psi, \Dirac \psi\right> - \norm{F}^2 - \frac16\left<SR^N(\psi), \psi\right>\d{vol},
	\end{split}
\end{equation}
which takes values in \(\underline{\R}(C)={\left(\cO_C\right)}_0\).
The Berezin integral is independent of the chosen superconformal metric on \(\ProjectiveSpace[\C]{1|1}\), but the reduction of the Berezin integral to an integral over \(\ProjectiveSpace[\C]{1}\) depends on the embedding \(i\colon \ProjectiveSpace[\C]{1}\to\ProjectiveSpace[\C]{1|1}\) discussed above.
For more details on the superconformal action, we refer to~\cite[Chapter~12]{EK-SGSRSSCA}.
If \(\Phi\) is a super \(\targetACI\)-holomorphic curve over \(C=\R^{0|0}\), we have \(\psi=0\) and \(F=0\) and the superconformal action reduces to the harmonic action
\begin{equation}
	A(\Phi)
	= \int_{\ProjectiveSpace[\C]{1}} \norm{\differential{\phi}}^2 \d{vol}
= \left<[\omega], \phi_*[\ProjectiveSpace[\C]{1}]\right>,
\end{equation}
which equals to the pairing of the symplectic structure with the homology class of the image.

\subsection{Definition of super stable maps of genus zero}\label{SSec:DefinitionOfSuperStableMapsOfGenusZero}
In this section we define super stable maps and their equivalence classes.
Intuitively, a super stable map is a map from a nodal supercurve into the almost Kähler manifold~\(N\).
Recall from Definition~\ref{defn:NodalSuperCurve} that we represent a nodal supercurve of genus zero modeled on the \(k\)-labeled tree \(T=(T, E, p)\) by \(\bm{z}=\left(\Set{z_{\alpha\beta}}_{E_{\alpha\beta}}, \Set{z_i}_{1\leq i\leq k}\right)\) where \(z_{\alpha\beta}\) are the nodal points on the copy of \(\ProjectiveSpace[\C]{1|1}\) at \(\alpha\in T\) and \(z_i\) the marked points at \(\alpha=p(i)\).
The reduction of the special points at the node \(\alpha\), that is the union of nodes and marked points at \(\alpha\), is required to be distinct.
\begin{defn}\label{defn:SuperStableMap}
	A super stable map over \(B\) of genus zero into \(N\) modeled over the \(k\)-labeled tree \(T\) is a tuple
	\begin{equation}
		(\bm{z}, \bm{\Phi}) = \left(\left(\Set{z_{\alpha\beta}}_{E_{\alpha\beta}}, \Set{z_i}_{1\leq i\leq k}\right), \Set{\Phi_\alpha}_{\alpha\in T}\right)
	\end{equation}
	consisting of a nodal supercurve \(\bm{z}=\left(\Set{z_{\alpha\beta}}_{E_{\alpha\beta}}, \Set{z_i}_{1\leq i\leq k}\right)\) over \(B\) of genus zero modeled on \(T\) and super \(\targetACI\)-holomorphic curves \(\Phi_\alpha\colon \ProjectiveSpace[\C]{1|1}\times B\to N\) such that the following are satisfied:
	\begin{description}
		\item[(Nodes)]
			For all \(\alpha\), \(\beta\in T\) with edges \(E_{\alpha\beta}\) it holds \(\Phi_\alpha \circ z_{\alpha\beta} = \Phi_\beta \circ z_{\beta\alpha}\).
		\item[(Stability)]
			The number of special points is at least three for every vertex \(\alpha\) such that \(\Red{\left(\Phi_\alpha\right)}\) is constant.
	\end{description}
\end{defn}
We point out that the stability condition of a super stable map \((\bm{z},\bm{\Phi})\) does not imply that the nodal curve \(\bm{z}\) is stable because the stability condition only applies to vertices \(\alpha\in T\) where the map \(\Phi_\alpha\) is constant.
The definition of super stable maps over \(B=\R^{0|0}\) coincides with the usual definition of stable maps as given in~\cite[Definition~5.1.1]{McDS-JHCST}.
Furthermore Definition~\ref{defn:SuperStableMap} is functorial in \(B\).
That is, for a map \(C_B\colon C\to B\) and a super stable map \((\bm{z}, \bm{\Phi})\) over \(B\) there is a super stable map over \(C\) modeled over the same tree given by
\begin{equation}
	C_B^*(\bm{z}, \bm{\Phi})
	= \left(\left(\Set{z_{\alpha\beta}\circ C_B}, \Set{z_i\circ C_B}\right), \Set{\Phi_\alpha\circ\left(\id_{\ProjectiveSpace[\C]{1|1}}\times C_B\right)}\right).
\end{equation}

The following lemma gives an understanding of the condition at nodes in terms of component fields:
\begin{lemma}
	There is a bijection between the set of \(C\)-points of \(\ProjectiveSpace[\C]{1|1}\) and the set of tuples
	\begin{equation}
		\Set{(\Smooth{p}, s)\given \Smooth{p}\in \underline{\ProjectiveSpace[\C]{1}}(C)\text{, }s\in\VSec{\Smooth{p}^*S}}.
	\end{equation}

	Let \(\Phi\colon \ProjectiveSpace[\C]{1|1}\to N\) and be a super \(\targetACI\)-holomorphic curve with component fields \((\varphi, \psi, F=0)\).
	We introduce the shorthands \(\phi=\Red{\Phi}\), \(\varphi_{\Smooth{p}}=\varphi\circ\Smooth{p}\) and \(\psi_{\Smooth{p}}=\Smooth{p}^*\psi\).
	Then, for any coordinate system \(y^a\) around \(\phi(\Red{p})\) we have
	\begin{equation}
		{\left(\Phi\circ p\right)}^\# y^a = \varphi_{\Smooth{p}}^\#y^a + \left<s, \psi_{\Smooth{p}}^a\right> + \left<s, \psi_{\Smooth{p}}^b\right>\left<s, \psi_{\Smooth{p}}^c\right>\varphi^\#\Gamma_{bc}^a.
	\end{equation}
	Here, \(\psi^a\in\VSec{\Smooth{p}^*\dual{S}}\) are the coefficients of \(\psi_{\Smooth{p}}=\psi_{\Smooth{p}^a}\otimes\partial_{y^a}\) and the Christoffel symbols of \(N\) are denoted by \(\Gamma\), that is, \(\nabla^N_{\partial_{Y^b}}\partial_{Y^c} = \Gamma_{bc}^a\partial_{Y^a}\).
\end{lemma}
\begin{proof}
	Pick Wess--Zumino coordinates \((x^a, \eta^\alpha)\) for \(\ProjectiveSpace[\C]{1|1}\) around \(\Red{p}\) with respect to some superconformal metric on \(\ProjectiveSpace[\C]{1|1}\), see~\cite[Chapter~11]{EK-SGSRSSCA}.
	Recall that in particular Wess--Zumino coordinates satisfy \(i^\#x^a=x^a\) and \(i^\#\eta^\alpha=0\).
	In those coordinates the point \(p\) is given by
	\begin{align}
		p^\# x^a &= p^a, &
		p^\# \eta^\alpha &= p^\alpha,
	\end{align}
	for some even elements \(p^a\cO_C\) and odd elements \(p^\alpha\in\cO_C\).
	We define \(\Smooth{p}\) and \(s\) by
	\begin{align}
		\Smooth{p}^\#x^a &= p^a, &
		s = p^\alpha \left(\Smooth{p}^*s_\alpha\right),
	\end{align}
	where \(s_\alpha = i^*\partial_{\eta^\alpha}\) is the local spinor frame induced by the Wess--Zumino coordinates \((x^a, \eta^\alpha)\).
	The bijection between \(C\)-points of \(\ProjectiveSpace[\C]{1|1}\) and tuples of the form \((\Smooth{p}, s)\) is obvious.
	The definition of \(\Smooth{p}\) and \(s\) does not depend on the actual choice of Wess--Zumino coordinates because the condition \(i^\#\eta^\alpha=0\) is preserved under change of Wess--Zumino coordinates.

	The coordinate expansion of \(\Phi\) in Wess--Zumino coordinates is given by
	\begin{equation}
		\Phi^\#y^a = \varphi^\#y^a + \eta^\alpha \tensor[_\alpha]{\psi}{^a} + \eta^\mu\eta^\nu \tensor[_\mu]{\psi}{^b}\tensor[_\nu]{\psi}{^c}\Gamma_{bc}^a,
	\end{equation}
	see~\cite[Chapter~12.3]{EK-SGSRSSCA}.
	Recall that for super \(\targetACI\)-holomorphic curves the component field \(F\) vanishes.
	The result follows by taking pullback along \(\Smooth{p}\).
\end{proof}
Assume now that \(\alpha\) and \(\beta\in T\) are connected by an edge, and represent the nodal points \(z_{\alpha\beta}\) and \(z_{\beta\alpha}\) of a super stable map \((\bm{z},\bm{\Phi})\) by \((\Smooth{z}_{\alpha\beta}, s_{\alpha\beta})\) and \((\Smooth{z}_{\beta\alpha}, s_{\beta\alpha})\).
If we write \((\varphi_\alpha, \psi_\alpha, F_\alpha)\) and \((\varphi_\beta, \psi_\beta, F_\beta)\) for the component fields of \(\Phi_\alpha\) and \(\Phi_\beta\) the condition on the nodes of stable maps \(\Phi_\alpha\circ z_{\alpha\beta} = \Phi_\beta\circ z_{\beta\alpha}\) reads in a coordinate system \(y^a\) of \(N\) around \(\Red{\left(\Phi_\alpha\circ z_{\alpha\beta}\right)}\) as
\begin{equation}
	\begin{split}
		\MoveEqLeft
		{\left(\varphi_\alpha\right)}_{\Smooth{z}_{\alpha\beta}}^\#y^a + \left<s_{\alpha\beta}, {\left(\psi_\alpha\right)}_{\Smooth{z}_{\alpha\beta}}^a\right> + \left<s_{\alpha\beta}, {\left(\psi_\alpha\right)}_{\Smooth{z}_{\alpha\beta}}^b\right>\left<s_{\alpha\beta}, {\left(\psi_\alpha\right)}_{\Smooth{z}_{\alpha\beta}}^c\right>\varphi_\alpha^\#\Gamma_{bc}^a \\
		&= {\left(\varphi_\beta\right)}_{\Smooth{z}_{\beta\alpha}}^\#y^a + \left<s_{\beta\alpha}, {\left(\psi_\beta\right)}_{\Smooth{z}_{\beta\alpha}}^a\right> + \left<s_{\beta\alpha}, {\left(\psi_\beta\right)}_{\Smooth{z}_{\beta\alpha}}^b\right>\left<s_{\beta\alpha}, {\left(\psi_\beta\right)}_{\Smooth{z}_{\beta\alpha}}^c\right>\varphi_\beta^\#\Gamma_{bc}^a
	\end{split}
\end{equation}

We will now turn to the action of the reparametrization group \(G^T\), see Definition~\ref{defn:ReparametrizationGroupOfNodalCurves}, on super stable maps and their equivalence classes.
\begin{defn}\label{defn:EquivalenceOfSuperStableMaps}
An element \(\bm{g}=\Set{g_\alpha}_{\alpha\in T}\) of \(\underline{G^T}(C)\) acts on a super stable map \((\bm{z}, \bm{\Phi})=\left(\left(\Set{z_{\alpha\beta}}_{E_{\alpha\beta}}, \Set{z_i}_{1\leq i\leq k}\right), \Set{\Phi_\alpha}_{\alpha\in T}\right)\) over \(C\) and modeled on \(T\) by
	\begin{equation}
		\bm{g}\cdot\left(\bm{z}, \bm{\Phi}\right)
		= \left(\left(\Set{g_\alpha(z_{\alpha\beta})}, \Set{g_{p(i)}(z_i)}\right), \Set{\Phi_\alpha\circ g_\alpha^{-1}}\right).
	\end{equation}
\end{defn}
Like Definition~\ref{defn:SuperStableMap}, this definition is functorial in \(C\).
The stability condition in Definition~\ref{defn:SuperStableMap} implies that every super stable map has only finitely many automorphisms, that is \(\bm{g}\cdot\left(\bm{z}, \bm{\Phi}\right) = \left(\bm{z}, \bm{\Phi}\right)\).
More precisely, the only super stable maps that have non-trivial automorphisms are super stable maps over \(\R^{0|0}\) which have nodes with constant maps:
Reparametrizations acting by reflection of the odd directions on some nodes with constant maps and by identity on the others are non-trivial automorphisms.

We will say that two super stable maps of genus zero over \(C\) and modeled over the same tree type~\(T\) are equivalent if they differ by a reparametrization from \(\underline{G^T}(C)\).
We are interested in the spaces of equivalence classes of super stable maps fixing the homology class of the image.
A super stable map \((\bm{z}, \bm{\Phi})\) represents the homology class \(A\in H_2(N, \Integers)\) if its reduction represents the homology class \(A\) according to the classical definition.
That is, for \(\phi_\alpha=\Red{\left(\Phi_\alpha\right)}\colon\ProjectiveSpace[\C]{1}\to N\) we have
\begin{equation}
	A = \sum_{\alpha\in T} {\left(\phi_\alpha\right)}_*[\ProjectiveSpace[\C]{1}].
\end{equation}
In particular, we obtain a function
\begin{equation}
	\begin{split}
		T &\to H_2(N, \Integers) \\
		\alpha &\mapsto A_\alpha = {\left(\phi_\alpha\right)}_* [\ProjectiveSpace[\C]{1}]
	\end{split}
\end{equation}
such that
\begin{itemize}
	\item
		\(A=\sum_{\alpha\in T} A_\alpha\)
	\item
		if \(A_\alpha=0\) the number of special points on the node \(\alpha\) is at least three, that is \(\#Y_\alpha\geq 3\)
	\item
		every \(A_\alpha\) is spherical, that is, in the image of the Hurewicz map \(\pi_2(N)\to H_2(N, \Integers)\).
\end{itemize}
We use the shorthand \(\Set{A_\alpha}\) for such maps.
For fixed tree type \(T\) and \(\Set{A_\alpha}\) we denote the set of equivalence classes of super stable maps \((\bm{z}, \bm{\Phi})\) over \(C\), modeled over \(T\) such that \(\phi_\alpha[\ProjectiveSpace[\C]{1}]=A_\alpha\) by \(\underline{\mathcal{M}_{0,T}\left(\Set{A_\alpha}\right)}(C)\).
As both Definition~\ref{defn:SuperStableMap} and Definition~\ref{defn:EquivalenceOfSuperStableMaps} are functorial in~\(C\), we have a functor
\begin{equation}
	\underline{\mathcal{M}_{0, T}(\Set{A_\alpha})}\colon \cat{SPoint}^{op}\to \cat{Set},
\end{equation}
called the moduli space of super stable maps of fixed tree type \(T\) and homology classes~\(\Set{A_\alpha}\).
Taking the union over all \(\Set{A_\alpha}\) we define the moduli space \(\underline{\mathcal{M}_{0,T}(A)}\) of super stable maps of fixed tree type \(T\) and total homology class \(A\) by
\begin{equation}
	\underline{\mathcal{M}_{0, T}(A)}(C)
	=\bigcup_{\Set{\Set{A_\alpha}}} \underline{\mathcal{M}_{0,T}(\Set{A_\alpha})}(C).
\end{equation}
This is a finite union because the homology class \(A\) can only be decomposed in a finite number of possible \(\Set{A_\alpha}\) with corresponding \(\targetACI\)-holomorphic curves \(\phi_\alpha\colon \ProjectiveSpace[\C]{1}\to (N, \omega, \targetACI)\).
Every homologically non-trivial curve \(\phi_\alpha\) contributes a minimum amount \(\left<[\omega], A_\alpha\right>>\hbar(N, \omega, \targetACI)>0\) to the total action \(\left<[\omega], A\right>\), see~\cite[Proposition~4.1.5]{McDS-JHCST}.
The number of \(\alpha\in T\) such that \(A_\alpha=0\) is bounded by the number of stable nodes in the tree \(T\).
This also implies that there are no possible partitions~\(\Set{A_\alpha}\) if the tree~\(T\) has too many unstable nodes.

As a special case, let us look at the stable \(k\)-labeled tree that is given by a single vertex, no edges and \(k\geq 3\) marked points.
We denote the moduli space of super stable maps with this tree type and total homology class \(A\) by \(\mathcal{M}_{0,k}(A)\).
If the moduli space of super \(\targetACI\)-holomorphic curves \(\mathcal{M}_0(A)\) can be constructed, we obtain using Proposition~\ref{prop:ModuliSpaceOfSRS0k}
\begin{equation}
	\mathcal{M}_{0,k}(A) = \mathcal{M}_{0,k} \times \mathcal{M}_0(A).
\end{equation}
That is, the moduli space of stable maps with one bubble and \(k\)-markings can be realized as a superorbifold and its orbit functor \(\underline{\mathcal{M}_{0,k}(A)}\) takes values in topological spaces.
The space \(\underline{\mathcal{M}_{0,k}(A)}(\R^{0|0})\) is homeomorphic to the moduli space of classical stable maps.

In general, the spaces \(\mathcal{M}_{0,T}(\Set{A_\alpha})\) cannot be equipped with the structure of a superorbifold.
As in the theory of classical \(\targetACI\)-holomorphic curves the situation improves if we restrict to simple super stable maps.
Recall that a super \(\targetACI\)-holomorphic curve \(\Phi\colon \ProjectiveSpace[\C]{1|1}\to N\) is called simple if its reduction \(\phi=\Red{\Phi}\) is simple, that is, there is no \(\targetACI\)-holomorphic curve \(\tilde{\phi}\colon \ProjectiveSpace[\C]{1}\to N\) and a branched holomorphic covering \(h\colon \ProjectiveSpace[\C]{1}\to \ProjectiveSpace[\C]{1}\) such that \(\phi=\tilde{\phi}\circ h\).
\begin{defn}
	A super stable map \((\bm{z}, \bm{\Phi})\) is called simple if \((\Red{\bm{z}}, \Red{\bm{\Phi}})\) is simple, that is, every \(\Red{\left(\Phi_\alpha\right)}\) is a simple map and no two maps \(\Red{\left(\Phi_\alpha\right)}\) and \(\Red{\left(\Phi_\beta\right)}\) have the same image.

	We denote the subfunctor of simple super stable maps of genus zero and modeled over the tree \(T\) and fixed homology classes \(\Set{A_\alpha}\) by \(\underline{\mathcal{M}_{0,T}^*(\Set{A_\alpha})}\subset\underline{\mathcal{M}_{0,T}(\Set{A_\alpha})}\).
\end{defn}
We will see in Section~\ref{SSec:SuperStableMapsOfFixedTreeType}, that under certain conditions on \(N\) the functor \(\underline{\mathcal{M}^*_{0,T}(\Set{A_\alpha})}\) is the orbit functor of a superorbifold.

Furthermore, we are interested in the functor \(\underline{\overline{\mathcal{M}}_{0,k}}(A)\) which is given by the union over all isomorphism classes of \(k\)-labeled trees:
\begin{equation}
	\begin{split}
		\underline{\overline{\mathcal{M}}_{0,k}(A)}\colon \cat{SPoint}^{op}&\to \cat{Set} \\
		C &\mapsto \bigcup_{\text{tree types }T} \underline{\mathcal{M}_{0,T}(A)}(C)
	\end{split}
\end{equation}
This is a finite union, because there is only a finite number of trees for which there exist an admissible \(\Set{A_\alpha}\) and hence a non-empty moduli space.

In Section~\ref{SSec:GromovTopologyOnM0kABar}, we will refine this functor to take values in topological spaces such that
\begin{itemize}
	\item
		The space \(\underline{\overline{\mathcal{M}}_{0,k}(A)}(\R^{0|0})\) is homeomorphic to the space of classical stable maps.
		If the target \(N\) is compact, \(\underline{\overline{\mathcal{M}}_{0,k}(A)}(\R^{0|0})\) is the compactification of \(\underline{\mathcal{M}_{0,k}(A)}(\R^{0|0})\).
	\item
		When restricting to simple super stable maps of fixed tree type \(T\) and cohomology classes \(\Set{A_\alpha}\) the topology coincides with the one given by the orbit functor of the superorbifold \(\mathcal{M}^*_{0,T}(\Set{A_\alpha})\).
	\item
		Under certain conditions on the homology class \(A\) the space of non-simple stable maps \(\underline{\mathcal{M}_{0,T}(\Set{A_\alpha})}(\R^{0|0})\setminus\underline{\mathcal{M}_{0,T}^*(\Set{A_\alpha})}(\R^{0|0})\) has a high codimension.
		Consequently, also the subfunctor \(\underline{\mathcal{M}_{0,T}(\Set{A_\alpha})}\setminus\underline{\mathcal{M}_{0,T}^*(\Set{A_\alpha})}\) has a high even codimension if it can be shown to be a superorbifold.
\end{itemize}
With this justification, we regard \(\underline{\overline{\mathcal{M}}_{0,k}(A)}\) as the natural compactification of \(\underline{\mathcal{M}_{0,k}(A)}\).

\begin{rem}
	There is a functorial forgetful map \(\underline{\pi}\colon \underline{\overline{\mathcal{M}}_{0,k}(A)}\to \underline{\overline{\mathcal{M}}_{0,k}}\) which sends \(\left(\bm{z}, \bm{\Phi}\right)\) to the stabilization \(\bm{z}^s\) which arises after deletion of unstable vertices from the tree.
	This essentially combinatorial construction works as in the case of classical stable maps, see~\cite[Chapter~5.1]{McDS-JHCST}.
	Similarly, the definition of the projections \(\underline{\pi_k}\colon \underline{\overline{\mathcal{M}}_{0,k}(A)}\to \underline{\overline{\mathcal{M}}_{0,k-1}(A)}\) which deletes the \(k\)-th marked point and the evaluation maps \(\underline{\ev_i}\colon \underline{\overline{\mathcal{M}}_{0,k}(A)}\to \underline{N}\) which send \(\left(\bm{z}, \bm{\Phi}\right)\) to \(\Phi_{p(i)}\circ z_i\in \underline{N}(C)\) is straightforward.
\end{rem}

\subsection{Simple super stable maps of fixed tree type}\label{SSec:SuperStableMapsOfFixedTreeType}
In~\cite[Chapter~6]{McDS-JHCST} it is shown that \(\underline{\mathcal{M}_{0,T}^*(\Set{A_\alpha})}(\R^{0|0})\) is a manifold for generic almost complex structure \(\targetACI\) on \(N\).
Here, we will extend the argument to obtain that \(\underline{\mathcal{M}^*_{0,T}(\Set{A_\alpha})}\) is the orbit functor of a superorbifold under certain conditions on the target.

We denote the number of vertices of the \(k\)-labeled tree \(T\) by \(\#T\), the number of edges by \(\#E=\#T-1\).
By definition, a simple stable \(\targetACI\)-holomorphic curve over \(C\) is a tuple
\begin{equation}
	\left(\left(\Set{z_{\alpha\beta}}_{E_{\alpha\beta}}, \Set{z_i}_{1\leq i\leq k}\right), \Set{\Phi_\alpha}_{\alpha\in T}\right)
	\in {\left(\underline{\ProjectiveSpace[\C]{1|1}}(C)\right)}^{2\#E+k}\times \prod_{\alpha\in T}\underline{\mathcal{M}_0^*(A_\alpha)}(C)
\end{equation}
subject to the conditions~\ref{item:ZT}--\ref{item:evT} below.
For the moment, we assume that all the moduli spaces on the right hand side are supermanifolds.
The moduli space will later be obtained as a quotient of a subset \(Z^T\times M^T\) of an appropriate sub-supermanifold by the group acting as equivalence transformations.
\begin{enumerate}
	\item\label{item:ZT}
		For a fixed node \(\alpha\) the reduction of the special points needs to be distinct.
		We denote by \(Z^T\) the open subsupermanifold of \({\left(\ProjectiveSpace[\C]{1|1}\right)}^{2\#E+k}\) such that for each fixed node \(\alpha\) the reduction of the special points contained in \(Y_\alpha\) are pairwise distinct.
		Then \(Z^T\) is an open supermanifold of real dimension
		\begin{equation}
			4\#E + 2k| 4\#E + 2k.
		\end{equation}
	\item\label{item:MT}
		The reduced maps \(\Red{\left(\Phi_\alpha\right)}\) have to be pairwise different.
		We call the open sub supermanifold of \(\prod_{\alpha\in T}\mathcal{M}_0^*(A_\alpha)\) such that \(\Red{\left(\Phi_\alpha\right)}\) are pairwise different \(M^T\).
		Then \(M^T\) is an open supermanifold of real dimension
		\begin{equation}
			\begin{split}
				\MoveEqLeft
				\sum_{\alpha\in T} \left(2n + 2\left<A_\alpha, c_1(TN)\right>\right)| 2\sum_{\alpha\in T}\left<A_\alpha, c_1(TN)\right> \\
				&= 2n(\#E+1) + 2\left<A, c_1(TN)\right> | 2\left<A, c_1(TN)\right>.
			\end{split}
		\end{equation}
	\item\label{item:evT}
		At the nodes, super stable maps have to satisfy \(\Phi_\alpha\circ z_{\alpha\beta} = \Phi_\beta\circ z_{\beta\alpha}\).
		In order to understand the geometry of this condition we define the evaluation map
		\begin{equation}
			\begin{split}
				\underline{\ev^T}(C)\colon \underline{Z^T}(C)\times\underline{M^T}(C) &\to \underline{N^{2\#E}}(C) \\
				(\bm{z}, \bm{\Phi}) &\mapsto {\left(\Phi_\alpha\circ z_{\alpha\beta}\right)}_{E_{\alpha\beta}}
			\end{split}
		\end{equation}
		Denote by \(\Delta^T\subset N^{2\#E}\) the subset determined by the edges:
		\begin{equation}
			\Delta^T
			=\Set{\left(n_{\alpha\beta}\right)\in N^{2\#E}\given n_{\alpha\beta}=n_{\beta\alpha}}
		\end{equation}
		Note that \(\Delta^T\) is of codimension \(2n\#E|0\) in \(N^{2\#E}\).
		The preimage \({\left(\ev^T\right)}^{-1}\Delta^T\) is the set of simple super stable maps of genus zero and modeled over the tree \(T\).
		If \(\ev^T\) is transversal to \(\Delta^T\) the preimage \({\left(\ev^T\right)}^{-1}\Delta^T\) is a supermanifold of dimension
		\begin{equation}
			\begin{split}
				\MoveEqLeft
				\dim M^T + \dim Z^T - \codim \Delta^T \\
				&= 2n + 2\left<A, c_1(TN)\right> + 4\#E + 2k | 2\left<A, c_1(TN)\right> + 4\#E + 2k.
			\end{split}
		\end{equation}
		Notice that the map \(\ev^T\) is transversal to \(\Delta^T\) if its reduction \(\Red{\ev^T}\) is transversal to \(\Delta^T\) because \(N^{2\#E}\) is an even manifold.
\end{enumerate}
The group of equivalence transformations acts properly on \({\left(\ev^T\right)}^{-1}\Delta^T\).
The only fix-points of this action are the \(\R^{0|0}\)-points of \({\left(\ev^T\right)}^{-1}\Delta^T\) and the isotropy group is generated by the automorphisms \(\Xi_-^\alpha\) which multiply the odd direction by \(-1\) on the bubble \(\alpha\in T\).
The group of equivalence transformations \(G^T\) is a super Lie group of dimension \(6(\#E+1)|4(\#E+1)\).
By Theorem~\ref{thm:QuotientByFiniteIsotropyProperGroupAction} the quotient is a superorbifold of dimension
\begin{equation}
	\begin{split}
		\MoveEqLeft
		\dim M^T + \dim Z^T - \codim \Delta^T - \dim G^T\\
		&= 2n + 2\left<A, c_1(TN)\right> - 2\#E + 2k - 6 | 2\left<A, c_1(TN)\right>  + 2k - 4.
	\end{split}
\end{equation}
This proves the following theorem:
\ModuliSpaceOfSimpleStableMapsFixedTreeType{}

To understand the conditions of Theorem~\ref{thm:ModuliSpaceOfSimpleStableMapsFixedTreeType} better, recall that the moduli space \(\mathcal{M}_0^*(A_\alpha)\) has the structure of a supermanifold if the maps \(D_\phi\) and \(\Dirac^{0,1}\) given in Equation~\eqref{eq:DefnDphi} and Equation~\eqref{eq:DefnDirac01} are surjective for all \(\phi=\Phi\in\underline{\mathcal{M}^*_{0}(A_\alpha)}(\R^{0|0})\).
Surjectivity depends on the choice of the tame almost complex structure \(\targetACI\) on the symplectic manifold \((N, \omega)\).
By~\cite[Theorem~3.1.5]{McDS-JHCST}, the set of \(\targetACI\) such that \(D_\phi\) is surjective for all \(\Phi\in\underline{\mathcal{M}^*_0(A)}(\R^{0|0})\) is dense.
In~\cite{KSY-SJC}, we have shown surjectivity of \(\Dirac^{0,1}\) with the help of a Bochner method under certain curvature conditions such as Kähler manifolds with positive holomorphic sectional curvature.
This argument stays valid under small variations of \(\targetACI\).
For the transversality of \(\ev^T\) it suffices to check transversality in the reduced case which has been shown in~\cite[Theorem~6.2.6]{McDS-JHCST} for generic \(\targetACI\).
Consequently, starting with an almost Kähler manifold~\((N, \omega, \targetACI)\) which satisfies the curvature conditions such that \(\Dirac^{0,1}\) is surjective, one can perturb the almost complex structure \(\targetACI\) a finite number of times such that the resulting almost Kähler manifold satisfies the conditions of the Theorem~\ref{thm:ModuliSpaceOfSimpleStableMapsFixedTreeType}.
Moreover, as there is a finite number of subfunctors \(\underline{\mathcal{M}^*_{0,T}(\Set{A_\alpha})}\) in \(\underline{\overline{\mathcal{M}}_{0,k}(A)}\), we can perturb \(\targetACI\) further such that all the subfunctors \(\underline{\mathcal{M}^*_{0,T}(\Set{A_\alpha})}\) are superorbifolds.

\subsection{Gromov topology on \texorpdfstring{\(\underline{\overline{\mathcal{M}}_{0,k}(A)}\)}{M0kABar}}\label{SSec:GromovTopologyOnM0kABar}
In this section we refine the functor \(\underline{\overline{\mathcal{M}}_{0,k}(A)}\) to take values in topological spaces.
The strategy is to adapt the definition of Gromov convergence as discussed in~\cite[Chapter~5.6]{McDS-JHCST} to \(\underline{\overline{\mathcal{M}}_{0,k}(A)}(C)\) in a way that is functorial in \(C\).
While this functor is not expected to represent a superorbifold, its restriction to fixed tree type coincides with \(\underline{\mathcal{M}^*_{0,T}(A)}\) and its reduced points are compact.

In order to define the Gromov topology on \(\underline{\overline{\mathcal{M}}_{0,k}(A)}\), we need some more notation.
Let \(\left(\bm{z}, \bm{\Phi}\right) = \left(\left(\Set{z_{\alpha\beta}}, \Set{z_i}\right), \Set{\Phi_\alpha}\right)\) be a super stable map over \(B\) of genus zero and modeled on \(T\) and \(\mathbf{g}\in\underline{G^T}(B)\).
Recall that every map \(C_B\colon C\to B\) between superpoints gives rise to a super stable map \(C_B^*(\bm{z}, \bm{\Phi})\) over \(C\) and an automorphisms \(C_B^*\mathbf{g}\in\underline{G^T}(C)\).
This holds in particular for the initial object \(\R^{0|0}_B\colon \R^{0|0}\to B\) of \(\cat{SPoint}_B\) in which case \({\left(\R^{0|0}_B\right)}^*\left(\bm{z}, \bm{\Phi}\right)\) is the reduction, that is a classical stable map of genus zero modeled over the tree~\(T\).

We denote the set of \(C\)-points that have the same reduction as a node on the vertex \(\alpha\) by
\begin{equation}
	\underline{Z_\alpha}(C)
= \Set{p\in \underline{\ProjectiveSpace[\C]{1|1}}(C)\given {\left(\R^{0|0}_B\right)}^*p= {\left(\R^{0|0}_B\right)}^*z_{\alpha\beta}\text{ for some edge }\beta\text{ connected with }\alpha}
\end{equation}
Hence \(\ProjectiveSpace[\C]{1|1}\setminus Z_\alpha\) is an open subsupermanifold of \(\ProjectiveSpace[\C]{1|1}\).
This will allow us to formulate the convergence of super \(\targetACI\)-holomorphic curves to \(\Phi_\alpha\) away from nodes.

Another important requirement is that the harmonic action is preserved under Gromov convergence, both locally and globally.
In order to express this condition we define for any edge \(E_{\alpha\beta}\) of \(T\) the tree \(T_{\alpha\beta}\) to be the component of \(T\) that contains \(\beta\) after removal of the edge \(E_{\alpha\beta}\).
The superconformal action on \(T_{\alpha\beta}\) given by
\begin{equation}
	A_{\alpha\beta}(\bm{z}, \bm{\Phi})
	= \sum_{\gamma\in T_{\alpha\beta}} A(\Phi_\gamma)
	= \sum_{\gamma\in T_{\alpha\beta}} \int_{\ProjectiveSpace[\C]{1|1}/B} \norm{\left.\differential{\Phi_\alpha}\right|_\cD}^2 [\d{vol}]
\end{equation}
taking values in \(\underline{\R}(B)\).
For a node \(\alpha\in T\) and an open subset \(U\subset\underline{\ProjectiveSpace[\C]{1|1}}(\R^{0|0})\) we define the action contained in that open subset as
\begin{equation}
	A_\alpha((\bm{z}, \bm{\Phi}); U)
	= \int_{U} \norm{\left.\differential{\Phi_\alpha}\right|_\cD}^2 [\d{vol}] + \sum_{\beta\in T_U} A_{\alpha\beta}(\bm{z}, \bm{\Phi})
\end{equation}
where the sum runs over the set of nodes contained in \(U\),
\begin{equation}
	T_U = \Set{\beta\in T\given {\left(\R_B^{0|0}\right)}^*z_{\alpha\beta}\in U}.
\end{equation}
We will only consider the case of \(U=B_\epsilon(z_{\alpha\beta})\) which is the open ball of radius \(\epsilon\) around \({\left(\R^{0|0}_B\right)}^*z_{\alpha\beta}\) in the round metric of \(\ProjectiveSpace[\C]{1}=\underline{\ProjectiveSpace[\C]{1|1}}(\R^{0|0})\).

Informally, Gromov convergence of super stable maps of genus zero can now be formulated as convergence of the maps \(\Phi_\alpha\) away from nodal points up to automorphisms of \(\ProjectiveSpace[\C]{1|1}\), convergence of the superconformal action on all subtrees \(T_{\alpha\beta}\) and convergence of the nodal curve following Definition~\ref{defn:GromovConvergenceStableCurves}.
More precisely:
\begin{defn}\label{defn:GromovConvergenceStableMaps}
	A sequence \(\left(\bm{z}^\nu, \bm{\Phi}^\nu\right) = \left(\left(\Set{z_{\alpha\beta}^\nu}, \Set{z_i^\nu}\right), \Set{\Phi_\alpha^\nu}\right)\), \(\nu=1, 2, \dotsc\) of super stable maps over \(B\) of genus zero modeled on \(T^\nu\) is said to Gromov converge to to a super stable map \(\left(\bm{z}, \bm{\Phi}\right) = \left(\left(\Set{z_{\alpha\beta}}, \Set{z_i}\right), \Set{\Phi_\alpha}\right)\) over \(B\) of genus zero and modeled on \(T\) if for \(\nu\) sufficiently large there exists a tree homomorphism \(f^\nu\colon T\to T^\nu\) and a collection of reparametrizations \(\bm{g}^\nu = {\Set{g^\nu_\alpha}}_{\alpha\in T}\in \underline{G^T}(B)\) sucht that the following hold:
	\begin{description}
		\item[(Map)]
			For every \(\alpha\in T\) and every \(C_B\colon C\to B\) the sequence
			\begin{equation}
				\underline{\Phi^\nu_{f^\nu(\alpha)}\circ g_\alpha^\nu}(C_B)\colon \underline{\ProjectiveSpace[\C]{1|1}}(C)\to \underline{N}(C)
			\end{equation}
			converges to \(\underline{\Phi_\alpha}(C)\) uniformly on compact subsets on \(\underline{\ProjectiveSpace[\C]{1|1}}(C)\setminus\underline{Z_\alpha}(C)\).
		\item[(Action)]
			For any edge \(E_{\alpha\beta}\) of \(T\)
			\begin{equation}
				A_{\alpha\beta}\left(\bm{z},\bm{\Phi}\right)
				= \lim_{\epsilon\to 0}\lim_{\nu\to\infty} A_{f^\nu(\alpha)}\left(\left(\bm{z}^\nu, \bm{\Phi}^\nu\right); B_\epsilon(z_{\alpha\beta})\right)
				\in \underline{\R}(B).
			\end{equation}
		\item[(Rescaling)]
			If \(\alpha\), \(\beta\in T\) are connected by an edge \(E_{\alpha\beta}\) and \(\nu_j\) is a subsequence such that \(f^{\nu_j}(\alpha) = f^{\nu_j}(\beta)\) then the sequence
			\begin{equation}
				\underline{{\left(g_\alpha^{\nu_j}\right)}^{-1}\circ g_\beta^{\nu_j}}(C_B)\colon \underline{\ProjectiveSpace[\C]{1|1}}(C)\to \underline{\ProjectiveSpace[\C]{1|1}}(C)
			\end{equation}
			converges uniformly on compact subsets to \(C_B^*z_{\alpha\beta}\in\underline{\ProjectiveSpace[\C]{1|1}}(C)\) for all \(C_B\colon C\to B\) and \(j\to\infty\).
		\item[(Nodal points)]
			If \(\alpha\), \(\beta\in T\) are connected by an edge \(E_{\alpha\beta}\) and \(\nu_j\) is a subsequence such that \(f^{\nu_j}(\alpha)\neq f^{\nu_j}(\beta)\) then
			\begin{equation}
				z_{\alpha\beta} = \lim_{j\to\infty}{\left(g_\alpha^{\nu_j}\right)}^{-1} (z^{\nu_j}_{f^{\nu_j}(\alpha)f^{\nu_j}(\beta)})
				\in\underline{\ProjectiveSpace[\C]{1|1}}(B).
			\end{equation}
		\item[(Marked Points)]
			For all~\(1\leq i\leq k\) it holds \(p^\nu(i)=f^\nu(p(i))\) and
			\begin{equation}
				z_i = \lim_{\nu\to\infty}{(g_{p(i)}^\nu)}^{-1}(z_i^\nu)
				\in \underline{\ProjectiveSpace[\C]{1|1}}(B).
			\end{equation}
	\end{description}
\end{defn}
Gromov convergence is defined up to automorphisms and hence defines a topology on \(\underline{\overline{\mathcal{M}}_{0,k}(A)}(B)\) as in the case of Gromov convergence of stable curves, see~\ref{defn:GromovConvergenceStableCurves}.
By construction, this topology is functorial in \(B\) and has the following desired properties:
\begin{itemize}
	\item
		The topological space \(\underline{\overline{\mathcal{M}}_{0,k}(A)}(\R^{0|0})\) is homeomorphic to the classical moduli space of stable maps of genus zero.
		In particular, it is Hausdorff and compact if the target~\(N\) is compact.
	\item
		The restriction of the Gromov topology to simple super stable maps of genus zero and fixed tree type \(T\) yields the same topology as the orbit functor of the superorbifold \(\mathcal{M}_{0,T}^*(A)\).
\end{itemize}
We will leave the questions of uniqueness of limits and construction of suitable integrals over \(\underline{\overline{\mathcal{M}}_{0,k}(A)}\) to define invariants for further work.
\begin{rem}
	Let \(f^\nu\colon M\to M'\) be a sequence of maps between supermanifolds over \(B\) and \(f\colon M\to M'\) a map of supermanifolds over \(B\).
	The condition that
	\begin{equation}
		\underline{f^\nu}(C_B)\colon \underline{M}(C_B)\to \underline{M'}(C_B)
	\end{equation}
	converges uniformly on compact subsets to \(\underline{f}(C_B)\) for all \(C_B\colon C\to B\) to \(\underline{f}(C_B)\) is equivalent to the uniform convergence on compact subsets of all coefficients of all coordinate expressions of the maps \(f^\nu\) to the corresponding coefficients of \(f\).
\end{rem}
\begin{rem}
	The Definition~\ref{defn:GromovConvergenceStableMaps} can easily be generalized so that every \(\left(\bm{z}^\nu, \bm{\Phi}^\nu\right)\) is a super stable map of genus zero in the almost Kähler manifold \((N, \omega, \targetACI^\nu)\) where \(\targetACI^\nu\) is a sequence of \(\omega\)-tame almost complex structures converging to an \(\omega\)-tame almost complex structure \(\targetACI\).
	This allows for additional flexibility as in~\cite[Chapter~5]{McDS-JHCST}.
\end{rem}

\printbibliography

\textsc{Enno Keßler\\
Center of Mathematical Sciences and Applications,
Harvard University,
20~Garden Street,
Cambridge, MA 02138,
USA}\\
\texttt{ek@cmsa.fas.harvard.edu} \\
\newpage
\textsc{Artan Sheshmani\\
Center of Mathematical Sciences and Applications,
Harvard University,
20~Garden Street,
Cambridge, MA 02138,
USA\\[.5em]
Institut for Matematik,
Aarhus Universitet,
Ny Munkegade 118,
8000, Aarhus C,
Denmark\\[.5em]
National Research University
Higher School of Economics, Russian Federation,
Laboratory of Mirror Symmetry,
NRU HSE,
6 Usacheva Street,
Moscow, Russia, 119048}\\
\texttt{artan@cmsa.fas.harvard.edu} \\

\textsc{Shing-Tung Yau\\
Center of Mathematical Sciences and Applications,
Harvard University,
20~Garden Street,
Cambridge, MA 02138,
USA\\[.5em]
Department of Mathematics,
Harvard University,
Cambridge, MA 02138,
USA
}\\
\texttt{yau@math.harvard.edu}
\end{document}